\documentclass[11pt, a4paper]{article}
\usepackage{amsmath,amssymb,amsfonts}
\usepackage{amsthm}
\usepackage{verbatim}
\usepackage{a4wide}
\usepackage{epsfig,epic,eepic,graphicx}
\usepackage[usenames,dvipsnames]{color}
\usepackage[colorlinks=true, pdfstartview=FitV, linkcolor=blue, citecolor=blue, urlcolor=blue]{hyperref}
\usepackage{enumerate}

\newtheorem{theorem}{Theorem}[section]
\newtheorem{lemma}{Lemma}[section]
\newtheorem{proposition}{Proposition}[section]
\newtheorem{assume}{Assumption}

\newtheorem{remark}{Remark}[section]
\newtheorem{corollary}{Corollary}[section]

\numberwithin{equation}{section}
\makeatletter

\@addtoreset{equation}{section}
\makeatother

\newcommand{\T}{{\mathbb T}}
\newcommand{\N}{{\mathbb N}}
\newcommand{\Z}{{\mathbb Z}}

\newcommand{\R}{{\mathbb R}}

\newcommand{\pa}{{\partial}}
\newcommand{\na}{{\nabla}}

\newcommand{\be}{\begin{equation}}
\newcommand{\ee}{\end{equation}}

\newcommand{\ba}{\begin{aligned}}
\newcommand{\ea}{\end{aligned}}
\newcommand{\sgn}{\mathrm{sgn}}

\newcommand{\cA}{\,\|\hspace{-.68em}|\,\,\,}

\def\Reel{{\cal R}e \,}

\def\mspace{\medskip \noindent}

\begin{document}

\title{Optimal Prandtl expansion around concave boundary layer}
\author{ \footnote{Universit\'e Paris Diderot, Sorbonne Paris Cit\'e, Institut de Math\'ematiques de Jussieu-Paris Rive Gauche, UMR 7586, F- 75205 Paris, France} David Gerard-Varet \and \footnote{Department of Mathematics, Graduate School of Science, Kyoto University
Kitashirakawa Oiwake-cho, Sakyo-ku, Kyoto 606-8502, Japan } Yasunori Maekawa \and
\footnote{Courant Institute of Mathematical Sciences, 251 Mercer
Street, New York, NY 10012, USA. } Nader Masmoudi}

\date{}

\maketitle
\begin{abstract}
We provide an optimal Gevrey stability result for general boundary layer expansions, under a mild concavity condition on the boundary layer profile. Our result generalizes (and even improves in the non strictly concave case) the one obtained in \cite{GMM},  restricted to expansions of shear flow type. 
\end{abstract}

\section{Introduction}\label{sec.intro}

We are interested in the high Reynolds number dynamics of the Navier-Stokes equation in a half-plane: 
\begin{equation}\label{ns}
\begin{aligned}
\pa_t u^\nu -\nu \Delta u^\nu + \nabla p^\nu + u^\nu\cdot \nabla u^\nu & = 0,  \qquad & t>0,~x\in \T,~y>0,\\
\nabla \cdot u^\nu & = 0,  \qquad  &t\geq 0,~x\in \T, ~y>0,\\
u^\nu|_{y=0} =0, \qquad u^\nu|_{t=0} &= u_0 &
\end{aligned}
\end{equation}
where $\nu$ stands for the inverse Reynolds number. Note that we consider periodic boundary conditions in $x$, but  could consider decay conditions as well. 
As is well-known, the Navier-Stokes solution $u^\nu$ exhibits a boundary layer near $y=0$, that is a region of high velocity gradients generated by the no-slip condition. A famous modelling of this boundary layer was provided by Prandtl. In modern language, he provided approximate solutions of Navier-Stokes in the form of  multiscale asymptotic expansions: 
\begin{equation} \label{Prandtl_profile}
 v  = \sum_{i=0}^N \sqrt{\nu}^i U^{E,i}(t,x,y) + \sum_{i=0}^N \sqrt{\nu}^i  \Big( V^{bl, i}_1\big(t,x,y/\sqrt{\nu}\big), \sqrt{\nu} V^{bl}_2\big(t,x,y/\sqrt{\nu}\big) \Big) 
\end{equation}
where the profiles $U^{E,i} = U^{E,i}(t,x,y)$ describe the flow away from the boundary, and  the profiles  $V^{bl, i} = V^{bl,i}(t,x,Y)$  are boundary layer correctors, that go to zero exponentially fast in variable$Y=y/\nu^\frac12$. We stress that there is a factor $\sqrt{\nu}$ between the amplitudes of the horizontal and vertical components of the boundary layer profiles: this is consistent with the divergence-free condition.  In particular, the leading order term $U^E := U^{E,0}$ solves the Euler equation, while the leading order boundary corrector $V^{bl} := V^{bl,0}$ solves the modified Prandtl equation 
\begin{align*}
& \pa_t V^{bl}_1 + (U^E_1\vert_{y=0} + V^{bl}_1 )  \pa_x  V^{bl,1} +  V^{bl}_1  \pa_x  U^E_1\vert_{y=0} +  (Y \pa_y U^{E}_2\vert_{y=0}  +  V^{bl}_2)  \pa_Y   V^{bl}_1 - \pa^2_Y V^{bl}_1 = 0 
\\
& \pa_x V^{bl,1} + \pa_Y V^{bl}_2 = 0 \\
& V^{bl}_1\vert_{Y=0} = - U^{E}_1\vert_{y=0}, \quad V^{bl} \rightarrow 0, \quad Y \rightarrow +\infty
 \end{align*}
Prandtl boundary layer theory has revealed illuminating about the mechanism of vorticity generation in fluids, and successful in the quantitative understanding of some model problems, notably the description of the Blasius flow near a flat plate. Still, Navier-Stokes flows of type \eqref{Prandtl_profile} are known to experience instabilities, due to two main mechanisms: 
\begin{itemize}
\item Boundary layer separation, which corresponds to a loss of monotonicity and concavity of the boundary layer profile $V^{BL}_1$, under an adverse pressure gradient. Mathematically, it corresponds to some ill-posedness or blow-up of the Prandtl model. 
\item Hydrodynamic instabilities of Tollmien-Schlichting type, experienced by concave boundary layer flows. 
\end{itemize}
These phenomena have crucial consequences in hydrodynamics and aerodynamics.  From the mathematical point of view, describing the stability/instability properties of flows $v$  of type \eqref{Prandtl_profile} is a difficult topic. The evolution of the perturbation $w=u^\nu-v$ obeys the perturbed Navier-Stokes system
\begin{align}\label{pns}
\begin{split}
&\pa_t w -\nu \Delta w + \nabla q + v\cdot \nabla w + w\cdot \nabla v =-w\cdot \nabla w + r , \qquad t>0,~x\in \T,~y>0,\\
& \nabla \cdot w =0,\qquad t\geq 0,~x\in \T, ~y>0,\\
& w|_{y=0} =0, \qquad w|_{t=0}=w_0.
\end{split}
\end{align}
Here, $r$  represents a remainder term due to the approximation $v$, while $w_0$ is a given initial perturbation of the velocity. We will assume that $r$ and $w_0$ are of the order $O(\nu^n)$ in some norm with $n \gg 1$. In the case of $r$, this is realized  by taking $N$ large enough in \eqref{Prandtl_profile}.
More precisely, one has to consider  functional frameworks such that the equations of both Prandtl type and Euler type are uniquely solvable at least locally in time. Then, the point is to understand under which conditions one can obtain uniform  (in $\nu$) estimates of $w$ in a suitable norm, that is justification of the Prandtl theory. 

An  important result in this direction is due to Caflisch and Sammartino \cite{SaCa1,SaCa2}, who proved local well-posedness of Euler and Prandtl equations, as well as stability results for \eqref{pns} in the case of analytic data. The stability result is then extended by \cite{Mae,MR3614755,FeTaZha,WaWa,KuViWa2020}, where all of them requires the analyticity near the boundary.  This general analytic stability result is somehow optimal, in view of a work of Grenier \cite{Gre}, see also \cite{MR4038143}. Grenier studied the case where the  Prandtl expansion $v$ in \eqref{Prandtl_profile} is a shear flow: this means that
\begin{equation} \label{shear_flow_profile}
 v = \Big(V^{bl}_1\big(t,x,y/\sqrt{\nu}\big), 0 \Big) 
 \end{equation}
where $V_1^{bl}$ solves the heat equation 
\begin{equation} \label{heat} 
\pa_t V^{bl}_1 - \pa_Y^2 V^{bl}_1 = 0, \quad V^{bl}_1\vert_{Y=0} = 0.
\end{equation}  
He proved that for some profiles $V^{bl}_1$ that have initially inflexion points, the linearized version of \eqref{pns} admits growing perturbations of the form 
$$ w^\nu(t,x,y) \approx e^{ \alpha t/\nu^{\frac12}} e^{i x/\nu^{\frac12}} \tilde w^{\nu}(y),  $$
with fixed $\alpha > 0$. This shows  that high frequencies   $k \approx 1/\nu^{\frac12}$ in variable $x$ may be amplified by $e^{\alpha k t}$. In other words, to obtain  a  bound  independent of $\nu$ over a time $T = O(1)$  will only be possible if those modes $k$ have amplitude less than  $e^{-\delta k}$, $\delta \le  \alpha T$. This necessary exponential decay of the frequency spectrum corresponds to analytic perturbations.  Let us note that the result of Grenier relies on the so-called Rayleigh instability, which is an inviscid instability mechanism for shear flows  with inflexion points.  In terms of hydrodynamics of the boundary layer, the appearance of inflexion points corresponds to the separation phenomenon. Hence, it is a framework in which various negative results exist for the Prandtl equation itself \cite{MR1476316,GeDo,GeNg,MR3590519}. 

The case without inflexion points, corresponding to the nicer situation where the boundary layer profile $V^{bl}_1$  is concave in variable $Y$, is much more involved. Again, the natural first step is to consider the shear flow situation \eqref{shear_flow_profile}. The stability of shear flows within the Navier-Stokes equation is an old topic of hydrodynamics, notably studied by Tollmien and Schlichting. See \cite{DrRe} for a detailed account. They showed that generic concave shear flows, although stable in the Euler evolution, exhibit instability in the Navier-Stokes one (albeit with a growth rate vanishing with viscosity). This is the so-called Tollmien-Schlichting instability, revisited on a rigorous basis by Grenier, Guo and Nguyen \cite{GGN2014}. Roughly,  by using a proper rescaling of these unstable eigenmodes, one can construct for the linearization of \eqref{pns} solutions of the type 
$$ w^\nu(t,x,y) \approx e^{ \alpha t/\nu^{\frac14}} e^{i x/\nu^{\frac38}} \tilde w^{\nu}(y).  $$
This time,  high frequencies   $k \approx 1/\nu^{\frac38}$ may be amplified by $e^{\alpha k^{\frac23} t}$. This is still  not compatible with Sobolev uniform bounds. More precisely, under the assumption that the spectral radius of the linearized Navier-Stokes operator is given by the growth rate of the Tollmien-Schlichting instability, one can obtain exponential bounds on the semigroup and from there show nonlinear Sobolev instability of Prandtl expansions of shear flow type: {\it cf} \cite{GNg2,GreToanJMPA}.

Nevertheless, in the setting of concave boundary layer flows, the class of data $w_0$ for which one can hope  uniform (in $\nu$) local (in time) control of $w$  is larger than analytic: namely, one may expect control for data whose Fourier spectrum  in $x$ decays like $O(e^{-k^{2/3}})$. This corresponds to the so-called Gevrey class of exponent $3/2$. 

{\em To show such optimal stability result for general "concave" Prandtl expansions is  the main goal of the present paper}. It extends the result established in \cite{GMM}, limited to the case when the boundary layer is the shear type like \eqref{shear_flow_profile}. See also the recent development \cite{ChenWuZhang}, still on shear flow expansions. Precise statements  will be given in next Section \ref{sec_statements}. Three preliminary remarks are in order: 
\begin{itemize}
\item  The approach in \cite{GMM} was very much based on Fourier transform in $x$, made easy because \eqref{shear_flow_profile} is independent of $x$. It does not adapt to general Prandtl expansions. The approach in the present paper relies on strongly different ideas. 
\item The main step in our approach is the derivation of stability estimates for the linearized equations:
\begin{equation}\label{ps}
\begin{aligned}
\pa_t w -\nu \Delta w + \nabla q + v\cdot \nabla w + w\cdot \nabla v & = f, \qquad & t>0,~x\in \T,~y>0,\\
\nabla \cdot w & = 0,\qquad &  t\geq 0,~x\in \T, ~y>0,\\
w|_{y=0} = 0, \qquad w|_{t=0} &= w_0.
\end{aligned}
\end{equation}
 But to derive such bounds, we do not make any assumption on the spectral radius of the linearized operator, in contrast with works \cite{GNg2,GreToanJMPA}. 
\item A strong point of our analysis is that it applies to boundary layer profiles $V^{bl}_1$ that are  concave in $Y$, but not necessarily strictly concave. See Section \ref{sec_statements} for detailed hypotheses. This is important for applications, as can be seen from \eqref{heat}: there,  $\pa^2_Y V^{bl}_1$ vanishes at the boundary for $Y=0$ at positive times. Despite such possible degeneracies, we are able to reach Gevrey $\frac32$ stability: this was not the case in our previous paper \cite{GMM}, where our Gevrey exponent for stability was less than $\frac32$ for non strictly concave flows. 
\end{itemize}

Let us insist that our result is the first one justifying boundary layer theory beyond the analytic scale. 

\section{Statements of the results} \label{sec_statements}

To state our stability result, we first introduce our functional framework. 
Let $p\in [1,\infty]$, $K\geq 1$, and $\nu\in (0,1]$. For simplicity we assume $\nu^{-\frac12}\in \N$, but it is not at all  essential to our argument. We set
\begin{align}\label{def.norm}
\| f\|_{p} =  \sum_{j=0}^{\nu^{-\frac12}} \frac{1}{(j!)^\frac32} \sup_{j_2=0,\cdots,j} \|e^{-Kt(j+1)}\beta_{j_2} \pa_x^{j-j_2} f \|_{L^p_t (0,\frac{1}{K}; L^2_{x,y})},
\end{align}
where 
\begin{align}\label{def.beta}
\beta_{j_2}=\chi^{j_2} \pa_y^{j_2}, \qquad \chi(y) = 1-e^{-\kappa y}.
\end{align}
Here $\kappa\in (0,1]$ is a fixed number, which will be taken small enough. 
We note that $\|f\|_p$ depends on $\nu, \kappa\in (0,1]$ and $K\geq 1$, though we drop this dependence to simplify the notation. 
Note that for each fixed $\nu$ the norm $\|f\|_p$ is of Sobolev type, but if $\|f\|_p$ is uniformly bounded  in $\nu$, it implies a usual Gevrey $\frac32$ regularity for the $C^\infty$ function $f$. The reason we can restrict to $j \le \nu^{-\frac12}$ in the sum above is that  in \eqref{pns}, the stretching term $\nabla v =  O(\nu^{-\frac12})$ creates at most an amplification $O(e^{C\nu^{-\frac12}t})$. For $j \sim  \nu^{-\frac12}$, it is therefore balanced by the factor $e^{-Kt(j+1)}$ for large enough $K$. This means that we will be able to close an estimate considering only derivatives up to order $\nu^{-\frac12}$. 

Our main theorem is the following. Let us set $H_{0,\sigma}^1(\T\times \R_+)=\{f\in H^1_0(\T\times \R_+)^2~|~{\rm div}\, f=0~~{\rm in}~\T\times \R_+\}$, the space of all $H^1$ solenoidal vector fields satisfying the noslip boundary condition at $Y=0$.
\begin{theorem} \label{thm.main.nl} {\bf (Nonlinear stability of concave Prandtl expansions)}
Let  $v = v(t,x,y)$ a divergence-free vector field that fulfills the regularity and concavity conditions gathered in Assumption \ref{assume} below, not necessarily of type \eqref{Prandtl_profile}. There exists 
$\kappa_0 > 0$, such that the following statement holds for any $\kappa \in (0, \kappa_0]$: there exist $C>0$, $K > 0$, $\delta_0 > 0$ such that:  for all $\nu \le K^{-2}$, if $r\in L^2(0,\frac1K; L^2(\T\times \R_+)^2)$ and $w_0\in H^1_{0,\sigma}(\T\times \R_+)$ satisfy
\begin{align}
[|w_0|] + [|{\rm rot}\, w_0|] \le \delta_0 \nu^\frac94,\qquad \| r\|_2\le \delta_0 \nu^\frac{11}{4},
\end{align}
then the system \eqref{pns} has a unique solution $w \in C([0,\frac{1}{K}], H^1_{0,\sigma} (\T \times \R_+))$, satisfying 
\begin{align}
\| w \|_\infty  + \nu^\frac12  \| {\rm rot} \,  w\|_\infty \leq C\nu^{-\frac12} \Big ( [| w_0|] +[| {\rm rot} \, w_0|] + \nu^{-\frac12} \|r\|_2\Big ).
\end{align} 
Here  ${\rm rot}\, w=\pa_x w_2-\pa_y w_1$ and $\displaystyle [|w_0|]=\sum_{j=0}^{\nu^{-\frac12}} \frac{1}{(j!)^\frac32} \sup_{j_2=0,\cdots,j} \| \beta_{j_2}\pa_x^{j-j_2} w_0\|_{L^2_{x,y}}$.
\end{theorem}

To complete the statement of our theorem, it remains to describe the set of assumptions on $v$  that yield Theorem \ref{thm.main.nl}. Of course, these assumptions are designed to be satisfied by Prandtl expansions of type \eqref{Prandtl_profile}, when $V^{bl}_1$ has some mild concavity.  Due  to the boundary layer variable $Y$, it is more convenient to work with  rescaled variables  $(\tau, X,Y) := \nu^{-\frac12} (t,x,y)$. 
Accordingly, we shall express our assumptions directly on 
$$ V(\tau,X,Y) := v(t,x,y), \quad \tau>0,~X\in \T_\nu,~Y>0. $$
Here, $\T_\nu := \nu^{-\frac12} \T$. We set 
\begin{align*}
\Omega = \pa_XV_2-\pa_Y V_1,
\end{align*}
which describes the vorticity field of the approximation in the rescaled variables. We also set
\begin{align}\label{def.chi}
\chi_\nu =\chi (\nu^\frac12 Y)= 1-e^{-\kappa \nu^\frac12 Y}. 
\end{align}
Note that $\kappa\in (0,1]$ is fixed, but taken small enough. Also, in the rescaled variables our almost Gevrey norm becomes
\begin{align}\label{def.Norm}
\cA F \cA_p = \sum_{j=0}^{\nu^{-\frac12}} \frac{1} {(j!)^\frac32 \nu^\frac{j}{2}} \sup_{j_2=0,\cdots, j} \| e^{-K\tau \nu^\frac12 (j+1)}B_{j_2} \pa_X^{j-j_2} F \|_{L^p_\tau (0,\frac{1}{K\nu^\frac12}; L^2_{X,Y})}, \qquad B_{j_2} = \chi_\nu^{j_2} \pa_Y^{j_2}.
\end{align}
We state our key assumptions in terms of $V$ and  $\Omega$.
\begin{assume} \label{assume} 
\end{assume}
{\em 
\noindent {\rm {\bf (i)  Divergence-free and Dirichlet condition on $V$:}} 
\begin{align}\label{div}
\pa_X V_1+\pa_Y V_2=0,\qquad V|_{Y=0}=0
\end{align}
Moreover, there exist constants $C_*\geq 1$ and $C_0^*, C^*_1, C^*_2>0$ such that the following statements hold for any $\nu\in (0,1]$ and $K\geq 1$.

\medskip
\noindent {\rm {\bf (ii) Almost Gevrey $L^\infty$ bounds for $V$ and $\nabla \Omega$:}} For any $\kappa\in (0,1]$ we have
\begin{align}\label{est.assume.1}
\begin{split}
& \sum_{j=0}^{\nu^{-\frac12}} \frac{1}{(j!)^\frac32 \nu^\frac{j}{2}}  \sup_{j_2=0,\cdots,j} \Big (  \| e^{-K\tau \nu^\frac12 j} B_{j_2} \pa_X^{j-j_2}  V_1 \|_{L^\infty_{\tau,X,Y}} + \kappa \| e^{-K\tau \nu^\frac12 j}\frac{\pa_X^{j} V_2}{\chi_\nu}  \|_{L^\infty_{\tau,X,Y}} \\
& \quad + \nu^{-\frac12}(j+1)^\frac12  \| e^{-K\tau \nu^\frac12 j} B_{j_2} \pa_X^{j-j_2}  \pa_X V_1 \|_{L^\infty_{\tau,X,Y}} + (j+1)^\frac12 \| e^{-K\tau \nu^\frac12 j} B_{j_2} \pa_X^{j-j_2}  \pa_Y V_1 \|_{L^\infty_{\tau,X,Y}} \\
&\quad + \nu^{-\frac12} \| \frac{1+Y}{1+\nu^\frac12 Y} e^{-K\tau \nu^\frac12 j} B_{j_2} \pa_X^{j-j_2} \pa_{X} \Omega \|_{L^\infty_{\tau,X,Y}} +  \| (\frac{1+Y}{1+\nu^\frac12 Y})^2 e^{-K\tau \nu^\frac12 j}B_{j_2} \pa_X^{j-j_2} \pa_{Y} \Omega \|_{L^\infty_{\tau,X,Y}}\Big ) \\
& \leq C_0^*.
\end{split}
\end{align}
Here $L^\infty_{\tau,X,Y}=L_\tau^\infty (0,\frac{1}{K\nu^\frac12}; L^\infty_{X,Y})$.

\medskip
\noindent {\rm {\bf (iii) Derivative bounds for $V$ and $\Omega$:}} We have
\begin{align}\label{est.assume.2}
\begin{split}
& \| V\|_{L^\infty_{\tau,X,Y}} + \nu^{-\frac12} \| \pa_X V \|_{L^\infty_{\tau,X,Y}} + \| \frac{1+Y}{1+\nu^\frac12 Y} \pa_Y V_1\|_{L^\infty_{\tau,X,Y}}   \\
& + \nu^{-\frac12} \| \frac{1+Y}{1+\nu^\frac12 Y} \pa_X \Omega \|_{L^\infty_{\tau,X,Y}}   + \| (\frac{1+Y}{1+\nu^\frac12 Y})^2 \pa_Y \Omega\|_{L^\infty_{\tau,X,Y}} \\
& + \nu^{-\frac12} \| \big (\frac{Y}{1+\nu^\frac12 Y} \big )^2 \pa_\tau \pa_Y \Omega\|_{L^\infty_{\tau,X,Y}} + \nu^{-\frac12} \| \frac{Y(1+Y)}{(1+\nu^\frac12 Y)^2}  \pa^2_{XY} \Omega\|_{L^\infty_{\tau,X,Y}} +\| (\frac{Y(1+Y)^2}{(1+\nu^\frac12 Y)^3} \pa_Y^2 \Omega\|_{L^\infty_{\tau,X,Y}} \\
& \leq C^*_1.
\end{split}
\end{align}

\medskip
\noindent {\rm {\bf (iv) Monotonicity of $\Omega$:}} Set $\rho (Y) =C_* \big ( (1+\frac{Y}{\nu^{\frac14}})^{-2} + \nu^\frac12 (1+Y)^{-2} + \nu\big ) $. Then we have
\begin{align}\label{est.assume.3}
\pa_Y \Omega + \rho  \geq 0,
\end{align}
and 
\begin{align}\label{est.assume.4}
\nu^{-\frac12} \| \frac{Y}{1+\nu^\frac12 Y} \frac{\pa^2_{XY}\Omega}{\sqrt{\pa_Y \Omega + 2\rho}} \|_{L^\infty_{\tau,X,Y}} + \| \frac{Y(1+Y)}{(1+\nu^\frac12 Y)^2} \frac{\pa_Y^2\Omega}{\sqrt{\pa_Y \Omega + 2\rho}} \|_{L^\infty_{\tau,X,Y}} \leq C^*_2.
\end{align}
}

\begin{remark}[Link between the Prandtl expansions and the assumptions]
\end{remark}
\noindent
Let us explain how the set of assumptions above relates to Prandtl expansions as given in \eqref{Prandtl_profile}. 

\medskip
\noindent
{\bf i)} The divergence-free and Dirichlet conditions are satisfied by Prandtl expansions of type \eqref{Prandtl_profile}. Fields $u^{E,i}$ solve Euler or linearized Euler equations, while fields $V^{bl,i}$ solve Prandtl or linearized Prandtl equations: in both cases, they are divergence-free. Moreover, they are constructed alternatively in order to satisfy the Dirichlet boundary condition : once $U^{E,i}$ is constructed, $V^{Bl,i}$ is constructed so that 
$$ U^{E,i}_1\vert_{y = 0} +  V^{bl,i}_1\vert_{Y = 0} = 0. $$ 
Then, $U^{E,i+1}$ is constructed by solving an Euler type equation with  the non-penetration condition 
$$ U^{E,i+1}_2\vert_{y=0} +  V^{bl,i}_2\vert_{Y=0} = 0. $$
More precisely, one can construct  $(U^{E,i}, V^{bl,i})$ in this way for $i \le N-1$, and conclude by 
$$U^{E,N}(t,x,y) := \Big(0,  -V^{bl,N-1}_2(t,x,0)\Big), \quad V^{BL,N} := 0. $$

\medskip
\noindent
{\bf ii)} Assumption ii) amounts essentially to a Gevrey $\frac{3}{2}$  bound on solutions $U^{E,i}$, resp. $V^{bl,i}$,  of Euler like and Prandtl like equations. Such solutions exist locally in time. For the Euler equations, we refer to \cite{MR2765482} and references therein. For the Prandl equations, as mentioned before, the works \cite{SaCa1,KuVi} provide local in time solutions for analytic data. These local solutions being analytic, they belong to the Gevrey class $\frac{3}{2}$. More recently, Gevrey  local in time well-posedness of the Prandtl equation has been established in \cite{MR3925144} (see \cite{GeMa,MR4055987} for preliminary partial results).  Also, if $v$ is given by \eqref{Prandtl_profile},  as $V_2(\tau,X,Y) = v_2(t,x,y)$ is zero at the boundary $Y=0$, we can write 
$$ V_2 = \int_0^Y \pa_Y V_2 \approx  \int_0^Y \Big(\nu^{1/2} \big( \pa_y V^{E,0}_2  +  \pa_Y  V^{bl,0}_2 \big) + \dots \Big)  = O(\nu^{\frac12} Y) = O(\frac{1}{\kappa} \chi_\nu(Y)) \quad \text{at}~~Y=0 $$
 so that $\frac{1}{\kappa}\frac{V_2}{\chi_\nu}$ is under control as required in ii).  

\medskip
\noindent
{\bf iii)} Again, assumption iii) is satisfied by classical Prandtl expansions of type \eqref{Prandtl_profile}. To check that, one has to keep in mind that $\pa_\tau \sim \nu^{\frac12} \pa_t$, $\pa_X \sim \nu^{\frac12} \pa_x$, so that for Prandtl expansions, which depend smoothly on $t$ and $x$, any $\tau$- or $X$-derivative allows to gain  $\nu^{1/2}$. This explains for instance the factor $\nu^{-\frac12}$ in front of the second and fourth terms of \eqref{est.assume.2}, related to $\pa_X V$ and $\pa_X \Omega$.  In the same spirit, as  $\pa_Y \sim \nu^{\frac12} \pa_y$, for the Euler part of the Prandtl expansion (which depends smoothly on $y$), any $Y$-derivative  allows to gain    $\nu^{1/2}$. This remark does not apply to the boundary layer part of the expansion, as it depends genuinely on $Y$. Still, this part has good decay in $Y$ (typically like $e^{-Y}$ or $(1+Y)^{-N}$ for large $Y$). This is coherent with  the weights $(1+Y)/(1+\nu^{\frac12} Y)$ or $Y/(1+\nu^{\frac12} Y)$ that can be found in \eqref{est.assume.2} in front of terms with $Y$ derivatives: outside the boundary layer ($Y \gg 1$), it yields  a gain of $\nu^{\frac12}$, but in the boundary layer ($Y \sim 1$), it yields some decay information on the boundary layer terms. 

\medskip
\noindent
{\bf iv)}  In the case $v$ is given by Prandtl expansions of type \eqref{Prandtl_profile},
$$\pa_Y \Omega = \pa^2_{XY} V_2 - \pa_Y^2 V_1= - \pa_Y^2 V^{bl}_1 +   O(\nu) + O(\sqrt{\nu} (1+Y)^{-2})   $$ 
Here, the $O(\nu)$ comes from the Euler part of the Prandtl expansion. The $O(\sqrt{\nu} (1+Y)^{-2})$ corresponds to the boundary layer profiles $V^{bl,i}$, $i \ge 1$. The last two terms in the definition of the weight $\rho$ allow to control them for $C_*$ large enough. Hence, condition \eqref{est.assume.3} is essentially a (non strict) concavity condition on the leading term of the Prandtl boundary layer,  $V^{bl} := V^{bl,0}$. Moreover,  by the addition  of the sublayer term 
$(1+(Y/\nu^{\frac14}))^{-2}$ in the definition of $\rho$, we allow any sign for $\pa_Y^2 V_{0,1}^P$ in the sublayer  $0\le Y\le O(\nu^\frac14)$, and the concavity is only needed for $Y\geq O(\nu^\frac14)$. In the original variables this sublayer is of the order $O(\nu^\frac34)$, which  is  typical order of  Kolmogorov dissipation length in the theory of turbulence. 

As regards \eqref{est.assume.4}, we notice that for Prandtl expansions:  
\begin{align*}
\pa_{XY}^2 \Omega & = -\pa_X \pa^2_Y V^{bl}_1 + O (\nu^{\frac32}) + O(\nu (1+Y)^{-2}),\\
\pa_Y^2 \Omega & = -\pa_Y^3 V^{bl}_1  + + O (\nu^{\frac32})  + O (\nu^\frac12 (1+Y)^{-2}) 
\end{align*}
Hence, by taking into account the bound $\frac{1}{\sqrt{\pa_Y \Omega+2\rho}} \leq  \frac{1}{C_* \nu^\frac12}$,  the condition \eqref{est.assume.4} is essentially verified if $V^{bl}_1$ satisfies
\begin{align*}
\nu^{-\frac12} \| \frac{Y \pa_X\pa_Y^2 V^{bl}_1}{\sqrt{-\pa_Y^2 V^{bl}_1 + 2 C_* (1+\frac{Y}{\nu^\frac14})^{-2}}} \|_{L^\infty_{\tau,X,Y}} +  \| \frac{Y(1+Y) \pa_Y^3 V^{bl}_1}{\sqrt{-\pa_Y^2 V^{bl}_1 + 2 C_* (1+\frac{Y}{\nu^\frac14})^{-2}}} \|_{L^\infty_{\tau,X,Y}} \leq C<\infty.
\end{align*}

In the next section, we will explain the general strategy for the proof of our main stability theorem. More precisely, we will briefly describe our stability analysis of the linearized equation \eqref{ps}, for $f$ a given force. This is the core of our paper:  the transition from linear to nonlinear stability is more standard. As explained before, we shall work with the rescaled variables $(\tau, X, Y)$. We set 
$$ W(\tau, X, Y) := w(t,x,y), \quad F(\tau, X, Y) := \sqrt{\nu} f(t,x,y), \quad W_0(X,Y) := w_0(x,y)  $$
(and still $V(\tau,X,Y) = v(t,x,y)$). System \eqref{ps} becomes
\begin{equation}\label{PS}
\begin{aligned}
\pa_\tau W -\nu^\frac12  \Delta W + \nabla Q + V\cdot \nabla W + W\cdot \nabla V & = F, \qquad  & \tau>0,~X\in \T_\nu,~Y>0,\\
\nabla \cdot W & = 0,\qquad & \tau \geq 0,~X\in \T_\nu, ~Y>0,\\
 W|_{Y=0} =0, \qquad W|_{\tau=0} & = W_0, 
\end{aligned}
\end{equation}

The main result on this linear system is 
\begin{theorem}\label{thm.main1} Suppose that Assumption \ref{assume} holds. Then there exists $\kappa_0\in (0,1]$ such that the following statement holds for any $\kappa \in (0,\kappa_0]$. There exists $K_0=K_0(\kappa, C_*,C^*_j)\geq 1$ such that if $K\geq K_0$ then the system \eqref{PS} admits a unique solution $W\in C([0,\infty); H^1_{0,\sigma} (\T_\nu \times \R_+))$ satisfying 
\begin{align}\label{est.thm.main1.1}
\begin{split}
& \cA W\cA_\infty  + \cA {\rm rot}\,  W\cA_\infty  \leq C \Big ( (\nu^{-\frac12} + K^\frac12 \nu^{-\frac14}) [\| W_0\|]+\nu^{-1}  [\|{\rm rot}\, W_0\|]  + \nu^{-\frac54} \cA F\cA_2 \Big ).
\end{split}
\end{align} 
Here ${\rm rot}\, W=\pa_XW_2-\pa_Y W_1$ and $\displaystyle [\|W_0\|]=\sum_{j=0}^{\nu^{-\frac12}} \frac{1}{(j!)^\frac32\nu^\frac{j}{2}} \sup_{j_2=0,\cdots,j} \| B_{j_2}\pa_X^{j-j_2} W_0\|_{L^2_{X,Y}}$,
and $C$ is a universal constant.
\end{theorem}

As a consequence, we have the following result in the original variables.
Note that, from $F(\tau,X,Y) = \nu^\frac12 f(t,x,y)$, we have $\nu^{-\frac54} \cA F\cA_2=\nu^{-\frac32} \| f\|_2$. 
\begin{theorem}\label{main.thm2} Suppose that Assumption \ref{assume} holds. Then there exists $\kappa_0\in (0,1]$ such that the following statement holds for any $\kappa \in (0,\kappa_0]$. There exists $K_0=K_0(\kappa, C_*,C^*_j)\geq 1$ such that if $K\geq K_0$ then the system \eqref{ps} admits a unique solution $w\in C([0,\infty); H^1_{0,\sigma} (\T\times \R_+))$ satisfying 
\begin{align}\label{est.thm.main2.1}
\| w \|_\infty  + \nu^\frac12  \|{\rm rot}\,  w\|_\infty \leq C\nu^{-\frac12} \Big ( (1+K^\frac12 \nu^\frac14) [| w_0|] +  [| {\rm rot}\, w_0|] + \nu^{-\frac12} \| f \|_2\Big ).
\end{align} 
Here  ${\rm rot}\, w=\pa_x w_2-\pa_y w_1$ and $\displaystyle [|w_0|]=\sum_{j=0}^{\nu^{-\frac12}} \frac{1}{(j!)^\frac32} \sup_{j_2=0,\cdots,j} \| \beta_{j_2}\pa_x^{j-j_2} w_0\|_{L^2_{x,y}}$, and $C$ is a universal constant.
\end{theorem}

\section{General strategy} \label{sec:strategy}
Estimates on system \eqref{PS} will be performed at the level of the vorticity,  through the Orr-Sommerfeld formulation: 
\begin{align}\label{eq.OS}
\begin{split}
& (\pa_\tau + V \cdot \nabla - \nu^{\frac12} \Delta) \omega + \nabla^\bot \phi \cdot \nabla \Omega =  {\rm rot}\, F,\qquad \tau>0,~X\in \T_\nu, ~Y>0,\\
& \phi|_{Y=0} =\pa_Y \phi|_{Y=0}=0.
\end{split} 
\end{align}
Here,  $\omega = {\rm rot} \, W  := \pa_X W_2 - \pa_Y W_1$ is the vorticity, and  $\phi$ is the stream function, satisfying  $W =  \nabla^\bot \phi := \left( \begin{smallmatrix} \pa_Y \phi \\ -\pa_X \phi \end{smallmatrix}\right)$ and   $-\Delta \phi = \omega$. 
We recall that $\tau = \nu^{-\frac12} t$: the point is to get estimates that are valid over time intervals of size $\nu^{-\frac{1}{2}}$, which is difficult due to the stretching term  $\nabla^\bot \phi \cdot \nabla \Omega$. Classical estimates and Gronwall lemma would  only yield a control  on time intervals $O(1)$. We have to use both our Gevrey functional framework and concavity condition.

Actually, several difficulties are already captured by the toy model 
\begin{align}\label{eq.toy}
\begin{split}
& (\pa_\tau - \nu^{\frac12} \Delta) \omega + \pa_X \phi \,  \pa_Y \Omega =  0,\qquad \tau>0,~X\in \T_\nu, ~Y>0,\\
& \phi|_{Y=0} =\pa_Y \phi|_{Y=0}=0.
\end{split} 
\end{align}
where $\Omega = \Omega(Y)$ (for simplicity, we assume no dependence on $\tau$ and $X$). We shall stick to this model for what follows. 

In the case of the inviscid equation 
$$ \pa_\tau \omega + \pa_X \phi \,  \pa_Y \Omega  = 0, \quad \phi\vert_{Y=0} = 0$$
under the strict sign condition $\pa_Y \Omega \ge C > 0$, a trick that goes back to \cite{MR1761422} is to test the equation against $\frac{\omega}{\pa_Y \Omega}$. By the cancellation 
\begin{align*}
 \int  \pa_X \phi \,  \pa_Y \Omega  \frac{\omega}{\pa_Y \Omega} = -  \int  \pa_X \phi \, \Delta \phi = \frac{1}{2} \int \pa_X |\na \phi|^2 = 0 
 \end{align*}
 one can obtain a uniform in time control on the weighted quantity $||\frac{\omega}{\sqrt{\pa_Y \Omega}}||_{L^2} \sim \|\omega\|_{L^2}$. However,  back to the model \eqref{eq.toy}, we are facing two difficulties: 
 \begin{enumerate}
 \item Inpired by the case of Prandtl layers, we must consider situations where $\pa_Y \Omega$  vanishes or even becomes slightly negative : see iv) in Assumption \ref{assume}.
 \item Even in the simpler case  $\pa_Y \Omega \ge C > 0$, the weighted estimate above is not compatible with the introduction of viscosity and no-slip conditions. 
 \end{enumerate}
We recall that these difficulties are not purely technical, as no uniform in $\nu$ stability estimate is expected below Gevrey $\frac{3}{2}$ regularity. 
To overcome these issues, we shall proceed in two steps. 
\subsection{First step : Gevrey estimates for artificial boundary conditions} 
The first step consists in deriving Gevrey bounds for the same equation, but with pure slip instead of no-slip conditions. For the real Orr-Sommerfeld equation, it will be performed in Section \ref{sec:vorticity:pure:slip}. For our toy model, this means that we consider 
\begin{align}\label{eq.toy.2}
\begin{split}
& (\pa_\tau - \nu^{\frac12} \Delta) \omega + \pa_X \phi \,  \pa_Y \Omega =  0,\qquad \tau>0,~X\in \T_\nu, ~Y>0,\\
& \phi|_{Y=0} = \omega|_{Y=0}=0.
\end{split} 
\end{align}
The main point in this change of boundary conditions is that the difficulty 2. mentioned above disappears: the Dirichlet condition on $\omega$ goes well with integration by parts, and  in the case $\pa_Y \Omega \ge C > 0$, one can achieve again some good control on  $||\frac{\omega}{\sqrt{\pa_Y \Omega}}||_{L^2}$. 
Still, we have to explain how to obtain stability under the less stringent condition  iv)  in Assumption \ref{assume}. Here, we need Gevrey regularity.  Let us for simplicity forget about $Y$-derivatives, which are not important for the toy model, and set 
$$ \omega^j := e^{-K\tau \nu^\frac12 (j+1)} \pa_X^j \omega, \quad \phi^j := e^{-K\tau \nu^\frac12 (j+1)} \pa_X^{j} \phi, \dots $$
The point is to obtain a bound on $\sum_{j \le \nu^{-1/2}}   \frac{1}{(j!)^{3/2} \nu^{j/2}} ||\omega^j||_{L^2_{X,Y}}$. As $\Omega = \Omega(Y)$, equation satisfied by $\omega^j$ is: 
\begin{equation}\label{eq.toy.2.j}
 (K \nu^{1/2} (j+1) + \pa_\tau - \nu^{\frac12} \Delta) \omega^j + \pa_X \phi^j \,  \pa_Y \Omega = 0. 
 \end{equation}
Roughly,  the idea is to control a weighted Gevrey norm of the form 
$$\sum_{j \le \nu^{-1/2}}  \frac{1}{(j!)^{3/2} \nu^{j/2}} \big\|\frac{\omega^j}{\sqrt{\pa_Y \Omega +  2 \rho_j}}\big\|_{L^2_{X,Y}},$$ 
where $\rho_j$ is added to compensate for possible degeneracies of $\pa_Y \Omega$. Testing \eqref{eq.toy.2.j} against $\frac{\omega_j}{\pa_Y \Omega +  2 \rho_j}$, we find
\begin{equation} \label{calculations.toy}
\begin{aligned}
& K \nu^{1/2} (j+1)  \big\|\frac{\omega^j }{\sqrt{\pa_Y \Omega +  2 \rho_j}}\big\|_{L^2}^2 + \frac{1}{2} \frac{d}{d\tau}  \big\|\frac{\omega^j}{\sqrt{\pa_Y \Omega +  2 \rho_j}}\big\|_{L^2}^2  + \nu^{\frac12} \big\|\frac{\na \omega^j}{\sqrt{\pa_Y \Omega +  2 \rho_j}}\big\|_{L^2}^2 \\
 &= - \nu^{\frac12} \int \na \frac{1}{\pa_Y \Omega +  2 \rho_j}  \cdot \na  \omega^j   \omega^j    - \int  \pa_X \phi^j \, \pa_Y \Omega \frac{\omega^j}{\pa_Y \Omega +  2 \rho_j} \\
 & =  \nu^{\frac12} \int \frac{\na \pa_Y  \Omega}{(\pa_Y \Omega +  2 \rho_j)^2}   \cdot \na  \omega^j   \omega^j    +  \nu^{\frac12} \int  \frac{\na \rho_j}{(\pa_Y \Omega +  2 \rho_j)^2}  \cdot \na  \omega^j   \omega^j   \\
 & + \int \pa_X \phi^j  \frac{2\rho_j}{\pa_Y \Omega +  2 \rho_j} \omega^j
\end{aligned}
\end{equation}
where we used again the cancellation property $\int \pa_X \phi^j \omega^j = 0$. One must then choose $\rho_j$ so that the three terms at the right are controlled by the left-hand side, for $K$ large enough. Roughly, this can be achieved by taking  $\rho_j$  in the form $\rho_j(Y) \approx  \rho +  (1+ \lambda_j Y)^{-2}$, $\lambda_j := (j+1)^{1/2}$. To give an idea of why it works, let us consider the first and last terms. As regards the first one, we write 
\begin{align*}
& \nu^{\frac12} \int \frac{\na \pa_Y  \Omega}{(\pa_Y \Omega +  2 \rho_j)^2}   \cdot \na  \omega^j   \omega^j \\
= \: & \nu^{\frac12}  \int_{\{Y \ge\frac{1}{\lambda_j}\}}   \frac{1}{Y \sqrt{\pa_Y \Omega +  2\rho_j}}    \frac{Y \na \pa_Y  \Omega}{\sqrt{\pa_Y \Omega +  2 \rho_j}}  \cdot    \frac{\na  \omega^j}{\sqrt{\pa_Y \Omega +  2 \rho_j}}  \frac{\omega^j}{\sqrt{\pa_Y \Omega +  2 \rho_j}}  \\
+ \: & \nu^{1/2} O\Big( \big\|   \frac{\na  \omega^j}{\sqrt{\pa_Y \Omega +  2 \rho_j}} \big\|_{L^2} \,  \big\|   \frac{\omega^j}{\sqrt{\pa_Y \Omega +  2 \rho_j}} \big\|_{L^2}  \Big) 
\end{align*}
The second term at the right corresponds to the contribution of the region $Y \le \frac{1}{\lambda_j}$, for which the weight $\pa_Y \Omega +  2 \rho_j$ is bounded from below and raises no issue (we further assumed here that $\pa_Y \na \Omega$ for the sake of brevity).  As regards the first term, we use the bounds 
$$ \forall Y \ge \frac{1}{\lambda_j}, \quad  \frac{1}{Y \sqrt{\pa_Y \Omega +  2\rho_j}}  \le  \frac{1}{Y \sqrt{2\rho_j}} \le  C \lambda_j,   $$ 
and 
$$   \frac{|Y \na \pa_Y  \Omega|}{\sqrt{\pa_Y \Omega +  2 \rho_j}} \le  \frac{|Y \na \pa_Y  \Omega|}{\sqrt{\pa_Y \Omega +  2 \rho}} \le C  $$
where we used  Assumption \ref{assume} iv). We end up with 
 \begin{align*}
 \nu^{\frac12} \int \frac{\na \pa_Y  \Omega}{(\pa_Y \Omega +  2 \rho_j)^2}   \cdot \na  \omega^j   \omega^j 
 \le C \nu^{\frac12} \lambda_j \big\| \frac{\na  \omega^j}{\sqrt{\pa_Y \Omega +  2 \rho_j}} \big\|_{L^2} \,  \big\|   \frac{\omega^j}{\sqrt{\pa_Y \Omega +  2 \rho_j}} \big\|_{L^2}  
 \end{align*}
 which is absorbed by the left-hand side under the constraint $ \lambda_j \lesssim (j+1)^{\frac12}$. As regards the third term at the right of \eqref{calculations.toy}, we use the inequality 
 \begin{align*}
 \frac{\rho_j}{\sqrt{\pa_Y \Omega + 2 \rho_j}} \: \le \:   \frac{\sqrt{\rho_j} }{\sqrt{2}} \: \le \: C \left(   \sqrt{\nu} + \frac{1}{\lambda_j Y} \right)
 \end{align*}
to obtain 
\begin{align*}
\int \pa_X \phi^j  \frac{2\rho_j}{\pa_Y \Omega +  2 \rho_j} \omega^j  & \le C   \sqrt{\nu}   \big\| \pa_X \phi^j \big\|_{L^2}  \big\|   \frac{\omega^j}{\sqrt{\pa_Y \Omega +  2 \rho_j}} \big\|_{L^2}  + \frac{C}{\lambda_j}    \big\| \frac{\pa_X \phi_j}{Y} \big\|_{L^2}  \big\|   \frac{\omega^j}{\sqrt{\pa_Y \Omega +  2 \rho_j}} \big\|_{L^2} \\
& \le C \left(  \sqrt{\nu}   \big\| \pa_X \phi^j \big\|_{L^2}  +   \frac{1}{\lambda_j}   \big\| \pa_Y \pa_X  \phi^j \big\|_{L^2} \right)   \big\|   \frac{\omega^j}{\sqrt{\pa_Y \Omega +  2 \rho_j}} \big\|_{L^2}  
\end{align*} 
where the second line comes from Hardy's inequality. Using that $\|\pa_X \phi^j\|_{L^2} \approx \|\phi^{j+1}\|_{L^2}$, we have  for any sequence $(a_j)$
$$ \sum_j   \frac{1}{(j!)^{3/2} \nu^{j/2}}  a_j   \| \pa_Y \pa_X \phi_{j}\|_{L^2} \approx  \sum_j    \frac{1}{(j!)^{3/2} \nu^{j/2}}   \nu^{1/2} (j+1)^{3/2} a_{j-1}  \| \pa_Y \phi_{j} \|_{L^2}  $$
In other words, at Gevrey $\frac{3}{2}$ regularity, $\: a_j \|\pa_Y \pa_X \phi^j \|_{L^2}$ behaves like $\: \nu^{1/2} (j+1)^{3/2} a_{j-1} \|\pa_Y \phi^j \|_{L^2}$. Combining this with a control of  $\|\na \phi_j\|_{L^2}$ by $\displaystyle \|\omega_j/\sqrt{\pa_Y \Omega + 2 \rho_j}\|_{L^2}$, with a precise statement to be given in Section  \ref{sec:vorticity:pure:slip}, the previous bound is in the same spirit as  
$$ "\int \pa_X \phi^j  \frac{2\rho_j}{\pa_Y \Omega +  2 \rho_j} \omega^j  \le  C \frac{\nu^{1/2}(j+1)^{3/2}}{\lambda_{j-1}}   \big\|   \frac{\omega^j}{\sqrt{\pa_Y \Omega +  2 \rho_j}} \big\|_{L^2}^2" $$  
which allows a control by the left-hand side of \eqref{calculations.toy} as soon as $ (j+1)^{1/2} \lesssim \lambda_j $. Hence, the choice $\lambda_{j} = (j+1)^{1/2}$. 

Of course, the elements above provide only glimpses of the approach carried in the first step of our stability study. The full study of the Orr-Sommerfeld equation with artificial boundary conditions is given in Section \ref{sec:vorticity:pure:slip}. 

\subsection{Recovery of the right boundary conditions}
We give again a few elements on the toy model \eqref{eq.toy}. The analysis of the complete model is carried in Section \ref{sec.boundary}. After the first step, one has a solution of system \eqref{eq.toy.2}, with the same initial condition and same boundary condition  $\phi\vert_{Y=0} = 0$ as in \eqref{eq.toy}, but not the same boundary condition on the derivative: $h := \pa_Y \phi\vert_{Y=0}   \neq 0$. Note that by the first step and trace theorem, one is able to get a Gevrey bound for $h$: as shown rigorously in the next sections, one may get an estimate of the form 
\begin{align*}
 \cA h \cA_{bc} & :=  \sum_{j \le \nu^{-1/2}}  \frac{1}{(j!)^{3/2} \nu^{j/2}} \|h^j\|_{L^2( (0, \frac{1}{K\nu^{1/2}}) ; L^2_X )} \\
  & \le \frac{C}{K^{1/4}} \left(  \|\na \phi_0\|_{L^2} +  C \sum_{j \le \nu^{-1/2}}  \frac{1}{(j!)^{3/2} \nu^{j/2}} ||\omega_0^j\|_{L^2_{X,Y}} \right) 
\end{align*}
where $\phi_0$ and $\omega_0 := - \Delta \phi_0$ are the initial data for the stream function and vorticity

 Working in Gevrey, regularity, the point is then to solve: 
\begin{align}\label{eq.toy.3}
\begin{split}
& (\pa_\tau - \nu^{\frac12} \Delta) \omega + \pa_X \phi \,  \pa_Y \Omega =  0,\qquad \tau>0,~X\in \T_\nu, ~Y>0,\\
& \phi|_{Y=0} = 0 \quad \pa_Y \phi|_{Y=0} = h, \quad \phi\vert_{t=0} = 0. 
\end{split} 
\end{align}
The main idea is to use the following scheme: 
\begin{itemize} 
\item[a)] We solve the approximate Stokes equation
\begin{align}\label{eq.stokes}
\begin{split}
& (\pa_\tau - \nu^{\frac12} \Delta) \Delta \phi =  0,\qquad \tau>0,~X\in \T_\nu, ~Y>0,\\
& \phi|_{Y=0} = 0 \quad \pa_Y \phi|_{Y=0} = h, \quad \phi\vert_{t=0} = 0. 
\end{split} 
\end{align}
and obtain in this way a solution $\phi_a = \phi_a[h]$. 
\item[b)] We correct the stretching term created by the previous approximation, by considering the full equation with artificial boundary condition: 
\begin{align}\label{eq.toy.4}
\begin{split}
& (\pa_\tau - \nu^{\frac12} \Delta) \omega + \pa_X \phi \,  \pa_Y \Omega =  - \pa_X \phi_a \,  \pa_Y \Omega  ,\qquad \tau>0,~X\in \T_\nu, ~Y>0,\\
& \phi|_{Y=0} = 0 \quad \Delta \phi|_{Y=0} = 0, \quad \phi\vert_{t=0} = 0. 
\end{split} 
\end{align}
We denote by $\phi_b = \phi_b[h]$ the solution of such system. It can be seen as a functional of $h$ through $\phi_a$. 
\item[c)] At the end of the steps a) and b), the function $\phi - \phi_a - \phi_b$ solves formally  the same system as $\phi$, replacing $h$ by $R_{bc} h := - \pa_Y \phi_b[h]\vert_{Y=0}$. The point is to show that for $K$ large enough,
\begin{equation} \label{estim_h}
\cA R_{bc} h \cA_{bc} \le \frac{1}{2} \cA  h \cA_{bc}. 
 \end{equation}
which allows to solve \eqref{eq.toy.3} by iteration. 
\end{itemize}
Obviously, to establish \eqref{estim_h}, one must have careful Gevrey stability estimates for systems \eqref{eq.stokes} and \eqref{eq.toy.4}.  The estimates  for  \eqref{eq.toy.4}  follows from the same  ideas as those described in the first step to treat \eqref{eq.toy.2} (the initial condition is just replaced by a source term).  
As regards \eqref{eq.stokes},  the initial data being zero, one can take Laplace transform in $\tau$ and Fourier transform in $X$ and solve explicitly the resulting ordinary differential equation in $Y$. It leads to sharp $L^2$ estimates on $\phi$ and its derivatives on the Fourier-Laplace side, which transfer to $L^2$ estimates in the physical space by Plancherel theorem. 

\mspace
All the analysis in the framework of the Orr-Sommerfeld equation is provided in Section \ref{sec.boundary}. In this setting, the iteration scheme mentioned above has to be modified, because the advection term creates extra difficulties. Namely, one has to add an intermediate step between steps a) and b) above, see Section \ref{sec.boundary} for details. 

\mspace
Of course, we have indicated here key ideas for the stability analysis of the linearized system \eqref{ps}. One has then to go from these estimates to the nonlinear Theorem \ref{thm.main.nl}. This will be achieved in Section  \ref{sec.nl}.  Finally we introduce the simplified notation
\begin{align*}
\| f\| = \| f\|_{L^2_{X,Y}}, \qquad \langle f,g\rangle = \langle f,g\rangle_{L^2_{X,Y}} 
\end{align*}
for convenience.

\section{Vorticity estimate under artificial boundary condition} \label{sec:vorticity:pure:slip}

In accordance with the strategy described in the previous section,  we consider here  the solution to the system
\begin{align}\label{eq.mOS}
\begin{split}
&\nu^\frac12 \Delta^2 \phi - \pa_\tau \Delta \phi - V \cdot \nabla \Delta \phi + \nabla^\bot \phi \cdot \nabla \Omega =  {\rm rot}\, F + G,\qquad \tau>0,~X\in \T_\nu, ~Y>0,\\
& \phi|_{Y=0} =\Delta \phi|_{Y=0}=0,\qquad \phi|_{\tau=0}=\phi_0.
\end{split} 
\end{align}
The goal of this section is to establish estimates for the vorticity $\omega = -\Delta \phi$, where $\phi$ is the streamfunction uniquely determined in the class $\phi\in \dot{H}_0^1(\T\times \R_+)$.
For ${\bf j}=(j_1,j_2)$ with $j_1+j_2=j$, we set 
\begin{align}\label{def.omegaj}
\omega^{\bf j} = e^{-K\tau\nu^\frac12 (j+1)}B_{j_2}\pa_X^{j_1} \omega, \qquad (\nabla \phi)^{\bf j} =  e^{-K\tau\nu^\frac12 (j+1)}B_{j_2}\pa_X^{j_1} \nabla \phi,
\end{align}
and similary, $(\Delta \omega)^{\bf j}= e^{-K\tau\nu^\frac12 (j+1)}B_{j_2}\pa_X^{j_1} \Delta \omega$. 
We also set 
\begin{align}\label{def.Vj}
V^{\bf j} =  e^{-K\tau\nu^\frac12 j}B_{j_2}\pa_X^{j_1} V, \qquad (\nabla \Omega)^{\bf j} =  e^{-K\tau\nu^\frac12 j} B_{j_2}\pa_X^{j_1} \nabla \Omega.
\end{align}
From the first equation of \eqref{eq.mOS} we observe that $\omega^{\bf j}$ satisfies, by setting ${\bf l}=(l-l_2,l_2)$,
\begin{align}\label{eq.omegaj}
\begin{split}
& -\nu^\frac12  (\Delta \omega)^{\bf j}  + (\pa_\tau + K\nu^\frac12 (j+1) + V\cdot \nabla ) \omega^{\bf j} + (\nabla^\bot \phi)^{\bf j} \cdot \nabla \Omega \\
&\quad  =  -V_2 [B_{j_2}, \pa_Y] e^{-K\tau\nu^\frac12 (j+1)} \pa_X^{j_1} \omega\\
& \qquad - \sum_{l=0}^{j-1} \sum_{\max\{0, l+j_2-j\} \leq l_2\leq \min\{l,j_2\}} \binom{j_2}{l_2} \, \binom{j-j_2}{l-l_2} V^{{\bf j-l}} \cdot (\nabla \omega)^{\bf l} \\
& \qquad -\quad  \sum_{l=0}^{j-1} \sum_{\max\{0, l+j_2-j\} \leq l_2\leq \min\{l,j_2\}} \binom{j_2}{l_2} \, \binom{j-j_2}{l-l_2} (\nabla^\bot \phi)^{{\bf l}} \cdot (\nabla \Omega)^{\bf j-l} \\
& \qquad +  {\rm rot}\, F^{\bf j} - [B_{j_2}, \pa_Y] \pa_X^{j_1} e^{-K\tau\nu^\frac12 (j+1)} F_1 + G^{\bf j}.
\end{split}
\end{align}
Here the sum $\displaystyle \sum_{l=0}^{j-1}$ is defined as $0$ for $j=0$, and  the definitions of $F^{\bf j}$ and $G^{\bf j}$ are straightforward.

To simplify notations let us introduce weighted seminorms; for a given nonnegative smooth function $\xi_j=\xi_j(\tau,X,Y)$, we set
\begin{align}\label{def.weightM}
M_{p,j, \xi_j}[\omega] = \sup_{j_2=0,\cdots,j}  \| \xi_j \omega^{(j-j_2,j_2)} \|_{L^p_\tau (0,\frac{1}{K\nu^\frac12}; L^2_{X,Y})}.
\end{align}
and also set with the definition ${\bf \xi}=(\xi_j)_{j=0}^\infty$,
\begin{align}\label{def.Norm'}
\cA F \cA_{p,{\bf \xi}}' = \sum_{j=0}^{\nu^{-\frac12}} \frac{\nu^\frac{1}{2p}(j+1)^\frac1p} {(j!)^\frac32 \nu^\frac{j}{2}} M_{p,j,\xi_j}[F].
\end{align}
Note that 
\begin{align}
\cA F\cA_{\infty,{\bf 1}}'=\cA F\cA_\infty,\qquad {\bf 1}=(1,1,\cdots).
\end{align}
The choice of $\xi_j$ is essential in the stability estimate for $\omega^{\bf j}$.
We will take
\begin{align}\label{def.xi}
\xi_j = \frac{1}{\sqrt{\pa_Y \Omega + 2 \rho_j}},
\end{align}
where 
\begin{align}\label{def.rhoj}
\rho_j = K^\frac14 C_* (1+ (j+1)^\frac12 Y)^{-2} + C_* \big ( (1+\frac{Y}{\nu^\frac14})^{-2} + \nu^\frac12 (1+Y)^{-2} + \nu \big ).
\end{align}
See Section \ref{sec:strategy} for more on the origin of this weight.  
We also introduce the norm of the boundary trace as 
\begin{align}
\cA \pa_Y \phi|_{Y=0} \cA_{bc} & = \sum_{j=0}^{\nu^{-\frac12}} \frac{\nu^\frac14 (j+1)^\frac12 }{(j!)^\frac32 \nu^\frac{j}{2}} \| e^{-K\tau\nu^\frac12 (j+1)}\pa_X^j \pa_Y \phi|_{Y=0} \|_{L^2(0,\frac{1}{K\nu^\frac12}; L^2_X)},
\end{align}
and we denote by $\dot{H}^{-1}$ the dual space of the homogeneous Sobolev space $\dot{H}_0^1(\T_\nu \times \R_+)$ (here, the subsprict $0$ means the zero boundary trace).

The main result of this section is:

\begin{proposition}\label{prop.mOS1} There exists $\kappa_1\in (0,1]$ such that the following statement holds for any $\kappa\in (0,\kappa_1]$. There exists $K_1=K_1(\kappa, C_*,C^*_j)\geq 1$ such that if $K\geq K_1$ then the system \eqref{eq.mOS} admits a unique solution  $\phi\in C([0,\infty); \dot{H}^1_0(\T_\nu \times \R_+))$ with $\omega=-\Delta \phi\in C([0,\infty); L^2(\T_\nu \times \R_+))$ satisfying
\begin{align}\label{est.prop.mOS1.1}
\begin{split}
& \cA \omega \cA_{\infty,\xi}'   + K^\frac12 \cA \omega \cA_{2,\xi}' + K^\frac14 \cA \nabla \phi \cA_{2,{\bf 1}}' + K^\frac14  \cA \pa_Y \phi|_{Y=0} \cA_{bc} \\
& \leq C \Big (  \| \nabla \phi_0 \|_{L^2_{X,Y}} + \nu^{-\frac12} [\| \Delta \phi_0 \|]    \\
& \qquad + (C_2^* +1) \nu^{-\frac12} \cA F\cA_{2,\tilde\xi^{(1)}}' +  \frac{1}{K^\frac12 \nu^\frac12} \cA G\cA_{2,\tilde \xi^{(2)}}'  +  \frac{1}{K^{\frac12}\nu^\frac14} \| G \|_{L^2(0,\frac{1}{K\nu^\frac12}; \dot{H}^{-1})}  \Big ).
\end{split}
\end{align} 
Here $C>0$ is a universal constant, while the weight $\tilde \xi^{(k)}$ is defined as 
\begin{align*}
\tilde \xi^{(k)} & = ( \frac{\xi_j}{(j+1)^\frac{k}{2}})_{j=0}^\infty.
\end{align*}
\end{proposition}

\begin{remark}\label{rem.prop.mOS1}{\rm (1) From the bound $\frac{1}{\xi_j} \le (C_1^*+8K^\frac14C_*)^\frac12$ in \eqref{est.lem.rhoj.4} below, we have 
\begin{align}\label{est.rem.prop.mOS1.1}
K^\frac{3}{16} \cA \omega \cA_{2,{\bf 1}}' \le K^\frac{3}{16}  (C_1^*+8K^\frac14C_*)^\frac12 \cA \omega \cA_{2,\xi}' \le K^\frac12 \cA \omega \cA_{2,\xi}'
\end{align}
if $K$ is large enough further depending only on $C_1^*$ and $C_*$.
Estimates \eqref{est.rem.prop.mOS1.1} and \eqref{est.prop.mOS1.1} gives the estimate of  $K^\frac{3}{16} \cA \omega \cA_{2,{\bf 1}}'$.

\noindent (2) By the deifnition of \eqref{def.Norm'}, we have 
\begin{align*}
\nu^{-\frac12} \cA F\cA_{2,\tilde\xi^{(1)}}'  = \nu^{-\frac14}  \sum_{j=0}^{\nu^{-\frac12}} \frac{M_{2,j,\xi_j}[F]}{(j!)^\frac32 \nu^\frac{j}{2}} ,\qquad 
 \nu^{-\frac12} \cA G\cA_{2,\tilde \xi^{(2)}}'   = \nu^{-\frac14}  \sum_{j=0}^{\nu^{-\frac12}} \frac{ M_{2,j,\xi_j}[G]}{(j!)^\frac32 \nu^\frac{j}{2}(j+1)^\frac12}.
 \end{align*}
Since $\xi_j \leq \frac{1}{\sqrt{\rho_j}}\leq \frac{1}{C_*\nu^\frac12}$ holds by the definitions \eqref{def.xi}-\eqref{def.rhoj} with the monotonicity condition \eqref{est.assume.3}, we have 
\begin{align}\label{est.rem.prop.mOS1.2}
\nu^{-\frac12} \cA F \cA_{2,\tilde\xi^{(1)}}'  \leq \frac{\cA F \cA_2}{C_*\nu^\frac34}.
\end{align}
}
\end{remark}

\

Before going into the details of the proof of Proposition \ref{prop.mOS1}, let us give a lemma for the weight $\xi_j$ and $\rho_j$, which will be used frequently. By the concavity condition on $\pa_Y \Omega$ in Assumption \ref{assume} (iv) and the definition of $\rho_j$ we have 
\begin{lemma}\label{lem.rhoj} There exists $C>0$ such that the following estimates hold for any $j\geq 0$.
\begin{align}\label{est.lem.rhoj.1}
\begin{split}
\xi_j^2 \leq \frac{1}{\rho_j} & \leq \frac{1}{C_* \max \{ K^\frac14  (1+ (j+1)^\frac12 Y)^{-2}, \nu \}} \qquad {\rm for}~~Y\geq 0,\\
\frac{1}{\rho_j} & \leq \frac{4}{K^\frac14 C_*} \qquad {\rm for}~~0\leq Y\leq (j+1)^{-\frac12}.
\end{split}
\end{align}
In particular, 
\begin{align}\label{est.lem.rhoj.2}
\| \frac{1+\nu^\frac12 Y}{1+Y} \xi_j\|_{L^\infty} + \| \frac{1+\nu^\frac12Y}{Y} \xi_j \|_{L^\infty (\{Y\geq (j+1)^{-\frac12}\})} \leq C (j+1)^\frac12.
\end{align}
Moreover, 
\begin{align}\label{est.lem.rhoj.3}
\| \rho_j \|_{L^\infty}  \leq 4 K^\frac14 C_*, \qquad \| \frac{Y \pa_Y \rho_j}{\rho_j} \|_{L^\infty}  \leq 2.
\end{align}
and 
\begin{align}\label{est.lem.rhoj.4}
\| \frac{1}{\xi_j}\|_{L^\infty} \leq  (C_1^* + 8K^\frac14 C_*)^\frac12, \qquad \sup_{j\geq 1} \| \frac{\xi_j}{\xi_{j-1}}\|_{L^\infty} \leq C.
\end{align} 
\end{lemma}

The proof of Lemma \ref{lem.rhoj} is straightforward from the definitions of $\xi_j$ and $\rho_j$, so we omit the details.

\subsection{Vorticity estimate for the modified system}\label{subsec.lem.prop}

In this subsection we collect lemmas for the solution to \eqref{eq.omegaj} and give the estimate for the vorticity. The main result of this subsection is as follows.

\begin{proposition}\label{prop.mOS2} There exists $\kappa_1'\in (0,1]$ such that the following statement holds for any $\kappa\in (0,\kappa_1']$.  There exists $K_1'=K_1'(\kappa, C_*, C_j^*)\geq 1$ such that if $K\geq K_1'$ then the system \eqref{eq.mOS} admits a unique solution $\phi\in C([0,\infty); \dot{H}^1_0(\T_\nu \times \R_+))$ with $\omega=-\Delta \phi\in C([0,\infty); L^2(\T_\nu \times \R_+))$ satisfying
\begin{align}\label{est.prop.mOS2}
\begin{split}
& \cA \nabla \omega \cA_{2,\tilde\xi^{(1)}}' + \cA \omega \cA_{2,\xi}' + K^\frac12 \cA \omega \cA_{2,\xi}' \\
& \leq C\Big ( \nu^{-\frac12} [\| \Delta \phi_0 \|]  +  \frac{C_2^*+1}{\nu^\frac12} \cA F\cA_{2,\tilde\xi^{(1)}}' +  \frac{1}{K^\frac12 \nu^\frac12} \cA G \cA_{2,\tilde \xi^{(2)}}'   +  \cA \nabla \phi \cA_{2,{\bf 1}}' \Big ).
\end{split}
\end{align} 
Here $C>0$ is a universal constant.
\end{proposition}

\

Since the unique solvability of the linear system \eqref{eq.mOS} itself follows from the standard theory of parabolic equations, we focus on establishing the estimate \eqref{est.prop.mOS2}.   
Then the core part of the proof of Proposition \ref{prop.mOS2} consists of the calculation of the inner product for each term in \eqref{eq.omegaj} with $\xi_j^2 \omega^{\bf j}$, where ${\bf j}=(j_1,j_2)$ with $j_1+j_2=j$ and the weight $\xi_j$ is defined as in \eqref{def.xi}.  
Let us start from the following lemma. The number $\tau_0\in (0,\frac{1}{K\nu^\frac12}]$ is taken arbitrary below.
\begin{lemma}\label{lem.prop.1} There exists $K_{1,1}=K_{1,1} (C^*_1, C_*)\geq 1$ such that if $K\geq K_{1,1}$ then we have 
\begin{align*}
& \int_0^{\tau_0} \langle  -\nu^\frac12  (\Delta \omega)^{\bf j}, \xi_j^2\omega^{{\bf j}} \rangle \, d \tau \\
& \geq \frac{\nu^\frac12}{2} \|\xi_j  (\nabla \omega)^{\bf j} \|_{L^2(0,\tau_0; L^2_{X,Y})}^2  \\
& \qquad - C\nu^\frac12 (\kappa \nu^\frac12 j_2)^2  M_{2,j-1,\xi_{j-1}}[\pa_Y \omega]^2  -C(C^*_2 +1) \nu^\frac12 (j+1) \|\xi_j \omega^{\bf j} \|_{L^2(0,\tau_0; L^2_{X,Y})}^2.
\end{align*}
Here $C>0$ is a universal constant.
\end{lemma}

\begin{proof} Let us write $\chi_\nu'=(\chi') (\nu^\frac12 Y)= \kappa e^{-\kappa \nu^\frac12 Y}$. We will frequently use the identity 
\begin{align}\label{proof.lem.prop.1.1}
[B_{j_2},\pa_Y] = -\nu^\frac12 j_2 \chi_\nu' B_{j_2-1}\pa_Y=-\frac{\nu^\frac12 j_2 \chi_\nu'}{\chi_\nu} B_{j_2}.
\end{align}
Then we observe that 
\begin{align}\label{proof.lem.prop.1.2}
(\Delta \omega)^{\bf j} = e^{-K\tau \nu^\frac12 (j+1)} B_{j_2}\pa_X^{j_1} \Delta \omega=\nabla \cdot (\nabla \omega)^{\bf j} -\frac{\nu^\frac12 j_2 \chi_\nu'}{\chi_\nu} (\pa_Y \omega)^{{\bf j}}
\end{align}
and 
\begin{align}\label{proof.lem.prop.1.3}
\nabla \omega^{\bf j} = (\nabla \omega)^{\bf j} +\nu^\frac12 j_2\chi_\nu' e^{-K\tau \nu^\frac12} (\pa_Y \omega)^{(j_1,j_2-1)}{\bf e}_2,\qquad \omega^{\bf j} =\chi_\nu e^{-K\tau \nu^\frac12}  (\pa_Y \omega)^{(j_1,j_2-1)}.
\end{align}
Here ${\bf e}_2=(0,1)$.
Hence the integration by parts gives 
\begin{align*}
\int_0^{\tau_0} -\nu^\frac12 \langle  (\Delta \omega)^{\bf j}, \xi_j^2\omega^{{\bf j}} \rangle \, d\tau  
& = \nu^\frac12 \int_0^{\tau_0}\Big (  \| \xi_j (\nabla \omega)^{\bf j} \|^2  + 2 \nu^\frac12 j_2 e^{-K\tau \nu^\frac12} \langle \xi_j (\nabla \omega)^{\bf j}, \chi_\nu' \xi_j (\pa_Y \omega)^{(j_1,j_2-1)}\rangle \\
& \qquad + \langle (\nabla \omega)^{\bf j}\cdot \nabla (\xi_j^2), \omega^{\bf  j}\rangle \Big ) \, d\tau\\
& \geq \frac{3\nu^\frac12}{4} \| \xi_j (\nabla \omega)^{\bf j} \|_{L^2(0,\tau_0; L^2)}^2 - C \nu^\frac12 (\kappa \nu^\frac12 j_2)^2 \| \xi_{j-1} (\pa_Y \omega)^{(j_1,j_2-1)} \|_{L^2 (0,\tau_0; L^2)}^2 \\
& \quad - \nu^\frac12 \int_0^{\tau_0} | \langle (\nabla \omega)^{\bf j}\cdot \nabla (\xi_j^2), \omega^{\bf  j}\rangle| \, d\tau .
\end{align*}
Here we have used $\|\frac{\xi_j}{\xi_{j-1}}\|_{L^\infty}\leq C$ in the last line as stated in Lemma \ref{lem.rhoj}. When $j_2=0$ the term $(\pa_Y \omega)^{(j_1,j_2-1)}$ is defined as $0$ for convenience.
It suffices to estimate $\langle (\nabla \omega)^{\bf j}\cdot \nabla (\xi_j^2),  \omega^{\bf  j}\rangle$.
We have 
\begin{align}\label{proof.lem.prop.1.4}
\nabla (\xi_j^2) = - \frac{\nabla \pa_Y \Omega + 2\nabla \rho_j}{\sqrt{\pa_Y \Omega + 2\rho_j}} \, \xi_j^3,
\end{align} 
which yields 
\begin{align*}
|\langle (\nabla \omega)^{\bf j}\cdot \nabla (\xi_j^2),  \omega^{\bf  j}\rangle| & \leq \| \xi_j (\nabla \omega)^{\bf j}\| \, \big ( \|  \frac{\nabla \pa_Y \Omega}{\sqrt{\pa_Y \Omega + 2\rho_j}} \, \xi_j ^2 \omega^{\bf j} \| + \| \frac{2\pa_Y \rho_j}{\sqrt{\pa_Y \Omega + 2\rho_j}} \, \xi_j ^2 \omega^{\bf j}\| \big ).
\end{align*}
To estimate  $\| \frac{\nabla \pa_Y \Omega}{\sqrt{\pa_Y \Omega + 2\rho_j}} \, \xi_j ^2 \omega^{\bf j} \|$ we decompose the integral about $Y$ into $0\leq Y\leq (j+1)^{-\frac12}$ and $Y\geq (j+1)^{-\frac12}$. Then we see from Lemma \ref{lem.rhoj} with $\frac{\xi_j^2}{\sqrt{\pa_Y \Omega + 2\rho_j}}=\xi_j^3\leq \frac{1}{\rho_j^\frac32}$,
\begin{align*}
\| \frac{\nabla \pa_Y \Omega}{\sqrt{\pa_Y \Omega + 2\rho_j}} \, \xi_j ^2 \omega^{\bf j} \|_{L^2(\{0<Y<(j+1)^{-\frac12}\})}& \leq  \|\frac{1}{\rho_j^\frac32}\|_{L^\infty (\{0<Y<(j+1)^{-\frac12}\})}  \| \nabla \pa_Y \Omega \, \omega^{\bf j}\|_{L^2(\{0<Y<(j+1)^{-\frac12}\})} \\
& \leq \frac{2}{(K^\frac14 C_*)^\frac32} \| \frac{Y}{1+\nu^\frac12 Y} \nabla \pa_Y \Omega\|_{L^\infty} \|  \frac{\omega^{\bf j}}{Y} \|\\
&\leq  \frac{C C^*_1}{(K^\frac14 C_*)^\frac32} \| \pa_Y \omega^{\bf j} \|.
\end{align*}
Here we have used Assumption \ref{assume} (iii) and the Hardy inequality $\|\frac{\omega^{\bf j}}{Y}\|\leq 4 \| \pa_Y \omega^{\bf j}\|$.
Then by using \eqref{proof.lem.prop.1.3} for $\pa_Y \omega^{\bf j}$ and \eqref{est.lem.rhoj.4} we have 
\begin{align}\label{proof.lem.prop.1.5}
\|\pa_Y\omega^{\bf j}\| &\leq \|(\pa_Y \omega)^{\bf j} \| + \kappa \nu^\frac12 j_2 \| (\pa_Y \omega)^{(j_1,j_2-1)}\|  \nonumber \\
& \leq \| \frac{1}{\xi_j} \|_{L^\infty} \| \xi_j (\pa_Y \omega)^{\bf j}\| + \kappa\nu^\frac12 j_2 \|\frac{1}{\xi_{j-1}}\|_{L^\infty} \| \xi_{j-1} (\pa_Y\omega)^{(j_1,j_2-1)} \|  \nonumber \\
& \leq C (C^*_1 + K^\frac14 C_*)^\frac12 \Big (  \| \xi_j (\pa_Y \omega)^{\bf j}\| +  \kappa \nu^\frac12 j_2  \| \xi_{j-1} (\pa_Y\omega)^{(j_1,j_2-1)} \|\Big ).
\end{align}
On the other hand, we have from Assumption \ref{assume} (iv) and \eqref{est.lem.rhoj.2} in Lemma \ref{lem.rhoj}, 
\begin{align*}
& \| \frac{\nabla \pa_Y \Omega}{\sqrt{\pa_Y \Omega + 2\rho_j}} \, \xi_j ^2 \omega^{\bf j} \|_{L^2(\{Y\geq (j+1)^{-\frac12}\})} \\
& \leq  \| \frac{Y\nabla \pa_Y \Omega}{(1+\nu^\frac12 Y) \sqrt{\pa_Y \Omega + 2\rho_j}}\| \, \| \frac{1+\nu^\frac12Y}{Y} \xi_j\|_{L^\infty (\{Y\geq (j+1)^{-\frac12}\})} \, \| \xi_j  \omega^{\bf j}\| \\
& \leq C C^*_2 (j+1)^\frac12 \| \xi_j \omega^{\bf j} \|.
\end{align*}
Next we estimate the term $\| \frac{2\pa_Y \rho_j}{\sqrt{\pa_Y \Omega + 2\rho_j}} \, \xi_j ^2 \omega^{\bf j}\|$. 
To this end we observe that 
\begin{align*}
|\pa_Y \rho_j| & \leq 2(j+1)^\frac12 K^\frac14 C_* (1+ (j+1)^\frac12 Y)^{-3} + 2C_* \nu^\frac12 (1+Y)^{-3} + 2C_* \nu^{-\frac14}(1+\frac{Y}{\nu^\frac14})^{-3} \\
& \leq 
\begin{cases}
& \displaystyle 2 (j+1)^\frac12 \rho_j + \frac{2 C_*}{Y}, \qquad 0<Y<(j+1)^{-\frac12},\\
& \displaystyle 2 (j+1)^\frac12 \rho_j + \frac{2\rho_j}{Y}, \qquad Y\geq (j+1)^{-\frac12},
\end{cases}
\end{align*}
which gives from Lemma \ref{lem.rhoj},
\begin{align*}
& \| \frac{2\pa_Y \rho_j}{\sqrt{\pa_Y \Omega + 2\rho_j}} \, \xi_j ^2 \omega^{\bf j}\|_{L^2} \\
& \leq 4(j+1)^\frac12 \| \xi_j \omega^{\bf j}\| + 2C_* \| \frac{\xi_j^3\omega^{\bf j}}{Y} \|_{L^2(\{0<Y<(j+1)^{-\frac12}\})} + 2 \| \frac{\rho_j \xi_j^3\omega^{\bf j}}{Y} \|_{L^2(\{Y\geq (j+1)^{-\frac12}\})} \\
& \leq  4(j+1)^\frac12 \| \xi_j \omega^{\bf j}\| + \frac{2C_*}{(K^\frac14 C_*)^\frac32} \| \frac{\omega^{\bf j}}{Y} \| + 2(j+1)^\frac12 \| \xi_j \omega^{\bf j} \|.
\end{align*}
Then we apply the Hardy inequality $\|\frac{\omega^{\bf j}}{Y}\|\leq 4 \| \pa_Y \omega^{\bf j}\|$ and then use \eqref{proof.lem.prop.1.5}.
Collecting these, we obtain 
\begin{align*}
& |\langle (\nabla \omega)^{\bf j}\cdot \nabla (\xi_j^2),  \omega^{\bf  j}\rangle|\\
& \leq \| \xi_j (\nabla \omega)^{\bf j}\| \, \Big ( \frac{C(C_1^* +1) (C_1^*+K^\frac14 C_*)^\frac12}{(K^\frac14 C_*)^\frac32} \big ( \| \xi_j (\pa_Y \omega)^{\bf j}\| +  \kappa \nu^\frac12 j_2  \| \xi_{j-1} (\pa_Y\omega)^{(j_1,j_2-1)} \|\big )\\
& \qquad  + C (C_2^*+1) (j+1)^\frac12 \| \xi_j \omega^{\bf j} \|\Big ).
\end{align*}
Thus, by taking $K$ large enough depending only on $C^*_1$ and $C_*$, we obtain the desired estimate as stated in Lemma \ref{lem.prop.2}. The proof is complete.
\end{proof}

\

\begin{lemma}\label{lem.prop.2} There exists $K_{1,2}=K_{1,2} (C^*_1, C_*) \geq 1$ such that if $K\geq K_{1,2}$ then we have 
\begin{align*}
& \int_0^{\tau_0} \langle  \big (\pa_\tau + K\nu^\frac12 (j+1) + V\cdot \nabla \big )\omega^{\bf j} , \xi_j^2\omega^{{\bf j}} \rangle \, d \tau \\
& \geq \frac12 \|\xi_j   \omega^{\bf j} (\tau_0) \|_{L^2_{X,Y}}^2 -\frac12\|\xi_j  \omega^{\bf j} (0) \|_{L^2_{X,Y}}^2 +\frac{K}{2} \nu^\frac12 (j+1) \|\xi_j \omega^{\bf j}\|_{L^2(0,\tau_0; L^2_{X,Y})}^2 \\
& \quad - \frac{CC_1^*\nu^\frac12}{K^\frac14C_*}\Big (  \|\xi_j  (\pa_Y \omega)^{\bf j} \|_{L^2(0,\tau_0; L^2_{X,Y})}^2 +  (\kappa \nu^\frac12 j)^2  M_{2,j-1,\xi_{j-1}} [\pa_Y \omega]^2\Big ).
\end{align*}
Here $C>0$ is a universal constant.
\end{lemma}

\begin{proof} The integration by parts yields 
\begin{align*}
& \int_0^{\tau_0} \langle  \big (\pa_\tau + K\nu^\frac12 (j+1) + V\cdot \nabla \big )\omega^{\bf j} , \xi_j^2\omega^{{\bf j}} \rangle \, d \tau \\
& = \frac12 \|\xi_j   \omega^{\bf j} (\tau_0) \|_{L^2_{X,Y}}^2 -\frac12\|\xi_j  \omega^{\bf j} (0) \|_{L^2_{X,Y}}^2 + K \nu^\frac12 (j+1) \|\xi_j \omega^{\bf j}\|_{L^2(0,\tau_0; L^2_{X,Y})}^2 \\
& \quad - \frac12 \int_0^{\tau_0} \langle \pa_\tau (\xi_j^2) + V\cdot \nabla (\xi_j^2), (\omega^{\bf j})^2\rangle\, d\tau.
\end{align*}
As for the term $\langle \pa_\tau (\xi_j^2), (\omega^{\bf j})^2\rangle$, we decompose the integral about $Y$ into $\{0<Y<(j+1)^{-\frac12}\}$ and $\{Y\geq (j+1)^{-\frac12}\}$ and compute as follows:
\begin{align}\label{proof.lem.prop.2.1}
& |\langle \pa_\tau (\xi_j^2), (\omega^{\bf j})^2\rangle|  \nonumber \\
& \leq \| (\frac{Y}{1+\nu^\frac12 Y})^2 \pa_\tau \pa_Y \Omega \|_{L^\infty} \| (\frac{1+\nu^\frac12 Y}{Y}) \xi_j^2 \omega^{\bf j} \|^2 \nonumber \\
& \leq C_1^* \nu^\frac12 \Big ( \| (1+\nu^\frac12 Y) \xi_j^2  \|_{L^\infty(\{0<Y<(j+1)^{-\frac12}\})}^2 \| \frac{\omega^{\bf j}}{Y}\|^2 + \| (\frac{1+\nu^\frac12 Y}{Y}) \xi_j \|_{L^\infty(\{Y\geq (j+1)^{-\frac12}\})}^2 \| \xi_j\omega^{\bf j}\|^2\Big ) \nonumber \\
& \leq C_1^* \nu^\frac12 \Big ( \frac{C}{(K^\frac14 C_*)^2} \| \pa_Y \omega^{\bf j}\|^2 + C(j+1) \|\xi_j \omega^{\bf j} \|^2\Big ). \quad ({\rm by ~the ~Hardy ~inequaity ~and ~Lemma \, \ref{lem.rhoj}})
\end{align}
Next we have 
\begin{align*}
|\langle V\cdot \nabla (\xi_j^2), (\omega^{\bf j})^2\rangle| \leq \| \frac{V\cdot \nabla (\pa_Y \Omega + 2\rho_j)}{\pa_Y \Omega + 2\rho_j} \|_{L^\infty} \| \xi_j \omega^{\bf j} \|^2
\end{align*} 
Then we have from Assumption \ref{assume} (iii) and Lemma \ref{lem.rhoj}, 
\begin{align*}
\| \frac{V_1\pa_Y \pa_X \Omega }{\pa_Y \Omega + 2\rho_j} \|_{L^\infty} 
&\leq \| \frac{Y(1+Y)}{(1+\nu^\frac12 Y)^2} \pa_X \pa_Y \Omega\|_{L^\infty} \| \frac{V_1 (1+\nu^\frac12 Y)^2}{Y(1+Y) \rho_j} \|_{L^\infty} \\
& \leq C_1^* \nu^\frac12 \Big ( 2  \| \frac{V_1}{Y(1+Y)\rho_j} \|_{L^\infty} + 2 \|V_1\|_{L^\infty} \| \frac{(\nu^\frac12 Y)^2}{Y(1+Y)\rho_j} \|_{L^\infty} \Big ) \\
& \leq C (C_1^*)^2\nu^\frac12 (j+1).
\end{align*}
Here we have computed as, using $V_1|_{Y=0}=0$,
\begin{align*}
 \| \frac{V_1}{Y(1+Y)\rho_j} \|_{L^\infty} & \leq \|\pa_Y V_1\|_{L^\infty}\| \frac{1}{\rho_j}\|_{L^\infty (\{0<Y<(j+1)^{-\frac12}\})} + \| V_1\|_{L^\infty} \| \frac{1}{Y(1+Y)\rho_j} \|_{L^\infty (\{Y\geq (j+1)^{-\frac12}\})}\\
 & \leq C_1^* (j+1).
 \end{align*}
 Similarly, 
 \begin{align*}
 \| \frac{V_2 (\pa_Y^2 \Omega + 2\pa_Y \rho_j)}{\pa_Y \Omega + 2\rho_j} \|_{L^\infty} 
&\leq \| \frac{Y(1+Y)^2}{(1+\nu^\frac12 Y)^3} \pa_Y^2 \Omega\|_{L^\infty} \| \frac{V_2 (1+\nu^\frac12 Y)^3}{Y(1+Y)^2 \rho_j} \|_{L^\infty} + \| \frac{V_2}{Y} \|_{L^\infty} \| \frac{Y\pa_Y \rho_j}{\rho_j} \|_{L^\infty} \\
& \leq C C_1^* \Big ( \| \frac{V_2}{Y(1+Y)^2\rho_j} \|_{L^\infty} +  \|V_2\|_{L^\infty} \| \frac{(\nu^\frac12 Y)^3}{Y(1+Y)^2\rho_j} \|_{L^\infty} \Big ) + 2 C_1^* \nu^\frac12 \\
& \leq C C_1^* \Big ( \| \pa_Y V_2\|_{L^\infty} \| \frac{1}{(1+Y)^2\rho_j} \|_{L^\infty} + \|V_2\|_{L^\infty} \, \nu^\frac32 \| \frac{1}{\rho_j}\|_{L^\infty}\Big ) + 2C_1^*\nu^\frac12 \\
& \leq CC_1^*(C_1^*+1) \nu^\frac12 (j+1).\qquad ({\rm by ~Lemma \, \ref{lem.rhoj}})
\end{align*}
Note that we have also used $\|\pa_Y V_2\|_{L^\infty}=\|\pa_X V_1\|_{L^\infty} \leq  C_1^* \nu^\frac12$.
Collecting these and applying the identity \eqref{proof.lem.prop.1.3} for $\pa_Y \omega^{\bf j}$ in \eqref{proof.lem.prop.2.1} (that is, we use \eqref{proof.lem.prop.1.5}), we obtain the desired estimate by taking $K$ large enough depending only on $C_1^*$ and $C_*$. The proof is complete. 
\end{proof}

\

\begin{lemma}\label{lem.prop.3} It follows that
\begin{align}\label{est.lem.prop.3.1}
\int_0^{\tau_0} |\langle   (\nabla^\bot \phi)^{\bf j} \cdot \nabla \Omega, \xi_j^2 \omega^{{\bf j}} \rangle | \, d \tau \leq \frac{C \big (R_{j, Lem\ref{lem.prop.3}}[\nabla \phi] \big )^2}{\nu^\frac12 (j+1)} + \frac{K}{8} \nu^\frac12 (j+1) \| \xi_j \omega^{\bf j}\|_{L^2(0,\tau_0; L^2_{X,Y})}^2, 
\end{align}
where 
\begin{align*}
& R_{j, Lem\ref{lem.prop.3}}[\nabla \phi]  \\
& := (\frac{C^*_1}{K^\frac12} + \frac{(K^\frac14 C_*)^\frac12}{K^\frac12} + \kappa^\frac12 ) \nu^\frac12 (j+1) M_{2,j}[\nabla \phi]  \, + \, \frac{(K^\frac12 C_*)^\frac12}{K^\frac12} \delta_{j\leq \nu^{-\frac12}-1} \frac{M_{2,j+1}[\pa_Y \phi]}{(j+1)^\frac12}.
\end{align*}
Here $\delta_{j\leq \nu^{-\frac12}-1}=1$ for $0\leq j\leq \nu^{-\frac12}-1$ and $0$ for $j=\nu^{-\frac12}$.
Moreover, there exists $K_{1,3}=K_{1,3}(C_1^*,C_*)\geq 1$ such that if $K\geq K_{1,3}$ then 
\begin{align}\label{est.lem.prop.3.2}
\sum_{j=0}^{\nu^{-\frac12}} \frac{R_{j, Lem\ref{lem.prop.3}}[\nabla \phi]}{(j!)^\frac32 \nu^\frac{j}{2} \nu^\frac14 (j+1)^\frac12}  \leq C \cA \nabla \phi \cA_{2,{\bf 1}}'.
\end{align}
Here $C>0$ is a universal constant.
\end{lemma}

\begin{proof} It suffices to show 
\begin{align}\label{proof.lem.prop.3.-1} 
& \int_0^{\tau_0} |\langle  (\pa_X\phi)^{\bf j} , \omega^{{\bf j}} \rangle | \, d \tau  \leq 2\kappa \nu^\frac12 j_2 (M_{2,j} [\nabla \phi])^2, 
\end{align}
\begin{align}\label{proof.lem.prop.3.-2} 
& \int_0^{\tau_0} |\langle \rho_j  (\pa_X\phi)^{\bf j} , \xi_j^2 \omega^{{\bf j}} \rangle | \, d \tau \\
& \leq 
\begin{cases}
& \displaystyle C (K^\frac14 C_*)^\frac12 \Big (  \frac{M_{2,j+1}[\pa_Y \phi]}{(j+1)^\frac12}  + \kappa \nu^\frac12 (j+1)^\frac12 \, M_{2,j}[\nabla \phi] \Big ) \, \| \xi_j \omega^{\bf j} \|_{L^2(0,\tau_0; L^2_{X,Y})},\\
& \qquad \qquad \qquad \qquad \qquad \qquad 0\leq j\leq \nu^{-\frac12}-1,\\
& \displaystyle C (K^\frac14 C_*)^\frac12  M_{2,j}[\pa_X\phi]\, \| \xi_j \omega^{\bf j} \|_{L^2(0,\tau_0; L^2_{X,Y})}, \qquad \qquad j=\nu^{-\frac12},
\end{cases}
\end{align}
and 
\begin{align}\label{proof.lem.prop.3.-3} 
& \int_0^{\tau_0} |\langle  (\pa_Y\phi)^{\bf j} \pa_X \Omega , \xi_j^2 \omega^{{\bf j}} \rangle | \, d \tau 
\leq C C^*_1 \nu^\frac12  (j+1)^\frac12 M_{2,j}[\pa_Y \phi] \| \xi_j \omega^{\bf j} \|_{L^2(0,\tau_0; L^2_{X,Y})}.
\end{align}

Let us start from \eqref{proof.lem.prop.3.-1}.
To compute $\langle (\pa_X \phi)^{\bf j}, \omega^{\bf j}\rangle$ we firstly observe that 
\begin{align}
\omega^{\bf j} = \nabla \cdot (\nabla \phi)^{\bf j} - \frac{\nu^\frac12 j_2 \chi_\nu'}{\chi_\nu} (\pa_Y \phi)^{\bf j}.
\end{align}
Then we have from the integration by parts and $[B_{j_2},\pa_Y]=-\frac{\nu^\frac12 j_2 \chi_\nu'}{\chi_\nu} B_{j_2}$,
\begin{align*}
\langle (\pa_X \phi)^{\bf j}, \omega^{\bf j}\rangle & = -\langle \nabla (\pa_X \phi)^{\bf j}, (\nabla \phi)^{\bf j}\rangle - \nu^\frac12 j_2 \langle (\pa_Y\phi)^{(j_1+1, j_2-1)},  \chi_\nu' (\pa_Y \phi)^{\bf j}\rangle\\
& =  -\langle  \pa_X (\nabla \phi)^{\bf j}, (\nabla \phi)^{\bf j}\rangle - 2 \nu^\frac12 j_2 \langle \chi_\nu' (\pa_Y \phi)^{(j_1+1,j_2-1)}, (\pa_Y\phi)^{\bf j}\rangle\\
& = - 2 \nu^\frac12 j_2 \langle \chi_\nu' (\pa_Y \phi)^{(j_1+1,j_2-1)}, (\pa_Y\phi)^{\bf j}\rangle .
\end{align*}
Hence we have from $\|\chi'_\nu\|_{L^\infty}=\kappa$,
\begin{align}\label{proof.lem.prop.3.1} 
\int_0^{\tau_0} |\langle (\pa_X \phi)^{\bf j}, \omega^{\bf j}\rangle|\, d\tau \leq 2\kappa \nu^\frac12 j_2 M_{2,j} [\pa_Y \phi]^2.
\end{align}
To estimate $\int_0^{\tau_0} |\langle \rho_j  (\pa_X\phi)^{\bf j} , \xi_j^2 \omega^{{\bf j}} \rangle | \, d \tau$ the key inequality from the definition \eqref{def.rhoj} is 
\begin{align}
\xi_j \rho_j \leq \sqrt{\rho_j} \leq C (K^\frac14 C_*)^\frac12 (1+(j+1)^\frac12 Y)^{-1} + C \nu^\frac12,
\end{align}
where $\nu^\frac12 (j+1)\leq 2$ is used. 
Thus we have from the Hardy inequality,
\begin{align}\label{proof.lem.prop.3.2} 
\int_0^{\tau_0} | \langle \rho_j  (\pa_X\phi)^{\bf j} , \xi_j^2 \omega^{{\bf j}} \rangle | \, d \tau 
& \leq \int_0^{\tau_0} \| \xi_j \rho_j (\pa_X\phi)^{\bf j} \| \| \xi_j \omega^{\bf j} \|\, d \tau\nonumber  \\
& \leq \frac{C (K^\frac14 C_*)^\frac12}{(j+1)^\frac12} \int_0^{\tau_0} \| \frac{(\pa_X\phi)^{\bf j}}{Y}  \|\,  \| \xi_j \omega^{\bf j} \|\, d\tau \nonumber \\
& \qquad + C \nu^\frac12 \| (\pa_X\phi )^{\bf j}\|_{L^2(0,\tau_0; L^2)} \| \xi_j \omega^{\bf j}\|_{L^2(0,\tau_0; L^2)}\nonumber \\
\begin{split}
& \leq \frac{C (K^\frac14 C_*)^\frac12}{(j+1)^\frac12} \| \pa_Y (\pa_X\phi)^{\bf j}  \|_{L^2(0,\tau_0; L^2)}\,  \| \xi_j \omega^{\bf j} \|_{L^2(0,\tau_0; L^2)} \\
& \qquad + C \nu^\frac12 \| (\pa_X\phi )^{\bf j}\|_{L^2(0,\tau_0; L^2)} \| \xi_j \omega^{\bf j}\|_{L^2(0,\tau_0; L^2)}
\end{split}
\end{align}
Then the desired estimate for $0\leq j\leq \nu^{-\frac12}-1$ follows from $K\tau \nu^\frac12 \leq 1$ and 
\begin{align}\label{proof.lem.prop.3.3} 
\pa_Y (\pa_X\phi)^{\bf j}=e^{K\tau \nu^\frac12} (\pa_Y \phi)^{(j_1+1,j_2)} + \nu^\frac12 j_2\chi_\nu' (\pa_Y \phi)^{(j_1+1,j_2-1)}.
\end{align}
On the other hand, the estimate for $j=\nu^{-\frac12}$ easily follows from 
\begin{align*}
\| \xi_j \rho_j (\pa_X\phi)^{\bf j} \| \leq \| \sqrt{\rho_j} \|_{L^\infty} \| (\pa_X\phi)^{\bf j}\| \leq C(K^\frac14 C_*)^\frac12 \| (\pa_X \phi)^{\bf j}\|.
\end{align*} 
Finally we have from Assumption \ref{assume} (iii) and Lemma \ref{lem.rhoj},
\begin{align*}
\| \xi_j (\pa_Y \phi)^{\bf j} \pa_X \Omega \| & \leq \| \frac{1+Y}{1+\nu^\frac12 Y} \pa_X \Omega\|_{L^\infty} \| \frac{1+\nu^\frac12 Y}{1+Y} \xi_j \|_{L^\infty} \| (\pa_Y \phi)^{\bf j}\| \\
& \leq C C_1^* \nu^{\frac12} (j+1)^\frac12 \| (\pa_Y \phi)^{\bf j}\|,
\end{align*} 
which gives 
\begin{align*}
\int_0^{\tau_0} |\langle  (\pa_Y\phi)^{\bf j} \pa_X \Omega , \xi_j^2 \omega^{{\bf j}} \rangle | \, d \tau 
& \leq C C^*_1 \nu^\frac12  (j+1)^\frac12 M_{2,j}[\pa_Y \phi] \| \xi_j \omega^{\bf j} \|_{L^2(0,\tau_0; L^2_{X,Y})}.
\end{align*}
Collecting these, we obtain \eqref{est.lem.prop.3.1}, for  the identity $\pa_Y \Omega \, \xi_j^2 = \frac{\pa_Y \Omega}{\pa_Y \Omega +2\rho_j} = 1 - 2\rho_j\xi_j^2$ holds.
The estimate \eqref{est.lem.prop.3.2} is verified from the definition $\cA \nabla \phi \cA_{2,{\bf 1}}'=\sum_{j=0}^{\nu^{-\frac12}} \frac{\nu^{\frac14} (j+1)^\frac12}{(j!)^\frac32 \nu^\frac{j}{2}} M_{2,j}[\nabla \phi]$ and 
\begin{align*}
\sum_{j=0}^{\nu^{-\frac12}-1} \frac{M_{2,j+1}[\pa_Y \phi]}{(j!)^\frac32 \nu^\frac{j}{2}\nu^\frac14 (j+1)}
& =  \sum_{j=0}^{\nu^{-\frac12}-1} \frac{\nu^\frac12 (j+1)^\frac32 M_{2,j+1}[\nabla \phi]}{((j+1)!)^\frac32\nu^{\frac{j+1}{2}} \nu^\frac14 (j+1)} \\
& \leq \sum_{j=1}^{\nu^{-\frac12}} \frac{\nu^\frac14 j^\frac12 M_{2,j}[\nabla \phi]}{(j!)^\frac32 \nu^\frac{j}{2}} \leq \cA \nabla \phi \cA_{2,{\bf 1}}'.
\end{align*} 
The proof is complete.
\end{proof}

\begin{lemma}\label{lem.prop.3'} Let $j_2\geq 1$. Then it follows that 
\begin{align}
\int_0^{\tau_0} |\langle V_2[B_{j_2},\pa_Y] e^{-K\tau\nu^\frac12 (j+1)} \pa_X^{j_1} \omega, \xi_j^2 \omega^{\bf j}\rangle|\, d\tau \leq C C_1^* \nu^\frac12 j_2 \, \| \xi_j \omega^{\bf j} \|_{L^2(0,\tau_0; L^2_{X,Y})}^2.   
\end{align}
Here $C>0$ is a universal constant.
\end{lemma}

\begin{proof} The estimate directly follows from \eqref{proof.lem.prop.1.1} and 
\begin{align*}
|V_2 \chi_\nu'| \leq \|\frac{V_2}{Y}\|_{L^\infty} | Y \chi_\nu'| \leq \|\pa_X V_1\|_{L^\infty}  | Y \chi_\nu'| \leq C_1^* \nu^\frac12 |Y\chi_\nu'| \leq C C_1^* \chi_\nu
\end{align*}
by Assumption \ref{assume} (iii) and $\kappa\nu^\frac12 Y e^{-\kappa\nu^\frac12 Y}\leq C \chi_\nu$ for a universal constant $C>0$. The proof is complete. 
\end{proof}

\

\begin{lemma}\label{lem.prop.4} Let $j\geq 1$. It follows that
\begin{align*}
& \int_0^{\tau_0} |\langle  \sum_{l=0}^{j-1} \sum_{\max\{0, l+j_2-j\} \leq l_2\leq \min\{l,j_2\}}  \binom{j_2}{l_2} \, \binom{j-j_2}{l-l_2} V^{{\bf j-l}} \cdot (\nabla \omega)^{\bf l}, \xi_j^2 \omega^{\bf j} \rangle |\, d \tau\\
& \leq \frac{C}{\kappa} R_{j, Lem\ref{lem.prop.4}}[\omega]\,  \| \xi_j \omega^{\bf j} \|_{L^2(0,\tau_0; L^2_{X,Y})},
\end{align*}
where
\begin{align*}
R_{j, Lem\ref{lem.prop.4}}[\omega] := \sum_{l=0}^{j-1} (j-l+1)^\frac12 \min\{ l+1, j-l+1\} \, \binom{j}{l} \, N_{\infty,j-l} [V] \, M_{2,l+1,\xi_l}[\omega],
\end{align*}
and 
\begin{align*}
N_{\infty,j}[V] := \sup_{j_2=0,\cdots,j} (\| B_{j_2}\pa_X^{j-j_2} V_1\|_{L^\infty (0,\frac{1}{K\nu^\frac12}; L^\infty_{X,Y})} + \kappa \| \frac{\pa_X^j V_2}{\chi_\nu} \|_{L^\infty (0,\frac{1}{K\nu^\frac12}; L^\infty_{X,Y})} ).
\end{align*} 
Moreover, 
\begin{align}\label{est.lem.prop.4.1}
\sum_{j=0}^{\nu^{-\frac12}} \frac{R_{j, Lem\ref{lem.prop.4}}[\omega]}{(j!)^\frac32 \nu^\frac{j}{2} \nu^\frac14 (j+1)^\frac12}  \leq CC_0^*  \cA \omega \cA_{2,\xi}'.
\end{align}
Here $C>0$ is a universal constant.
\end{lemma}

\begin{proof} We first observe that 
\begin{align}\label{proof.lem.prop.4.1}
\binom{j_2}{l_2}\,  \binom{j-j_2}{l-l_2}  \leq \binom{j}{l}\,, \qquad 0\leq j_2\leq l_2\leq l\leq j,
\end{align}
and 
\begin{align}\label{proof.lem.prop.4.2}
\# \Big \{ l_2\in \N\cup\{0\}~|~ \max\{0,l+j_2-j  \}\leq l_2 \leq \min\{l,j_2\} \Big \} & \leq \min\{ l+1, j-l+1\}.
\end{align}
Hence we have 
\begin{align*}
& \int_0^{\tau_0} |\langle  \sum_{l=0}^{j-1} \sum_{\max\{0, l+j_2-j\} \leq l_2\leq \min\{l,j_2\}} \binom{j_2}{l_2} \, \binom{j-j_2}{l-l_2} V^{{\bf j-l}} \cdot (\nabla \omega)^{\bf l}, \xi_j^2 \omega^{\bf j} \rangle |\, d \tau\\
& \leq \sum_{l=0}^{j-1} \binom{j}{l} \min \{l+1,j-l+1\} \| \xi_j V^{\bf j-l}\cdot (\nabla \omega)^{\bf l}\|_{L^2(0,\tau_0;L^2)} \,  \| \xi_j \omega^{\bf j} \|_{L^2(0,\tau_0; L^2)}.
\end{align*}
From the definition of $\xi_j$, we see for $0\leq l\leq j-1$,
\begin{align*}
\frac{\xi_j}{\xi_l}\leq  \sqrt{1 + \frac{(1+(j+1)^\frac12 Y)^{-2}}{(1+(l+1)^\frac12 Y)^{-2}}} \leq C (j+l-1)^\frac12,
\end{align*}
where $C>0$ is a universal constant, and thus,
\begin{align*}
 \| \xi_j V^{\bf j-l}\cdot (\nabla \omega)^{\bf l}\|_{L^2(0,\tau_0;L^2)} \leq C (j+l-1)^\frac12 \| \xi_{l} V^{\bf j-l} \cdot (\nabla \omega)^{\bf l}\|_{L^2(0,\tau_0; L^2)}.
 \end{align*}
Next we have 
\begin{align*}
\| \xi_l V_1^{\bf j-l} (\pa_X \omega)^{\bf l} \|_{L^2(0,\tau_0; L^2)} & \le \| \frac{\xi_l}{\xi_{l+1}}\|_{L^\infty} \|V_1^{\bf j-l}\|_{L^\infty} \| \xi_{l+1} \omega^{(l_1+1,l_2)} \|_{L^2(0,\tau_0; L^2)} \\
& \le C N_{\infty,j-l}[V] ~ M_{2,l+1,\xi_{l+1}}[\omega],
\end{align*}
and similarly, 
\begin{align*}
\| \xi_l V_2^{\bf j-l} (\pa_Y \omega)^{\bf l} \|_{L^2(0,\tau_0; L^2)} & \le \| \frac{\xi_l}{\xi_{l+1}}\|_{L^\infty} \| \frac{V_2^{\bf j-l}}{\chi_\nu} \|_{L^\infty} \| \xi_{l+1} \omega^{(l_1,l_2+1)} \|_{L^2(0,\tau_0; L^2)} \\
& \le \frac{C}{\kappa} N_{\infty,j-l}[V] ~ M_{2,l+1,\xi_{l+1}}[\omega],
\end{align*}
Here we have used from $\pa_XV_1+\pa_Y V_2=0$ that $\frac{V_2^{\bf j-l}}{\chi_\nu}  = (\pa_Y V_2)^{(j_1-l_1,j_2-l_2-1)} = - V_1^{(j_1-l_1+1,j_2-l_2-1)}$ for $j_2-l_2\geq 1$, which verifies $\|\frac{V_2^{\bf j-l}}{\chi_\nu} \|_{L^\infty} \leq C N_{\infty,j-l}[V]$.
The estimate \eqref{est.lem.prop.4.1} follows from 
\begin{align*}
& \sum_{j=0}^{\nu^{-\frac12}} \frac{1}{(j!)^\frac32 \nu^\frac{j}{2} \nu^\frac14 (j+1)^\frac12} \sum_{l=0}^{j-1} (j-l+1)^\frac12 \min\{ l+1, j-l+1\} \, \binom{j}{l} \, \{(j-l)! (l+1)!\}^\frac32 \nu^{\frac{j+1}{2}} \,\\
& \qquad \times  \frac{N_{\infty,j-l} [V]}{((j-l)!)^\frac32\nu^\frac{j-l}{2}}  \, \frac{M_{2,l+1,\xi_l}[\omega]}{((l+1)!)^\frac32 \nu^\frac{l+1}{2}}\\
& \leq \sum_{j=0}^{\nu^{-\frac12}}  \sum_{l=0}^{j-1} (j-l+1)^\frac12 \min\{ l+1, j-l+1\} \frac{(l+1)^\frac32}{(j+1)^\frac12(l+2)^\frac12} \big (\frac{(j-l)! l!}{j!} \big )^\frac12  \\
& \qquad \times \frac{N_{\infty,j-l} [V]}{((j-l)!)^\frac32\nu^\frac{j-l}{2}}  \, \frac{\nu^\frac14 (l+2)^\frac12 M_{2,l+1,\xi_l}[\omega]}{((l+1)!)^\frac32 \nu^\frac{l+1}{2}}\\
& \leq C \sum_{j=0}^{\nu^{-\frac12}}  \sum_{l=0}^{j-1} \frac{N_{\infty,j-l} [V]}{((j-l)!)^\frac32\nu^\frac{j-l}{2}}  \, \frac{\nu^\frac14 (l+2)^\frac12 M_{2,l+1,\xi_l}[\omega]}{((l+1)!)^\frac32 \nu^\frac{l+1}{2}}.
\end{align*}
Here we have used for $j\geq 1$,
\begin{align}\label{proof.lem.prop.4.3}
 (j-l+1)^\frac12 \min\{ l+1, j-l+1\} \frac{(l+1)^\frac32}{(j+1)^\frac12(l+2)^\frac12} \big (\frac{(j-l)! l!}{j!} \big )^\frac12  \leq C,\qquad 0\leq l\leq j-1,
 \end{align}
 with a universal constant $C>0$. Here the key is the following estimate for each $k=0,1,2,3$:
 \begin{align}\label{proof.lem.prop.4.4}
 \frac{(j-l)!l!}{j!} \leq  \frac{C}{(j+1)^{1+k}} \qquad {\rm for}~~1+k\leq l\leq j-1-k.
 \end{align}
 Then we obtain \eqref{est.lem.prop.4.1} from the Young inequality by convolution in the $l^1$ space.
 The proof is complete.
\end{proof}

\

\begin{lemma}\label{lem.prop.5} Let $j\geq 1$. It follows that
\begin{align*}
& \int_0^{\tau_0} |\langle  \quad  \sum_{l=0}^{j-1} \sum_{\max\{0, l+j_2-j\} \leq l_2\leq \min\{l,j_2\}}\binom{j_2}{l_2} \, \binom{j-j_2}{l-l_2} (\nabla^\bot \phi)^{{\bf l}} \cdot (\nabla \Omega)^{\bf j-l} , \xi_j^2 \omega^{\bf j}\rangle |\, d\tau \\
& \leq  C R_{j, Lem\ref{lem.prop.5}} [\nabla \phi] \, \| \xi_j \omega^{\bf j}\|_{L^2(0,\tau_0; L^2_{X,Y})},
\end{align*}
where
\begin{align*}
R_{j, Lem\ref{lem.prop.5}}[\nabla \phi] & := C^*_2 \nu^\frac12 j \big ( M_{2,j}[\nabla \phi] + \nu^\frac12 j M_{2,j-1}[\nabla \phi] \big ) \\
&  \quad + (j+1)^\frac12 \sum_{l=0}^{j-2}  \min\{ l+1, j-l+1\}  \, \binom{j}{l} \, N_{\infty, j-l} [\nabla \Omega] \, \\
& \qquad \qquad \qquad  \times \big ( M_{2,l+1}[\pa_Y \phi]  + \nu^\frac12 (l+1) M_{2,l} [\nabla \phi] \big ) \\
&  \quad + \nu^\frac12 (j+1)^\frac32 \,N_{\infty, 1} [\nabla \Omega]\,  M_{2,j-1}[\pa_Y \phi],
\end{align*}
and 
\begin{align*}
& N_{\infty,j-l}[\nabla \Omega] \\
&:= \sup_{j_2=0,\cdots,j} \Big ( \| (\frac{1+Y}{1+\nu^\frac12 Y})^2 (\pa_Y \Omega)^{\bf j}\|_{L^2(0,\frac{1}{K\nu^\frac12}; L^2_{X,Y})} + \nu^{-\frac12} \| \frac{1+Y}{1+\nu^\frac12 Y} (\pa_X \Omega)^{\bf j}\|_{L^2(0,\frac{1}{K\nu^\frac12}; L^2_{X,Y})} \Big ).
\end{align*}
Here the second term in the right-hand side is defined as zero when $j=1$.
Moreover, 
\begin{align}\label{est.lem.prop.5}
\sum_{j=0}^{\nu^{-\frac12}} \frac{R_{j, Lem\ref{lem.prop.5}}[\nabla \phi]}{(j!)^\frac32 \nu^\frac{j}{2} \nu^\frac14 (j+1)^\frac12}  \leq C(C_0^* + C^*_2) \cA \nabla \phi \cA_{2,{\bf 1}}'.
\end{align}
\end{lemma}

\begin{proof} As in the proof of Lemma \ref{lem.prop.4}, we have from \eqref{proof.lem.prop.4.1} and \eqref{proof.lem.prop.4.2},
\begin{align*}
& \int_0^{\tau_0} |\langle  \quad  \sum_{l=0}^{j-1} \sum_{\max\{0, l+j_2-j\} \leq l_2\leq \min\{l,j_2\}}\binom{j_2}{l_2} \, \binom{j-j_2}{l-l_2} (\nabla^\bot \phi)^{{\bf l}} \cdot (\nabla \Omega)^{\bf j-l} , \xi_j^2 \omega^{\bf j}\rangle |\, d\tau \\
& \leq  \sum_{l=0}^{j-1}  {}_jC_l \min\{l+1,j-l+1\} \| \xi_j (\nabla ^\bot \phi)^{\bf l}\cdot (\nabla\Omega)^{\bf j-l} \|_{L^2(0,\tau_0; L^2)} \| \xi_j \omega^{\bf j} \|_{L^2(0,\tau_0; L^2)}.
\end{align*}
Then we have from Lemma \ref{lem.rhoj},
\begin{align*}
\| \xi_j (\pa_Y \phi)^{\bf l} (\pa_X\Omega)^{\bf j-l} \|_{L^2(0,\tau_0; L^2)} 
& \leq \| \frac{(1+\nu^\frac12Y)\xi_j}{1+Y}\|_{L^\infty} \| \frac{1+Y}{1+\nu^\frac12 Y} (\pa_X \Omega)^{\bf j-l}\|_{L^\infty} \| (\pa_Y \phi)^{\bf l}\|_{L^2(0,\tau_0;L^2)}\\
& \leq C\nu^\frac12 (j+1)^\frac12 N_{\infty,j-l}[\nabla \Omega] \, M_{2,l}[\pa_Y \phi].
\end{align*}
Let $j\geq 2$ and $0\leq l\leq j-2$. Then,
\begin{align*}
& \| \xi_j (\pa_X \phi)^{\bf l} (\pa_Y\Omega)^{\bf j-l} \|_{L^2(0,\tau_0; L^2)} \\
& \leq \| \frac{(1+\nu^\frac12Y)\xi_j}{1+Y}\|_{L^\infty} \| (\frac{1+Y}{1+\nu^\frac12 Y})^2 (\pa_Y \Omega)^{\bf j-l}\|_{L^\infty} \| \frac{1+\nu^\frac12 Y}{1+Y} (\pa_X \phi)^{\bf l}\|_{L^2(0,\tau_0;L^2)}\\
& \leq  C(j+1)^\frac12 N_{\infty,j-l}[\nabla \Omega] \big ( \|\pa_Y (\pa_X\phi)^{\bf l} \|_{L^2(0,\tau_0; L^2)} + \nu^\frac12 \| (\pa_X \phi)^{\bf l}\|_{L^2(0,\tau_0; L^2)}\big ),
\end{align*}
where the Hardy inequality is applied in the last line. Then \eqref{proof.lem.prop.3.3} gives 
\begin{align*}
& \| \xi_j (\pa_X \phi)^{\bf l} (\pa_Y\Omega)^{\bf j-l} \|_{L^2(0,\tau_0; L^2)} \\
& \leq C(j+1)^\frac12 N_{\infty,j-l}[\nabla \Omega] \big ( M_{2,l+1}[\pa_Y \phi] + \kappa \nu^\frac12 (l+1) M_{2,l}[\nabla  \phi] \big ),\qquad 0\leq l\leq j-2.
\end{align*}
As for the case $l=j-1$, we rather compute as, by recalling $\xi_j\leq \frac{1}{\sqrt{\pa_Y \Omega + 2\rho}}$,
\begin{align*}
\|\xi_j (\pa_X \phi)^{\bf l} (\pa_Y\Omega)^{\bf j-l} \|_{L^2(0,\tau_0; L^2)} 
& \leq \|  \frac{Y}{1+\nu^\frac12 Y}  \xi_j (\pa_Y \Omega)^{\bf j-l} \|_{L^\infty} \| \frac{1+\nu^\frac12 Y}{Y} (\pa_X \phi)^{\bf l}\|_{L^2(0,\tau_0;L^2)}\\
& \leq  C \big ( \|\frac{Y}{1+\nu^\frac12 Y}  \frac{\pa^2_{XY} \Omega}{\sqrt{\pa_Y \Omega + 2\rho}} \|_{L^\infty} + \| \frac{Y}{1+\nu^\frac12 Y} \frac{\chi_\nu \pa_Y^2\Omega}{\sqrt{\pa_Y \Omega +2\rho}} \|_{L^\infty} \big ) \\
& \qquad \times \big ( \|\pa_Y (\pa_X\phi)^{\bf l} \|_{L^2(0,\tau_0; L^2)} + \nu^\frac12 \| (\pa_X \phi)^{\bf l}\|_{L^2(0,\tau_0; L^2)}\big ).
\end{align*}
Here we have used that, when $l=j-1$,  either $(\pa_Y\Omega)^{\bf j-l}=\pa_{XY}^2\Omega$ or $\chi_\nu \pa_Y^2\Omega$ holds, and that the Hardy inequality. Then, by using $\|\frac{1+\nu^\frac12 Y}{Y} \chi_\nu \|_{L^\infty} \leq C\nu^\frac12$, Assumption \ref{assume} (iii), and \eqref{proof.lem.prop.3.3}, we have 
\begin{align*}
\|\xi_j (\pa_X \phi)^{\bf l} (\pa_Y\Omega)^{\bf j-l} \|_{L^2(0,\tau_0; L^2)} \leq  CC_2^*  \nu^\frac12 \big ( M_{2,l+1}[\pa_Y \phi] + \kappa \nu^\frac12 (l+1) M_{2,l}[\nabla \phi] \big ),\qquad l=j-1.
\end{align*}
Collecting these, we obtain the term $R_{j,Lem\ref{lem.prop.5}}[\nabla \phi]$ by noticing ${}_jC_l=j$ for $l=j-1$, as desired. The estimate \eqref{est.lem.prop.5} is proved as in \eqref{est.lem.prop.4.1} but by also using the Young inequality for convolution in the $l^1$ space together with the estimates for $j\geq 2$,
\begin{align*}
(j+1)^\frac12 \min\{l+1,j-l+1\} \frac{(l+1)^\frac32}{(j+1)^\frac12 (l+2)^\frac12} \big ( \frac{(j-l)! l!}{j!}\big )^\frac12 &\leq C, \qquad 0\leq l\leq j-2,\\
(j+1)^\frac12 \min\{l+1,j-l+1\} \frac{l+1}{(j+1)^\frac12 (l+1)^\frac12} \big ( \frac{(j-l)! l!}{j!}\big )^\frac12 &\leq C, \qquad 0\leq l\leq j-2.
\end{align*}
Note that the condition $l\leq j-2$ is crucial here, for we apply \eqref{proof.lem.prop.4.4}.
 We omit the details. The proof is complete.
\end{proof}

\begin{lemma}\label{lem.prop.6} There exists $K_{1,4}=K_{1,4}(C_1^*, C_*)\geq 1$ such that for $K\geq K_{1,4}$, 
\begin{align*}
& \int_0^{\tau_0} \langle {\rm rot}\, F^{\bf j}-[B_{j_2},\pa_Y] \pa_X^{j_1} e^{-K\tau\nu^\frac12 (j+1)} F_1, \xi_j^2\omega^{\bf j}\rangle \, d\tau \\
& \quad \leq C (C_2^*+1)  M_{2,j,\xi_j}[F]  \\
& \qquad \times \Big ( \| \xi_j (\nabla \omega)^{\bf j}\|_{L^2(0,\tau_0; L^2_{X,Y})} + \kappa \nu^\frac12 j M_{2,j-1,\xi_{j-1}}[\pa_Y \omega] + (j+1)^\frac12 \| \xi_j \omega^{\bf j} \|_{L^2(0,\tau_0; L^2_{X,Y})} \Big ),
\end{align*}
and 
\begin{align*}
\int_0^{\tau_0} \langle G^{\bf j}, \xi_j^2 \omega^{\bf j}\rangle \, d\tau \leq M_{2,j,\xi_j} [G]\, \| \xi_j \omega^{\bf j} \|_{L^2(0,\tau_0; L^2_{X,Y})}.
\end{align*}
Here $C>0$ is a universal constant.
\end{lemma}

\begin{proof} The estimate about $G^{\bf j}$ is straightforward and we focus on the estimate about $F^{\bf j}$. The integration by parts and also \eqref{proof.lem.prop.1.1} yield
\begin{align*}
& \int_0^{\tau_0} \langle {\rm rot}\, F^{\bf j}-[B_{j_2},\pa_Y] \pa_X^{j_1} e^{-K\tau\nu^\frac12 (j+1)} F_1, \xi_j^2\omega^{\bf j}\rangle \, d\tau \\
& = \int_0^{\tau_0} \langle F^{\bf j}, \nabla^\bot (\xi_j^2 \omega^{\bf j})\rangle +\nu^\frac12 j_2 \langle \chi_\nu' F_1^{\bf j}, \xi_j^2 e^{-K\tau\nu^\frac12} (\pa_Y \omega)^{(j_1,j_2-1)}\rangle \, d\tau. 
\end{align*}
The second term is bounded from above by $C \kappa \nu^\frac12 j_2 \| \xi_j F_1^{\bf j} \|_{L^2(0,\tau_0; L^2)}M_{2,j-1,\xi_{j-1}}[\pa_Y \omega]$, and thus we focus on the first term:
\begin{align*}
& \int_0^{\tau_0} \langle F^{\bf j}, \nabla^\bot (\xi_j^2 \omega^{\bf j})\rangle \, d\tau \\
& = \int_0^{\tau_0} \langle F^{\bf j}\cdot \nabla^\bot (\xi_j^2), \omega^{\bf j} \rangle + \langle F^{\bf j},\xi_j^2 (\nabla^\bot\omega)^{\bf j}\rangle +\nu^\frac12 j_2 \langle F_1^{\bf j}, \xi_j^2 \chi_\nu' e^{-K\tau\nu^\frac12} (\pa_Y\omega)^{(j_1,j_2-1)}\rangle\, d\tau\\
& \leq  \int_0^{\tau_0} \langle F^{\bf j}\cdot \nabla^\bot (\xi_j^2), \omega^{\bf j} \rangle\, d\tau + M_{2,j,\xi_j}[F] \| \xi_j (\nabla \omega)^{\bf j} \|_{L^2(0,\tau_0; L^2)} \\
& \qquad + C \kappa \nu^\frac12 j_2 \, M_{2,j,\xi_j}[F] M_{2,j-1,\xi_{j-1}}[\pa_Y \omega].
\end{align*}
Then we have from Assumption \ref{assume} (iv) and Lemma \ref{lem.rhoj}, by recalling $\nabla^\bot (\xi_j^2)=-\frac{\nabla^\bot (\pa_Y \Omega+2\rho_j)}{\sqrt{\pa_Y \Omega + 2\rho_j}} \xi_j^3= -\frac{\nabla^\bot \pa_Y \Omega}{\sqrt{\pa_Y \Omega + 2\rho_j}} \xi_j^3 -2(\nabla^\bot \rho_j) \xi_j^4$,
\begin{align*}
& \langle F^{\bf j}\cdot \nabla^\bot (\xi_j^2), \omega^{\bf j} \rangle \\
& \leq \| \xi_j F^{\bf j} \|\, \Big ( \| \frac{Y \nabla (\pa_Y\Omega+2\rho_j)}{(1+\nu^\frac12 Y)\sqrt{\pa_Y \Omega+2\rho_j}} \xi_j^2\|_{L^\infty(\{0<Y<(j+1)^{-\frac12}\})} \|\frac{1+\nu^\frac12Y}{Y}\omega^{\bf j}\|_{L^2(\{0<Y<(j+1)^{-\frac12}\})}  \\
& \quad +  \| \frac{Y \nabla \pa_Y\Omega}{(1+\nu^\frac12 Y)\sqrt{\pa_Y \Omega+2\rho_j}} \|_{L^\infty(\{Y\geq (j+1)^{-\frac12}\})}\| \frac{1+\nu^\frac12 Y}{Y} \xi_j\|_{L^\infty (\{Y\geq (j+1)^{-\frac12}\})}\|\xi_j \omega^{\bf j}\|  \\
& \qquad + \| Y \pa_Y\rho_j\, \xi_j^2\|_{L^\infty} \|\frac{1}{Y} \|_{L^\infty(\{Y\ge (j+1)^{-\frac12}\})}   \|\xi_j \omega^{\bf j}\|  \Big )\\
& \le C \|\xi_j F^{\bf j} \| \Big ( \big ( C_2^*\| \xi_j^2\|_{L^\infty (\{0<Y<(j+1)^{-\frac12}\})} + \|Y \pa_Y \rho_j \, \xi_j^2\|_{L^\infty} \|\xi_j \|_{L^\infty (\{0<Y<(j+1)^{-\frac12}\})} \big )  \| \frac{\omega^{\bf j}}{Y} \| \\
& \qquad + (C_2^* + 1)(j+1)^\frac12  \| \xi_j \omega^{\bf j} \| \Big )\\
& \leq C \|\xi_j F^{\bf j} \| \Big ( (\frac{C_2^*}{K^\frac14 C_*} + \frac{1}{(K^\frac14 C_*)^\frac12} ) \| \pa_Y \omega^{\bf j}\| + (C_2^* + 1)(j+1)^\frac12  \| \xi_j \omega^{\bf j} \| \Big ). 
\end{align*} 
Thus, the estimate \eqref{proof.lem.prop.1.5} for $\pa_Y\omega^{\bf j}$ yields the desired estimate by taking $K$ large enough depending only on $C_1^*$ and $C_*$. The proof is complete.
\end{proof}

\

We are now in position to prove Proposition \ref{prop.mOS2}.
Lemmas \ref{lem.prop.1}-\ref{lem.prop.6} imply that, by taking the supremum over $j_2=0,\cdots,j$, 
\begin{align*}
& \nu^\frac14 M_{2,j, \xi_j}[\nabla \omega] + M_{\infty,j,\xi_j}[\omega] + (K\nu^\frac12 (j+1))^\frac12 M_{2,j,\xi_j}[\omega] \\
& \leq C \Big ( \sup_{j_2=0,\cdots,j} \|\xi_j \omega^{\bf j}(0)\| + \kappa \nu^\frac14 \nu^\frac12 j M_{2,j-1, \xi_{j-1}}[\nabla \omega]  \\
& \quad + \frac{R_{j,Lem\ref{lem.prop.3}}[\nabla \phi]}{\nu^\frac14 (j+1)^\frac12} + \frac{\kappa^{-1}R_{j,Lem\ref{lem.prop.4}}[\omega] + R_{j,Lem\ref{lem.prop.5}}[\nabla \phi] + M_{2,j,\xi_j}[G]}{(K\nu^\frac12 (j+1))^\frac12} + (C_2^*+1)\nu^{-\frac14} M_{2,j,\xi_j}[F]\Big )
\end{align*}
for $j=0,1,\cdots, \nu^{-\frac12}$. Here $K\geq 1$ is taken large enough depending only on $C_*$ and $C_j^*$, while $C>0$ is a universal constant.
Hence, by taking the sum $\sum_{j=0}^{\nu^{-\frac12}}$ with the factor $\frac{1}{(j!)^\frac32 \nu^\frac{j}{2}}$, we obtain 
\begin{align*}
& \cA \nabla \omega \cA_{2,\tilde \xi^{(1)}}' + \cA \omega \cA_{\infty,\xi}' + K^\frac12 \cA \omega \cA_{2,\xi}' \\
& \leq C \Big ( \sum_{j=0}^{\nu^{-\frac12}} \frac{1}{(j!)^\frac32 \nu^\frac{j}{2}} \sup_{j_2=0,\cdots,j} \|\xi_j \omega^{\bf j}(0)\|  \, + \, \kappa \cA \nabla \omega \cA_{2,\tilde \xi^{(1)}}'  + \frac{C_0^*}{K^\frac12\kappa} \cA \omega \cA_{2,\xi}' \\
& \qquad +  ( 1+ \frac{C_0^* + C_2^*}{K^\frac12}) \cA \nabla \phi \cA_{2,{\bf 1}}' + \frac{1}{K^\frac12 \nu^\frac12} \cA G\cA_{2,\tilde \xi^{(2)}}' + \frac{C_2^* +1}{\nu^\frac12} \cA F \cA_{2,\tilde \xi^{(1)}}'\Big ).
\end{align*}
Thus we obtain  \eqref{est.prop.mOS2} by first taking $\kappa>0$ small enough and then by taking $K$ large enough, and also by using $\xi_j\leq \frac{1}{C_*\nu^\frac12}\leq \frac{1}{\nu^\frac12}$ to bound $\|\xi_j \omega^{\bf j}(0)\|$. Note that the required smallness on $\kappa$ is independent of $\nu$, $K$, $C_*$, and $C_j^*$, while the required largeness of $K$ depends only on $\kappa$, $C_*$, $C_j^*$. The proof of Proposition \ref{prop.mOS2} is complete.

\subsection{Estimate for the velocity in terms of the vorticity}

In this subsection we give the estimate of the streamfunction $\phi$ in terms of the vorticity $\omega$. We remind that $\omega=-\Delta \phi$ with the boundary condition $\phi|_{Y=0}=0$.
\begin{proposition}\label{prop.stream.mOS} There exists $\kappa_2\in (0,1]$ such that for any $K\geq 1$, $\kappa\in (0,\kappa_2]$, and $p\in [1,\infty]$,
\begin{align*}
\cA \nabla \phi \cA_{p,{\bf 1}}'  \leq C( K^\frac14 C_* + C_1^*)^\frac12  \cA \omega \cA_{p,\xi}' + C\nu^\frac{1}{2p} \| \nabla \phi^{(0,0)}\|_{L^p(0,\frac{1}{K\nu^\frac12}; L^2_{X,Y})}.
\end{align*}
Here $C>0$ is a universal constant.
\end{proposition}

\begin{proof} It suffices to show 
\begin{align}\label{proof.prop.stream.mOS.1}
\sum_{j=1}^{\nu^{-\frac12}} \frac{\nu^\frac{1}{2p} (j+1)^\frac1p}{(j!)^\frac32\nu^\frac{j}{2}} M_{p,j,1}[\nabla \phi]  \leq C( K^\frac14 C_* + C_1^*)^\frac12 \cA \omega \cA_{p,\xi}' + C\nu^\frac{1}{2p} \| \nabla \phi^{(0,0)}\|_{L^p(0,\frac{1}{K\nu^\frac12}; L^2_{X,Y})}.
\end{align}
Let $j\geq 1$ and let us recall that $\omega^{\bf j}=e^{-K\tau \nu^\frac12 (j+1)} B_{j_2}\pa_X^{j-j_2} \omega$ with $\omega=-\Delta \phi$. Computations similar to those in \eqref{proof.lem.prop.1.2} imply $\omega^{\bf j} = - \nabla \cdot (\nabla \phi)^{\bf j} + \frac{\nu^\frac12 j_2\chi_\nu'}{\chi_\nu} (\pa_Y \phi)^{\bf j}$. Then the integration by parts together with the identity $\nabla \phi^{\bf j}=(\nabla \phi)^{\bf j} + \nu^\frac12 j_2 \chi_\nu' e^{-K\tau \nu^\frac12} (\pa_Y \phi)^{(j-j_2,j_2-1)}{\bf e}_2$ yields 
\begin{align}\label{proof.prop.stream.mOS.2}
\langle \omega^{\bf j}, \phi^{\bf j}\rangle = \| (\nabla \phi)^{\bf j}\|^2 + 2\nu^\frac12 j_2 e^{-K\tau \nu^\frac12 }\langle \chi_\nu' (\pa_Y\phi)^{\bf j}, (\pa_Y \phi)^{(j-j_2,j_2-1)}\rangle.
\end{align}
Then $\langle \omega^{\bf j},\phi^{\bf j}\rangle\leq \| \xi_j\omega^{\bf j}\|\, \| \frac{\phi^{\bf j}}{\xi_j} \|$ and the definition of $\xi_j$ in \eqref{def.xi} gives
\begin{align*}
\| \frac{\phi^{\bf j}}{\xi_j} \| = \| \sqrt{\pa_Y \Omega+2\rho_j} \, \phi^{\bf j} \| 
& \leq \| (\frac{1+Y}{1+\nu^\frac12 Y})^2\pa_Y \Omega\|_{L^\infty}^\frac12 \| \frac{1+\nu^\frac12 Y}{1+Y} \phi^{\bf j} \| + \sqrt{2} \| \sqrt{\rho_j} \phi^{\bf j} \| \\
& \le (C_1^*)^\frac12 \big ( C\| \pa_Y \phi^{\bf j} \| + \nu^\frac12 \| \phi^{\bf j} \| \big ) + \sqrt{2} \| \sqrt{\rho_j} \phi^{\bf j} \|.
\end{align*}
Here we have used Assumption \ref{assume} (iii) and the Hardy inequality.
Next the definition of $\rho_j$ in \eqref{def.rhoj} implies 
\begin{align*}
\sqrt{\rho_j} \leq K^\frac18 C_*^\frac12 (1+(j+1)^\frac12 Y)^{-1} + C_*^\frac12 \Big ( (1+\frac{Y}{\nu^\frac14})^{-1} + \nu^\frac14 (1+Y)^{-1} + \nu^\frac12 \Big ),
\end{align*}
which gives from the Hardy inequality and $\nu^\frac12 (j+1)\leq 2$ and $K\geq 1$,
\begin{align*}
\| \sqrt{\rho_j} \phi^{\bf j} \| \leq C K^\frac18 C_*^\frac12 (j+1)^{-\frac12} \|\pa_Y \phi^{\bf j} \| +  C_*^\frac12 \nu^\frac12 \| \phi^{\bf j}\|.
\end{align*} 
Thus we have 
\begin{align*}
\| \frac{\phi^{\bf j}}{\xi_j} \| & \le C \big ( C_1^* + K^\frac14 C_* \big )^\frac12 \|\pa_Y \phi^{\bf j} \| + C \big ( C_1^* + C_* \big )^\frac12 \nu^\frac12 \|\phi^{\bf j} \|.
\end{align*}
Thus \eqref{proof.prop.stream.mOS.2} and the identity $\pa_Y \phi^{\bf j}=(\pa_Y \phi)^{\bf j} + \nu^\frac12 j_2 \chi_\nu' e^{-K\tau \nu^\frac12} (\pa_Y \phi)^{(j-j_2,j_2-1)}$ finally give 
\begin{align*}
\| (\nabla \phi)^{\bf j}\|\le  C \big ( C_1^* + K^\frac14 C_* \big )^\frac12 \| \xi_j\omega^{\bf j} \| + C \kappa \nu^\frac12 j_2 \| (\pa_Y \phi)^{(j-j_2,j_2-1)} \| + \frac{1}{16}\nu^\frac12 \|\phi^{\bf j}\|.
\end{align*}
Here $C>0$ is a universal constant. Taking the supremum about $j_2=0,\cdots,j$ yields
\begin{align*}
M_{p,j,1}[\nabla \phi]\leq C \big ( C_1^* + K^\frac14 C_* \big )^\frac12  M_{p,j,\xi_j}[\omega] + C\kappa \nu^\frac12 j M_{p,j-1,1}[\nabla \phi]  + \frac{1}{16} \nu^\frac12 M_{p,j,1}[\phi].
\end{align*}
Thus we have from $M_{p,j,1}[\phi]\le M_{p,j-1,1}[\nabla \phi]$ and $\frac{j+1}{j}\leq 2$ for $j \geq 1$,
\begin{align*}
& \sum_{j=1}^{\nu^{-\frac12}} \frac{\nu^\frac{1}{2p} (j+1)^\frac1p}{(j!)^\frac32\nu^\frac{j}{2}} M_{p,j,1}[\nabla \phi]  \\
& \leq C( K^\frac14 C_* + C_1^*)^\frac12 \cA \omega \cA_{p,\xi}' + \big ( C \kappa + \frac{1}{8} \big )\sum_{j=0}^{\nu^{-\frac12}} \frac{\nu^\frac{1}{2p} (j+1)^\frac1p}{(j!)^\frac32\nu^\frac{j}{2}} M_{p,j,1}[\nabla \phi].
\end{align*}
Here $C>0$ is a universal constant. By taking $\kappa$ small enough we obtain \eqref{proof.prop.stream.mOS.1}. The proof is complete.
\end{proof}

\

In view of the estimate in Proposition \ref{prop.stream.mOS} our next task is to show the estimate of the {\it zero}-th order term $\nabla \phi^{(0,0)}$.
\begin{proposition}\label{prop.zero.mOS} Let $\kappa_2\in (0,1]$ be the number in Proposition \ref{prop.stream.mOS}. There exists $K_2=K_2(C_*, C_1^*)\geq 1$ such that for any $K\geq K_2$ and $\kappa\in (0,\kappa_2]$,
\begin{align}\label{est.prop.zero.mOS.1}
\begin{split}
& \nu^\frac14 \| \omega^{(0,0)}\|_{L^2(0,\frac{1}{K\nu^\frac12}; L^2_{X,Y})} + \| \nabla \phi^{(0,0)} \|_{L^\infty (0,\frac{1}{K\nu^\frac12}; L^2_{X,Y})} + K^\frac12 \nu^\frac14 \| \nabla \phi^{(0,0)} \|_{L^2(0,\frac{1}{K\nu^\frac12}; L^2_{X,Y})} \\
& \leq C \Big ( \| \nabla \phi (0) \|_{L^2_{X,Y}} + \frac{1}{K^\frac12 \nu^\frac14} \| F \|_{L^2(0,\frac{1}{K\nu^\frac12}; L^2_{X,Y})} + \frac{1}{K^{\frac12} \nu^{\frac14}} \| G \|_{L^2(0,\frac{1}{K\nu^\frac12}; \dot{H}^{-1})} +  \cA \omega \cA_{2,\xi}'\Big ),
\end{split}
\end{align} 
Here $C>0$ is a universal constant.
\end{proposition}

\begin{proof} It suffices to show 
\begin{align}\label{proof.prop.zero.mOS.1}
\begin{split}
& \nu^\frac14 \| \omega^{(0,0)}\|_{L^2(0,\frac{1}{K\nu^\frac12}; L^2_{X,Y})} + \| \nabla \phi^{(0,0)} \|_{L^\infty (0,\frac{1}{K\nu^\frac12}; L^2_{X,Y})} + K^\frac12 \nu^\frac14 \| \nabla \phi^{(0,0)} \|_{L^2(0,\frac{1}{K\nu^\frac12}; L^2_{X,Y})} \\
& \leq C \Big ( \| \nabla \phi (0) \|_{L^2_{X,Y}} + \frac{1}{K^\frac12 \nu^\frac14} \| F \|_{L^2(0,\frac{1}{K\nu^\frac12}; L^2_{X,Y})} + \frac{1}{K^{\frac12} \nu^{\frac14}} \| G \|_{L^2(0,\frac{1}{K\nu^\frac12}; \dot{H}^{-1})} \\
& \qquad + \frac{C^*_1}{K^\frac12} \cA \pa_Y \phi \cA_{2,{\bf 1}}'\Big ).
\end{split}
\end{align}  
Indeed, estimate \eqref{est.prop.zero.mOS.1} is a direct consequence of \eqref{proof.prop.zero.mOS.1} and Proposition \ref{prop.stream.mOS} by taking $K$ large enough depending only on $C_1^*$ and $C_*$.
To prove \eqref{est.prop.zero.mOS.1} let us go back to \eqref{eq.mOS}, and we take the inner product with $\eta_R \phi$ for \eqref{eq.mOS}, where $\eta_R=\eta(Y/R)$ with a smooth cut-off $\eta$ such that $\eta=1$ for $0\leq Y\leq 1$ and $\eta=0$ for $Y\geq 1$. Then, taking the limit $R\rightarrow \infty$ after the integration by parts verifies the identity 
\begin{align}
\begin{split}
& \nu^\frac12 \| \omega^{(0,0)} \|^2 +\frac12 \frac{d}{d\tau} \| \nabla \phi^{(0,0)} \|^2 + K\nu^\frac12 \| \nabla \phi^{(0,0)} \|^2  \\
& = - \langle \Delta \phi^{(0,0)}, V \cdot \nabla \phi^{(0,0)} \rangle +  \langle F^{(0,0)}, \nabla^\bot \phi^{(0,0)} \rangle + \langle G^{(0,0)}, \phi^{(0,0)}\rangle,\qquad \tau>0.
\end{split}
\end{align}
Note that $|\langle F^{(0,0)}, \nabla^\bot \phi^{(0,0)}\rangle|\leq \|F\|\, \|\nabla \phi^{(0,0)}\|$ and $|\langle G^{(0,0)}, \phi^{(0,0)}\rangle|\leq \|G\|_{\dot{H}^{-1}}\| \nabla \phi^{(0,0)}\|$. Thus it suffices to focus on the term $- \langle \Delta \phi^{(0,0)}, V \cdot \nabla \phi^{(0,0)} \rangle$. The integration by parts and $\nabla \cdot V=0$ imply 
\begin{align*}
- \langle \Delta \phi^{(0,0)}, V \cdot \nabla \phi^{(0,0)} \rangle & =\langle \pa_X \phi^{(0,0)}, (\pa_X V) \cdot \nabla \phi^{(0,0)}\rangle + \langle \pa_Y \phi^{(0,0)}, (\pa_Y V) \cdot \nabla \phi^{(0,0)}\rangle\\
& =\langle \pa_X \phi^{(0,0)}, (\pa_X V) \cdot \nabla \phi^{(0,0)}\rangle - \langle \pa_Y \phi^{(0,0)}, (\pa_X V_2) \pa_Y \phi^{(0,0)}\rangle \\
& \quad +  \langle \pa_Y \phi^{(0,0)}, (\pa_Y V_1) \pa_X \phi^{(0,0)}\rangle\\
& \leq 2C_1^* \nu^\frac12 \| \nabla \phi^{(0,0)} \|^2 +\langle \pa_Y \phi^{(0,0)}, (\pa_Y V_1) \pa_X \phi^{(0,0)}\rangle.
\end{align*}
Here we have used Assumption \ref{assume} (ii). 
Then the last term is estimated as 
\begin{align*}
\langle \pa_Y \phi^{(0,0)}, (\pa_Y V_1) \pa_X \phi^{(0,0)}\rangle & \leq \| \frac{1+Y}{1+\nu^\frac12 Y}\pa_Y V_1\|_{L^\infty}  \| \pa_Y \phi^{(0,0)} \|\, \| \frac{1+\nu^\frac12 Y}{1+Y} \pa_X \phi^{(0,0)} \| \\
& \le C_1^* \| \pa_Y \phi^{(0,0)} \|  \Big ( C \| \pa_{XY}^2 \phi^{(0,0)} \| + \nu^\frac12 \| \pa_X\phi^{(0,0)} \| \Big ).
\end{align*}
Here we have used Assumption \ref{assume} (ii) and the Hardy inequality.
Hence by taking $K$ large enough depending only on $C_1^*$ we obtain 
\begin{align*}
 \nu^\frac12 \| \omega^{(0,0)} \|^2 +\frac12 \frac{d}{d\tau} \| \nabla \phi^{(0,0)} \|^2 + K\nu^\frac12 \| \nabla \phi^{(0,0)} \|^2   & \le \frac{C (C_1^*)^2}{K\nu^\frac12} \| \pa_X \pa_Y\phi^{(0,0)} \|^2 + C (\| F\|^2 + \|G\|_{\dot{H}^{-1}}^2).
\end{align*}
Integrating about $\tau$ shows \eqref{proof.prop.zero.mOS.1}, for $\nu^{-\frac12} \| \pa_X \pa_Y\phi^{(0,0)} \|_{L^2(0,\frac{1}{K\nu^\frac12}; L^2_{X,Y})}^2\leq (\cA \pa_Y \phi^{(0,0)} \cA_{2,{\bf 1}}')^2$ holds. 
The proof is complete.
\end{proof}

\subsection{Proof of Proposition \ref{prop.mOS1}}

Propositions \ref{prop.stream.mOS} and \ref{prop.zero.mOS} yield
\begin{align}\label{proof.prop.mOS1.1}
\begin{split}
K^\frac14 \cA \nabla \phi \cA_{2,{\bf 1}}' 
& \leq C \Big ( K^\frac14 (K^\frac14 C_* + C_1^* )^\frac12  \cA \omega \cA_{2,{\bf \xi}}' \\
& \qquad + \| \nabla \phi (0) \|_{L^2_{X,Y}} +  \frac{1}{K^\frac12\nu^\frac14} \| F \|_{L^2(0,\frac{1}{K\nu^\frac12}; L^2_{X,Y})} + \frac{1}{K^{\frac12}\nu^\frac14} \| G \|_{L^2(0,\frac{1}{K\nu^\frac12}; \dot{H}^{-1})} \Big ).
\end{split}
\end{align} 
Then \eqref{proof.prop.mOS1.1} and Proposition \ref{prop.mOS2}  give 
\begin{align}\label{proof.prop.mOS1.2}
\begin{split}
& \cA \omega \cA_{\infty,\xi}'   + K^\frac12 \cA \omega \cA_{2,\xi}' + K^\frac14 \cA \nabla \phi \cA_{2,{\bf 1}}'  \\
& \leq C \Big (  \| \nabla \phi_0 \|_{L^2_{X,Y}} + \nu^{-\frac12} [\| \Delta \phi_0 \|]  \\
& \qquad + (C_2^*+1) \nu^{-\frac12} \cA F\cA_{2,\tilde\xi^{(1)}}' +  \frac{1}{K^\frac12 \nu^\frac12} \cA G\cA_{2,\tilde \xi^{(2)}}'  +  \frac{1}{K^{\frac12}\nu^\frac14} \| G \|_{L^2(0,\frac{1}{K\nu^\frac12}; \dot{H}^{-1})}  \Big ).
\end{split}
\end{align} 
It remains to estimate the boundary trace $\cA \pa_Y \phi|_{Y=0}\cA_{bc}$.
By the interpolation inequality we have 
\begin{align*}
|\pa_X^j \pa_Y\phi (\tau, X,0)|\leq C \| \pa_X^j \pa_Y^2 \phi (\tau, X,\cdot) \|_{L^2_Y}^\frac12  \| \pa_X^j \pa_Y \phi (\tau, X,\cdot) \|_{L^2_Y}^\frac12,
\end{align*}
which implies 
\begin{align}\label{proof.prop.mOS1.3}
K^\frac14 \| \pa_Y\phi^{(j,0)}|_{Y=0} \|_{L^2(0,\frac{1}{K\nu^\frac12}; L^2_X)} 
& \leq CK^\frac14  \|  \pa_Y^2 \phi^{(j,0)} \|_{L^2(0,\frac{1}{K\nu^\frac12}; L^2_{X,Y})}^\frac12  \|  \pa_Y \phi^{(j,0)} \|_{L^2(0,\frac{1}{K\nu^\frac12}; L^2_{X,Y})}^\frac12 \nonumber \\
& \leq C \big (K^\frac14 \| \omega^{(j,0)} \|_{L^2(0,\frac{1}{K\nu^\frac12}; L^2_{X,Y})} \big )^\frac12  \big ( K^\frac14 \| \pa_Y \phi^{(j,0)} \|_{L^2(0,\frac{1}{K\nu^\frac12}; L^2_{X,Y})}\big )^\frac12.
\end{align}
Here we used the Calder{\'o}n-Zygmund inequality.
Since  \eqref{est.lem.rhoj.4}  yields $\| \omega^{(j,0)} \|_{L^2(0,\frac{1}{K\nu^\frac12}; L^2_{X,Y})} \leq (C_1^* + 8 K^\frac14 C_*)^\frac12 M_{2,j,.\xi_j} [\omega]$, we have from \eqref{proof.prop.mOS1.2} that, by taking $K$ further large enough if necessary,
\begin{align*}
K^\frac14 \cA \pa_Y \phi|_{Y=0}\cA_{bc} & \leq C (K^\frac12 \cA \omega\cA_{2,\xi}')^\frac12 (K^\frac14 \cA \nabla \phi \cA_{2,{\bf 1}}')^\frac12 \\
& \le C \Big (  \| \nabla \phi_0 \|_{L^2_{X,Y}} + \nu^{-\frac12} [\| \Delta \phi_0 \|]  \\
& \qquad + (C_2^*+1) \nu^{-\frac12} \cA F\cA_{2,\tilde\xi^{(1)}}' +  \frac{1}{K^\frac12 \nu^\frac12} \cA G\cA_{2,\tilde \xi^{(2)}}'  +  \frac{1}{K^{\frac12}\nu^\frac14} \| G \|_{L^2(0,\frac{1}{K\nu^\frac12}; \dot{H}^{-1})}  \Big ).
\end{align*}
The proof of Proposition \ref{prop.mOS1} is complete.

\section{Construction of the boundary corrector}\label{sec.boundary}

In the previous section, we constructed a solution to the Orr-Sommerfeld equation with arbitrary initial data, but artificial boundary conditions: we replaced condition $\pa_Y \phi\vert_{Y=0} = 0$ by $\Delta \phi \vert_{Y=0} = 0$.  Hence, to prove Theorem \ref{thm.main1}, we still need to understand how to correct the Neumann condition, that is how to construct solutions for  systems of the following type 
\begin{align}\label{eq.bc}
\begin{split}
&\nu^\frac12 \Delta^2 \phi - \pa_\tau \Delta \phi - V \cdot \nabla \Delta \phi + \nabla^\bot \phi \cdot \nabla \Omega = 0,\qquad \tau>0,~X\in \T_\nu, ~Y>0,\\
& \phi|_{Y=0} =0,\qquad \pa_Y \phi|_{Y=0}=h,\qquad \phi|_{\tau=0}=0. 
\end{split} 
\end{align}
Such construction will be performed through an iteration, with first approximation given by the  Stokes equation. 

\subsection{Stokes estimate}\label{subsec.Stokes}

In this subsection we consider the solution to the Stokes equations (in terms of the streamfunction):
\begin{align}\label{eq.s}
\begin{split}
& \nu^\frac12 \Delta^2 \phi - \pa_\tau \Delta \phi =0\,, \qquad \tau>0,~X\in \T_\nu,~Y>0,\\
& \phi|_{Y=0} = 0,\qquad \pa_Y \phi|_{Y=0} =h, \qquad \phi|_{\tau=0} =0.
\end{split}
\end{align}
Here $h$ is a given boundary data satisfying $h(\tau)=0$ for $\tau =0$ and $\tau \geq \frac{1}{K\nu^\frac12}$, and the bound
\begin{align}
\cA h\cA_{bc} = \sum_{j=0}^{\nu^{-\frac12}} \frac{\nu^\frac14 (j+1)^\frac12}{(j!)^\frac32 \nu^\frac{j}{2}} \| e^{-K\tau\nu^\frac12 (j+1)} \pa_X^j h\|_{L^2(0,\frac{1}{K\nu^\frac12}; L^2_X)}<\infty.
\end{align}

Set $\psi = e^{-K\tau\nu^\frac12 (j+1)} \pa_X^{j_1} \phi$, $0\le j_1\le j$, with the zero extension for $\tau\leq 0$ and let $\hat{\psi}=\hat{\psi} (\lambda,\alpha,Y)$ be the Fourier (in $X$ and $\tau$) transform of $\psi$.
Then $\hat{\psi}$ obeys the ODE
\begin{align}
\begin{split}
& \nu^\frac12 (\pa_Y^2-\alpha^2)^2 \hat{\psi} - (i\lambda +K\nu^\frac12 (j+1)) ( \pa_Y^2-\alpha^2) \hat{\psi}=0,\qquad Y>0,\\
& \hat{\psi}|_{Y=0}=0,\qquad \pa_Y\hat{\psi}|_{Y=0}=\hat{g}^{(j_1)},
\end{split}
\end{align}
where $\lambda\in \R$ and $\hat{g}^{(j_1)}$ is the Fourier transform of $g^{(j_1)}:=e^{-K\tau\nu^\frac12 (j+1)}\pa_X^{j_1} h$.
We note that 
\begin{align}
\alpha = \nu^\frac12 n,
\end{align}
where $n$ is the $n$th Fourier mode in the original variable $x\in \T$.
Assuming the decay of $(|\alpha| \hat{\psi},\pa_Y\hat{\psi})$ and the boundedness of $\hat{\psi}$, we obtain the formula
\begin{align}\label{formula.ve}
\begin{split}
\hat{\psi} (\lambda,\alpha,Y) & = - \frac{e^{-\gamma\, Y} - e^{-|\alpha| Y}}{\gamma - |\alpha|} \hat{g}^{(j_1)} (\lambda,\alpha),\\
\gamma  =\gamma_j (\lambda,\alpha,\nu,K) & = \sqrt{\alpha^2 + K(j+1) + \frac{i\lambda}{\nu^\frac12}},
\end{split}
\end{align}
where the square root is taken so that the real part is positive, and it follows that 
\begin{align}\label{est.gamma}
|\alpha|\le \sqrt{\alpha^2 + K(j+1)} \le \Reel (\gamma)\le |\gamma| \le \sqrt{2} \Reel (\gamma).
\end{align}
This inequality will be freely used. 
We can also check the identity
\begin{align}\label{formula.ve'}
\begin{split}
\pa_Y \hat{\psi} (\lambda,\alpha,Y) & = -e^{-\gamma Y} \hat{g}^{(j_1)} (\lambda,\alpha) + \sgn (\alpha) \alpha \hat{\psi} (\lambda,\alpha,Y).
\end{split}
\end{align}
We also have from \eqref{formula.ve}, 
\begin{align}\label{formula.vo}
-(\pa_Y^2-\alpha^2) \hat{\psi} & = \big (\gamma + |\alpha|\big ) e^{-\gamma\, Y} \hat{g}^{(j_1)}.
\end{align}
This formula will be used in estimating the vorticity field.

\begin{lemma}\label{lem.Stokes} There exists $\kappa'\in (0,1]$ such that the following statement holds for any $\kappa\in (0,\kappa']$. Let $j_1=0,\cdots, j$ and $j_2 = j-j_1$. Then 
\begin{align}\label{est.lem.Stokes.1}
|B_{j_2} i\alpha \hat{\psi} (\lambda,\alpha,Y)|  & \le    \frac{C \nu^\frac{j_2}{2} j_2! \, |\alpha \hat{g}^{(j_1)}| }{j_2+1} \Big ( Y e^{-\frac{\Reel (\gamma)}{2}Y}+  e^{-\frac{|\alpha|}{2}Y} \, \big |\frac{1-e^{-(\gamma-|\alpha|)Y}}{\gamma-|\alpha|}\big | \Big ),\\
|B_{j_2} \pa_Y \hat{\psi}(\lambda,\alpha,Y)| & \le \frac{C \nu^\frac{j_2}{2} j_2! \, |\hat{g}^{(j_1)}| }{j_2+1}  e^{-\frac{\Reel (\gamma)}{2}Y}. 
\end{align}
As a consequence,
\begin{align}\label{est.lem.Stokes.2}
\begin{split}
& \Big ( \sum_{\alpha\in \nu^\frac12 \Z} \|B_{j_2} i\alpha \hat{\psi} (\cdot,\alpha,\cdot)\|_{L^2_{\lambda,Y}}^2  + \|B_{j_2} \pa_Y \hat{\psi}(\cdot,\alpha,\cdot)\|_{L^2_{\lambda,Y}}^2 \Big )^\frac12 \\
& \leq \frac{C\nu^\frac{j_2}{2}j_2!}{K^\frac14 (j+1)^\frac14(j_2+1)} \Big ( \sum_{\alpha \in \nu^\frac12 \Z} \| \hat{g}^{(j_1)} (\cdot,\alpha)\|_{L^2_{\lambda}}^2\Big )^\frac12.
\end{split}
\end{align} 
We also have 
\begin{align}\label{est.lem.Stokes.3}
\Big ( \sum_{\alpha\in \nu^\frac12 \Z} \|\frac{1}{1+Y} B_{j_2} i\alpha \hat{\psi} (\cdot,\alpha,\cdot)\|_{L^2_{\lambda,Y}}^2  \Big )^\frac12 \leq \frac{C\nu^\frac{j_2}{2}j_2!}{K^\frac12 (j+1)^\frac12(j_2+1)} \Big ( \sum_{\alpha \in \nu^\frac12 \Z} \| \alpha \hat{g}^{(j_1)} (\cdot,\alpha)\|_{L^2_{\lambda}}^2\Big )^\frac12.
\end{align} 
Here $C>0$ is a universal constant.
\end{lemma}

\begin{proof} We first show \eqref{est.lem.Stokes.1} for $B_{j_2}i\alpha \hat{\psi}$. 
It suffices to consider the case $j_2\ge 1$, for the case $j_2=0$ is trivial from \eqref{formula.ve}.
We observe from \eqref{formula.ve} that 
\begin{align}\label{proof.lem.Stokes.1}
B_{j_2} \hat{\psi} & = - \frac{\hat{g}^{(j_1)}\chi_\nu^{j_2}}{\gamma-|\alpha|} \Big ( (-\gamma)^{j_2} e^{-\gamma Y} - (-|\alpha|)^{j_2} e^{-|\alpha|Y} \Big ) \nonumber \\
& =  - \frac{(-\gamma)^{j_2} - (-|\alpha|)^{j_2}}{\gamma-|\alpha|} \chi_\nu^{j_2} e^{-\gamma Y} \hat{g}^{(j_1)}  + (-|\alpha|)^{j_2} \chi_\nu^{j_2} e^{-|\alpha|Y}  \hat{g}^{(j_1)}\, \frac{1- e^{-(\gamma -|\alpha|) Y} }{\gamma-|\alpha|}.
\end{align} 
Since 
\begin{align*}
 (-\gamma)^{j_2} - (-|\alpha|)^{j_2} = (-1)^{j_2} \sum_{l_2=0}^{j_2-1} \binom{j_2}{l_2} (\gamma-|\alpha|)^{j_2-l_2} |\alpha|^{l_2},
 \end{align*}
we have from $\binom{j_2}{l_2} \le j_2 \binom{j_2-1}{l_2} $ for $0\le l_2\leq j_2-1$,
\begin{align*}
\big |\frac{(-\gamma)^{j_2} - (-|\alpha|)^{j_2}}{\gamma-|\alpha|}\big | \le \sum_{l_2=0}^{j_2-1} \binom{j_2}{l_2}   \big | \gamma-|\alpha| \big |^{j_2-l_2-1} |\alpha|^{l_2} 
& \le j_2 \sum_{l_2=0}^{j_2-1} \binom{j_2-1}{l_2}  \big | \gamma-|\alpha| \big |^{j_2-l_2-1} |\alpha|^{l_2} \\
& = j_2 (\big |\gamma - |\alpha|\big | + |\alpha|)^{j_2-1} \\
& \le j_2 (3|\gamma|)^{j_2-1}.
\end{align*}
Here we have used $|\alpha|\le |\gamma|$ by \eqref{est.gamma}.
Then the ineuqlaity $\chi_\nu=1-e^{-\kappa\nu^\frac12 Y} \le \kappa \nu^\frac12 Y$ implies
\begin{align*}
\big | \frac{(-\gamma)^{j_2} - (-|\alpha|)^{j_2}}{\gamma-|\alpha|} \chi_\nu^{j_2} e^{-\gamma Y}\big |
& \le j_2 \kappa \nu^\frac12 Y \big ( 3 \kappa \nu^\frac12 |\gamma| Y )^{j_2-1} e^{-\Reel (\gamma) Y} \\
& \le j_2 \kappa \nu^\frac{j_2}{2} Y\big ( 3\sqrt{2} \kappa \Reel (\gamma) Y )^{j_2-1} e^{-\Reel (\gamma) Y} \quad {\rm (by~\eqref{est.gamma})}.
\end{align*}
From the bound $r^k e^{-r}\le (\frac{k}{e})^k$ and the Stirling bound $(\frac{k}{e})^k \le (2\pi)^{-\frac12} k^{-\frac12} k!$ for $k\in \N$, we have 
\begin{align*}
(\frac12\Reel (\gamma) Y)^{j_2-1} e^{-\frac12\Reel(\gamma) Y}\le \frac{(j_2-1)!}{\sqrt{2\pi} (j_2-1)^\frac12},\qquad j_2\geq 2.
\end{align*}
This gives when $6\sqrt{2}\kappa\le \frac12$,
\begin{align*}
\big | \frac{(-\gamma)^{j_2} - (-|\alpha|)^{j_2}}{\gamma-|\alpha|} \chi_\nu^{j_2} e^{-\gamma Y}\big |
& \le \frac{\nu^\frac{j_2}{2} j_2!}{(j_2+1)} Y e^{-\frac12 \Reel(\gamma) Y}\,, \qquad j_2\geq 1.
\end{align*}
Similarly, we have for $j_2\ge 1$,
\begin{align*}
\big |(-|\alpha|)^{j_2} \chi_\nu^{j_2} e^{-|\alpha|Y} \big |\le \frac{\nu^\frac{j_2}{2}j_2!}{j_2+1} e^{-\frac12|\alpha|Y}.
\end{align*}
Hence \eqref{est.lem.Stokes.1} for $B_{j_2}i\alpha \hat{\psi}$ follows by collecting these with \eqref{proof.lem.Stokes.1}. The estimate for $B_{j_2} \pa_Y \hat{\psi}$ is proved in the same manner in view of \eqref{formula.ve'}, and we omit the details. Estimate \eqref{est.lem.Stokes.2} follows from \eqref{est.lem.Stokes.1} and the Plancherel theorem, by observing the estimates for the multipliers
\begin{align}
\| \alpha Y e^{-\frac{\Reel(\gamma)}{2} Y} \|_{L^2_Y} & \le \frac{C}{K^\frac14 (j+1)^\frac14}, \label{proof.lem.Stokes.2}\\
\| \alpha  e^{-\frac{|\alpha|}{2}Y} \, \big |\frac{1-e^{-(\gamma-|\alpha|)Y}}{\gamma-|\alpha|}\big | \|_{L^2_Y} & \le \frac{C}{K^\frac14 (j+1)^\frac14}.\label{proof.lem.Stokes.3}
\end{align} 
Here $C>0$ is a universal constant.
Estimate \eqref{proof.lem.Stokes.2} is a consequence of \eqref{est.gamma}.
As for \eqref{proof.lem.Stokes.3}, we devide into two cases. (i) The case $|\alpha|\le \frac12 K^\frac12 (j+1)^\frac12$: in this case we have from \eqref{est.gamma},
\begin{align*}
|\gamma-|\alpha|\ge |\gamma|-|\alpha| \ge \frac{|\alpha| + K^\frac12 (j+1)^\frac12}{C}
\end{align*}
with a universal constant $C>0$, which gives 
\begin{align*}
\| \alpha  e^{-\frac{|\alpha|}{2}Y} \, \big |\frac{1-e^{-(\gamma-|\alpha|)Y}}{\gamma-|\alpha|}\big | \|_{L^2_Y}  \le \frac{C}{|\alpha| + K^\frac12 (j+1)^\frac12} \| \alpha  e^{-\frac{|\alpha|}{2}Y} \|_{L^2_Y} 
& \le \frac{C|\alpha|^\frac12}{|\alpha|+K^\frac12 (j+1)^\frac12}\\
& \le \frac{C}{K^\frac14 (j+1)^\frac14}.
\end{align*}
(ii) The case $|\alpha|\ge \frac12 K^\frac12 (j+1)^\frac12$:  In this case we used the bound 
\begin{align*}
\sup_{\Reel(z)>0} \big | \frac{1-e^{-z}}{z} \big |\le C,
\end{align*} 
which gives 
\begin{align*}
\| \alpha  e^{-\frac{|\alpha|}{2}Y} \, \big |\frac{1-e^{-(\gamma-|\alpha|)Y}}{\gamma-|\alpha|}\big | \|_{L^2_Y}  \le C \| \alpha Y  e^{-\frac{|\alpha|}{2}Y} \|_{L^2_Y} 
& \le \frac{C}{|\alpha|^\frac12} \le \frac{C}{K^\frac14 (j+1)^\frac14}.
\end{align*}
The proof of \eqref{proof.lem.Stokes.3} is complete, and \eqref{est.lem.Stokes.2} is proved.
Estimate \eqref{est.lem.Stokes.3} is proved similarly by using \eqref{est.lem.Stokes.1}, the Plancherel theorem, and
\begin{align}
\| \frac{Y}{1+Y} e^{-\frac{\Reel(\gamma)}{2} Y} \|_{L^2_Y} & \le \frac{C}{K^\frac34 (j+1)^\frac34},\label{proof.lem.Stokes.4}\\
\| \frac{1}{1+Y} e^{-\frac{|\alpha|}{2}Y} \, \big |\frac{1-e^{-(\gamma-|\alpha|)Y}}{\gamma-|\alpha|}\big | \|_{L^2_Y} & \le \frac{C}{K^\frac12 (j+1)^\frac12}.\label{proof.lem.Stokes.5}
\end{align} 
Here $C>0$ is a universal constant. Indeed, \eqref{proof.lem.Stokes.4} is straightforward, while in \eqref{proof.lem.Stokes.5}, the reason why the estimate becomes worse is due to the case $|\alpha|\le \frac12 K^\frac12 (j+1)^\frac12$ with $|\alpha|\ll 1$, where we compute as 
\begin{align*}
  \| \frac{1}{1+Y} e^{-\frac{|\alpha|}{2}Y} \, \big |\frac{1-e^{-(\gamma-|\alpha|)Y}}{\gamma-|\alpha|}\big | \|_{L^2_Y}  \le \frac{C}{|\alpha| + K^\frac12 (j+1)^\frac12} \| \frac{1}{1+Y} e^{-\frac{|\alpha|}{2}Y} \|_{L^2_Y} 
& \le \frac{C}{K^\frac12 (j+1)^\frac12}.
\end{align*}
Here we essentially use the factor $\frac{1}{1+Y}$ to obtain the uniform estimate in $\alpha$. 
The proof is complete.
\end{proof}

\

In Propositions \ref{prop.Stokes.velocity} and \ref{prop.Stokes.vorticity} below we give estimates for the solution to \eqref{eq.bc} given by the formula as above in terms of the Fourier transform.
We always take $\kappa$ small enough so that $\kappa\in (0,\kappa']$ as in Lemma \ref{lem.Stokes}.
\begin{proposition}[Estimate for velocity]\label{prop.Stokes.velocity} It follows that
\begin{align}
\begin{split}
& \sum_{j=0}^{\nu^{-\frac12}} \frac{\nu^\frac14 (j+1)^\frac34}{(j!)^\frac32 \nu^\frac{j}{2}} M_{2,j,1} [\nabla \phi]   + \sum_{j=0}^{\nu^{-\frac12}} \frac{1}{(j!)^\frac32 \nu^{\frac{j}{2}+\frac14} (j+1)^\frac12} M_{2,j,\frac{1}{1+Y}} [\pa_X \phi]   \leq \frac{C}{K^\frac14} \cA h \cA_{bc}.
\end{split}
\end{align} 
Here $C>0$ is a universal constant.
\end{proposition}

\begin{proof} Assume that $M_{2,j,1}[\nabla \phi] = \| (\nabla \phi)^{\bf j} \|_{L^2(0,\frac{1}{K\nu^\frac12}; L^2_{X,Y})}$ for some ${\bf j}=(j_1,j_2)$ with $j_1+j_2=j$.
Note that this $j_1$ depends on $j$, and we write $j_1[j]$ if necessary.
By the Plancherel theorem the estimate \eqref{est.lem.Stokes.2} implies 
\begin{align*}
 \| (\nabla \phi)^{\bf j} \|_{L^2(0,\frac{1}{K\nu^\frac12}; L^2_{X,Y})}
 & \le \frac{C\nu^\frac{j-j_1[j]}{2}(j-j_1[j])!}{K^\frac14 (j+1)^\frac14 (j-j_1[j]+1)} \| h^{(j_1)} \|_{L^2(0,\frac{1}{K\nu^\frac12}; L^2_X)}, \\
 \qquad h^{(j_1)} & = e^{-K\tau \nu^\frac12 (j_1+1)} \pa_X^{j_1} h.
\end{align*}
Thus we have 
\begin{align*}
& \sum_{j=0}^{\nu^{-\frac12}} \frac{\nu^\frac14 (j+1)^\frac34}{(j!)^\frac32 \nu^\frac{j}{2}} M_{2,j,1} [\nabla \phi]  \\
& \le \frac{C}{K^\frac14} \sum_{j=0}^{\nu^{-\frac12}}  \frac{(j-j_1[j])!j_1[j]!}{j!(j-j_1[j]+1)} (\frac{j_1[j]!}{j!})^\frac12  (\frac{j+1}{j_1[j]+1})^\frac12  \Big ( \frac{\nu^\frac14 (j_1[j]+1)^\frac12}{(j_1[j]!)^\frac32\nu^\frac{j_1[j]}{2}}  \| h^{(j_1[j])} \|_{L^2(0,\frac{1}{K\nu^\frac12}; L^2_X)} \Big ).
\end{align*}
We decompose the summation in the right-hand side as $\sum_{j_1[j]=j}$ (i.e., $j$'s such that $0\le j\le\nu^{-\frac12}$ and $j_1[j]=j$) and $\sum_{j_1[j]\le j-1}$ (i.e., $j$'s such that $0\le j\le \nu^{-\frac12}$ and $j_1[j]\le j-1$). 
Then the sum of $\sum_{j_1[j]=j}$ is bounded from above by $\cA h\cA_{bc}$, while the sum of $\sum_{j_1[j]\le j-1}$ is bounded as
\begin{align*}
& \sum_{j_1[j]\le j-1}  \frac{(j-j_1[j])!j_1[j]!}{j!(j-j_1[j]+1)} (\frac{j_1[j]!}{j!})^\frac12  (\frac{j+1}{j_1[j]+1})^\frac12  \Big ( \frac{\nu^\frac14 (j_1[j]+1)^\frac12}{(j_1[j]!)^\frac32\nu^\frac{j_1[j]}{2}}  \| h^{(j_1[j])} \|_{L^2(0,\frac{1}{K\nu^\frac12}; L^2_X)} \Big ) \\
& \le \sum_{j_1[j]\le j-1} \frac{(j-j_1[j])!j_1[j]!}{j!(j-j_1[j]+1)} (\frac{j_1[j]!}{j!})^\frac12  (\frac{j+1}{j_1[j]+1})^\frac12  \, \sup_{0\le k\le\nu^{-\frac12}}   \Big ( \frac{\nu^\frac14 (k+1)^\frac12}{(k!)^\frac32\nu^\frac{k}{2}}  \| h^{(k)} \|_{L^2(0,\frac{1}{K\nu^\frac12}; L^2_X)} \Big ) \\
& \le C \cA h\cA_{bc}.
\end{align*}
Indeed, it suffices to use 
\begin{align}
\sum_{j_1[j] \le j-1}  \frac{(j-j_1[j])!j_1[j]!}{j!(j-j_1[j]+1)} (\frac{j_1[j]!}{j!})^\frac12  (\frac{j+1}{j_1[j]+1})^\frac12  & \le  C \sum_{j_1[j]\le  j-1} (j+1)^{-\frac32} \le C.
\end{align} 
Next we prove the estimate about $M_{2,j,\frac{1}{1+Y}}[\pa_X\phi]$.
Arguming as above, we have from \eqref{est.lem.Stokes.3} that, for $0\le j\le \nu^{-\frac12}-1$,
\begin{align*}
M_{2,j,\frac{1}{1+Y}}[\pa_X \phi] \le \frac{C\nu^\frac{j-j_1[j]}{2}(j-j_1[j])!}{K^\frac12 (j+1)^\frac12 (j-j_1[j]+1)}  \| \pa_X h^{(j_1[j])} \|_{L^2(0,\frac{1}{K\nu^\frac12}; L^2_X)},
\end{align*}
where $j_1[j]$ is taken similarly as in the above argument.
Thus we have 
\begin{align*}
& \sum_{j=0}^{\nu^{-\frac12}} \frac{1}{(j!)^\frac32 \nu^{\frac{j}{2}+\frac14} (j+1)^\frac12} M_{2,j,\frac{1}{1+Y}} [\pa_X \phi]  \\
& \le \frac{C}{K^\frac12} \sum_{j=0}^{\nu^{-\frac12}-1} \frac{(j-j_1[j])!}{(j!)^\frac32 \nu^{\frac{j_1}{2}+\frac14} (j+1)(j-j_1[j]+1)}  \| \pa_X h^{(j_1)} \|_{L^2(0,\frac{1}{K\nu^\frac12}; L^2_X)}  \\
& \qquad + \frac{M_{2,j,\frac{1}{1+Y}} [\pa_X \phi]  }{(j!)^\frac32 \nu^{\frac{j}{2}+\frac14} (j+1)^\frac12} \Big |_{j=\nu^{-\frac12}}\\
& \le \frac{C}{K^\frac12} \sum_{j=0}^{\nu^{-\frac12}-1} \frac{(j-j_1[j])!}{(j!)^\frac32 \nu^{\frac{j_1}{2}+\frac14} (j+1)(j-j_1[j]+1)}  \| h^{(j_1+1)} \|_{L^2(0,\frac{1}{K\nu^\frac12}; L^2_X)} \\
& \qquad  + \frac{ \nu^\frac14 (j+1)^\frac34 M_{2,j,1} [\pa_X \phi]  }{(j!)^\frac32 \nu^{\frac{j}{2}}} \Big |_{j=\nu^{-\frac12}}.
\end{align*}
The second term is bounded from above by $\frac{C}{K^\frac14} \cA h \cA_{bc}$, as we have shown as above.
As for the first term, we again decompose the summation $\sum_{j=0}^{\nu^{-\frac12}-1}$ into $\sum_{j_1[j]=j}$ and $\sum_{j_1[j]\le j-1}$, as we have done previously.
Then the sum of  $\sum_{j_1[j]=j}$ is bounded from above by $C\cA h\cA_{bc}$, while the sum of $\sum_{j_1[j]\le j-1}$ is estimated as
\begin{align*}
& \sum_{j_1[j]\le j-1}  \frac{(j-j_1[j])!}{(j!)^\frac32 \nu^{\frac{j_1[j]}{2}+\frac14} (j+1)(j-j_1[j]+1)}  \| h^{(j_1[j]+1)} \|_{L^2(0,\frac{1}{K\nu^\frac12}; L^2_X)} \\
& \le \sum_{j_1[j]\le j-1} \frac{(j-j_1[j])!(j_1[j]+1)!}{j!} (\frac{(j_1[j]+1)!}{j!})^\frac12   \frac{1}{(j+1)(j_1[j]+1)^\frac12 (j-j_1[j]+1)}\\
& \qquad \qquad \qquad \times   \, \sup_{0\le k\le\nu^{-\frac12}}   \Big ( \frac{\nu^\frac14 (k+1)^\frac12}{(k!)^\frac32\nu^\frac{k}{2}}  \| h^{(k)} \|_{L^2(0,\frac{1}{K\nu^\frac12}; L^2_X)} \Big )\\
& \le C \sum_{j=0}^{\nu^{-\frac12}} \frac{1}{(j+1)^\frac32} \, \sup_{0\le k\le\nu^{-\frac12}}   \Big ( \frac{\nu^\frac14 (k+1)^\frac12}{(k!)^\frac32\nu^\frac{k}{2}}  \| h^{(k)} \|_{L^2(0,\frac{1}{K\nu^\frac12}; L^2_X)} \Big ) \le C \cA h \cA_{bc}.
\end{align*}
The proof is complete.
\end{proof}

\

Next we show the estimate for the vorticity field.
The argument is similar to the one for the velocity.
\begin{lemma}\label{lem.Stokes.vo} There exists $\kappa''\in (0,1]$ such that the following statement holds for any $\kappa\in (0,\kappa'']$. Let $j_1=0,\cdots, j$ and $j_2 = j-j_1$. Then 
\begin{align}\label{est.lem.Stokes.1'}
|B_{j_2} (\pa_Y^2-\alpha^2) \hat{\psi} (\lambda,\alpha,Y)| + | Y B_{j_2} \pa_Y (\pa_Y^2-\alpha^2) \hat{\psi} (\lambda,\alpha,Y)|   \le  \frac{C \nu^\frac{j_2}{2} j_2!}{j_2+1} |\gamma| \,  e^{-\frac{\Reel (\gamma)}{2}Y}|\hat{g}^{(j_1)}| .
\end{align}
As a consequence, for $\theta'\in [-\frac12,2]$,
\begin{align}\label{est.lem.Stokes.2'}
\begin{split}
& \Big ( \sum_{\alpha\in \nu^\frac12 \Z} \| Y^{1+\theta'} B_{j_2}  (\pa_Y^2-\alpha^2) \hat{\psi} (\cdot,\alpha,\cdot)\|_{L^2_{\lambda,Y}}^2   \\
& \quad  + \| Y^{2+\theta'} B_{j_2} \alpha (\pa_Y^2-\alpha^2) \hat{\psi} (\cdot,\alpha,\cdot)\|_{L^2_{\lambda,Y}}^2 +  \| Y^{2+\theta'} B_{j_2} \pa_Y (\pa_Y^2-\alpha^2) \hat{\psi} (\cdot,\alpha,\cdot)\|_{L^2_{\lambda,Y}}^2 \Big )^\frac12 \\
& \leq \frac{C\nu^\frac{j_2}{2}j_2!}{K^{\frac{\theta'}{2}+\frac14} (j+1)^{\frac{\theta'}{2}+\frac14} (j_2+1)} \Big ( \sum_{\alpha \in \nu^\frac12 \Z} \| \hat{g}^{(j_1)} (\cdot,\alpha)\|_{L^2_{\lambda}}^2\Big )^\frac12.
\end{split}
\end{align} 
Here $C>0$ is a universal constant.
\end{lemma}

\begin{proof} Estimate \eqref{est.lem.Stokes.1'} follows from \eqref{formula.vo} by arguing as in the proof of \eqref{est.lem.Stokes.1}. 
Estimates \eqref{est.lem.Stokes.2'}  then follows from \eqref{est.lem.Stokes.2'}, the Plancherel theorem, and 
\begin{align*}
\| Y^{1+m} |\gamma| e^{-\frac{\Reel(\gamma)}{2}Y} \|_{L^2_Y} \le \frac{C}{\big (\Reel (\gamma)\big )^{m+\frac12}} \le \frac{C}{\big (|\alpha|+ K^\frac12 (j+1)^\frac12 \big )^{m+\frac12}} \qquad {\rm by ~\,  \eqref{est.gamma}}
\end{align*}
for $m\in [-\frac12,3]$. The details are omitted here.
The proof is complete.
\end{proof}

\

\begin{proposition}[Estimate for vorticity]\label{prop.Stokes.vorticity} Let $\theta\in [0,2]$. It follows that
\begin{align}
\sum_{j=0}^{\nu^{-\frac12}} \frac{(j+1)^\frac14}{(j!)^\frac32 \nu^\frac{j}{2}} \nu^\frac14 (j+1)^\frac12 \Big (M_{2,j,Y} [\Delta \phi]  + M_{2,j,Y^2}[\nabla \Delta \phi] \Big ) & \le \frac{C}{K^\frac14} \cA h \cA_{bc},\label{est.prop.Stokes.vorticity.1} \\
\sum_{j=0}^{\nu^{-\frac12}} \frac{(j+1)^{\frac{\theta-1}{2}}}{(j!)^\frac32 \nu^{\frac{j}{2}+\frac14}}  \Big ( M_{2,j,Y^{\frac32+\theta}} [\pa_X \Delta \phi] + \nu^\frac12 M_{2,j,Y^{\frac32+\theta}} [\pa_Y \Delta \phi] \Big ) & \leq \frac{C}{K^\frac{\theta}{2}} \cA h \cA_{bc}.\label{est.prop.Stokes.vorticity.2} 
\end{align} 
Here $C>0$ is a universal constant.
\end{proposition}

\begin{proof} Estimate \eqref{est.prop.Stokes.vorticity.1} is a consequence of \eqref{est.lem.Stokes.2'} with $\theta'=0$, by introducing $j_1[j]$ as in the proof of Proposition \ref{prop.Stokes.velocity}. As for \eqref{est.prop.Stokes.vorticity.2}, we have from \eqref{est.lem.Stokes.2'} with $\theta'=\theta-\frac12$ that 
\begin{align*}
\sum_{j=0}^{\nu^{-\frac12}} \frac{\nu^\frac14(j+1)^{\frac{\theta-1}{2}}}{(j!)^\frac32 \nu^{\frac{j}{2}}}  M_{2,j,Y^{\frac32+\theta}} [\pa_Y \Delta \phi]& \leq \frac{C}{K^\frac{\theta}{2}} \cA h \cA_{bc}.
\end{align*}
Next we have from $M_{2,j,Y^{\frac32+\theta}}[\pa_X\Delta \phi] \le C M_{2,j+1,Y^{\frac32+\theta}}[\Delta \phi]$ that 
\begin{align*}
\sum_{j=0}^{\nu^{-\frac12}-1} \frac{(j+1)^{\frac{\theta-1}{2}}}{(j!)^\frac32 \nu^{\frac{j}{2}+\frac14}}  M_{2,j,Y^{\frac32+\theta}} [\pa_X \Delta \phi] 
& \le C \sum_{j=0}^{\nu^{-\frac12}-1} \frac{(j+1)^{\frac{\theta-1}{2}}}{(j!)^\frac32 \nu^{\frac{j}{2}+\frac14}}  M_{2,j+1,Y^{\frac32+\theta}} [\Delta \phi] \\
& =C \sum_{j=0}^{\nu^{-\frac12}-1} \frac{\nu^\frac14 (j+1)^{\frac32+\frac{\theta-1}{2}}}{((j+1)!)^\frac32 \nu^{\frac{j+1}{2}}}  M_{2,j+1,Y^{\frac32+\theta}} [\Delta \phi] \\
& = C \sum_{j=1}^{\nu^{-\frac12}} \frac{\nu^\frac14 j^{\frac{\theta}{2}+1}}{(j!)^\frac32 \nu^{\frac{j}{2}}}  M_{2,j,Y^{\frac32+\theta}} [\Delta \phi]. 
\end{align*}
By arguing as in the proof of Proposition \ref{prop.Stokes.velocity}, the application of  \eqref{est.lem.Stokes.2'} gives
\begin{align*}
C \sum_{j=1}^{\nu^{-\frac12}} \frac{\nu^\frac14 j^{\frac{\theta}{2}+1}}{(j!)^\frac32 \nu^{\frac{j}{2}}}  M_{2,j,Y^{\frac32+\theta}} [\Delta \phi]\le \frac{C}{K^\frac{\theta+1}{2}} \sum_{j=1}^{\nu^{-\frac12}} \frac{\nu^\frac14 (j+1)^\frac12}{(j!)^\frac32 \nu^\frac{j}{2}} \| e^{-K\tau\nu^\frac12 (j+1)} \pa_X^j h \|_{L^2(0,\frac{1}{K\nu^\frac12}; L^2_X)},
\end{align*}
where the smoothing factor $(j+1)^{-\frac{\theta'}{2}-\frac14}$ with $\theta'=\theta+\frac12$ in \eqref{est.lem.Stokes.2'} plays a key role.
When $j=\nu^{-\frac12}$ we have 
\begin{align*}
\frac{(j+1)^{\frac{\theta-1}{2}}}{(j!)^\frac32 \nu^{\frac{j}{2}+\frac14}}  M_{2,j,Y^{\frac32+\theta}} [\pa_X \Delta \phi] \Big |_{j=\nu^{-\frac12}} 
& \le 
\frac{\nu^\frac14(j+1)^{\frac{\theta+1}{2}}}{(j!)^\frac32 \nu^{\frac{j}{2}}}  M_{2,j,Y^{\frac32+\theta}} [\pa_X \Delta \phi] \Big |_{j=\nu^{-\frac12}} \\
& \le  \frac{C}{K^\frac{\theta}{2}} \sum_{j=0}^{\nu^{-\frac12}} \frac{\nu^\frac14(j+1)^\frac12}{(j!)^\frac32 \nu^\frac{j}{2}} \| e^{-K\tau\nu^\frac12 (j+1)} \pa_X^j h \|_{L^2(0,\frac{1}{K\nu^\frac12}; L^2_X)}\\ 
& \qquad \qquad\qquad \qquad \qquad  ({\rm by~\eqref{est.lem.Stokes.2'}~{\rm with}~} \theta'=\theta-\frac12)\\
& \le \frac{C}{K^\frac{\theta}{2}}\cA h \cA_{bc}.
\end{align*}
The proof is complete.
\end{proof}

\subsection{Vorticity transport estimate}\label{subsec.Airy}

Propositions  \ref{prop.Stokes.velocity} and \ref{prop.Stokes.vorticity} of the previous paragraph reflect a strong difference between the weighted fields $(\na \phi)^{\bf j}$ and $(\Delta \phi)^{\bf j}$ associated to the Stokes solution $\phi$ of \eqref{eq.bc} : the former is not localized near the boundary, while the latter is, at scale $(K (j+1))^{-\frac12}$. This is due to a harmonic non-localized part in $\phi$, see expression \eqref{formula.ve}. As a consequence, as shown in Proposition \ref{prop.Stokes.vorticity}, for the vorticity field the weight $Y^\theta$ gives a gain  $(j+1)^{-\frac\theta2}$. In particular, the transport term $V\cdot \nabla \Delta \phi$ shares similar properties. When working in the Gevrey class $\frac32$, this term can be seen to be formally of the same size as the Stokes term $\nu^\frac12 \Delta^2 \phi-\pa_\tau \Delta \phi$. Hence, we need to add one step to our iteration in which we solve the heat-transport equations:
\begin{align}\label{eq.airy}
\begin{split}
& \nu^\frac12 \Delta^2 \phi - \pa_\tau \Delta \phi  - V\cdot \nabla \Delta \phi=H,\qquad \tau>0,~X\in \T_\nu,~Y>0,\\
& \phi|_{Y=0} = \Delta\phi|_{Y=0}=0,\qquad \phi|_{\tau=0} =0.
\end{split}
\end{align}
with $H$ will be the transport term created by the Stokes approximation. A key point in dealing with this equation rather than with the full Orr-Sommerfeld equation is that we will be able to propagate weighted estimates with weight $Y^{\theta}$, which is crucial to have sharp bounds. In the last step of our iteration, we will correct non-local stretching terms using the Orr-Sommerfeld equation with artificial boundary conditions, using the bounds of Section \ref{sec:vorticity:pure:slip}. The main result of this paragraph is

%
%
%
\begin{proposition}\label{prop.Airy}  There exists $K_3=K_3(C_1^*)\geq 1$ such that if $K\geq K_3$ then the system \eqref{eq.airy} admits a unique solution $\phi\in C([0,\infty); \dot{H}^1_0(\T_\nu \times \R_+))$ with $\omega=-\Delta \phi\in C([0,\infty); L^2(\T_\nu \times \R_+))$ satisfying, for $0\le j\le \nu^{-\frac12}$, $\kappa\in (0,1]$, and $\theta=0,1,2$,
\begin{align}\label{est.prop.Airy}
\begin{split}
& \nu^\frac14 M_{2,j,Y^\theta} [\nabla \omega] + M_{\infty,j,Y^\theta}[\omega] + K^\frac12 \nu^\frac14 (j+1)^\frac12 M_{2,j,Y^\theta}[\omega]\\
& \leq C \Big ( \kappa \nu^\frac34 j  M_{2,j-1,Y^\theta}[\nabla \omega] + \nu^\frac14 \theta M_{2,j,Y^{\theta-1}}[\omega]  + + \frac{1}{K^\frac14\nu^\frac14(j+1)^\frac14} M_{2,j,Y^{\theta+\frac12}}[H]\\
& \quad + \frac{1}{\kappa K^\frac12 \nu^\frac14 (j+1)^\frac12} \sum_{l=0}^{j-1} \min\{l+1, j-l+1\} \,  \binom{j}{l} \, N_{\infty,j-l} [V] \, M_{2,l+1,Y^\theta} [\omega] \Big ).
\end{split}
\end{align}
Here $C>0$ is a universal constant.
\end{proposition}

\begin{remark}\label{rem.prop.Airy}{\rm The solution $\omega=-\Delta\phi$ to \eqref{eq.airy} in Proposition \ref{prop.Airy} has the regularity $(\pa_\tau -\nu^\frac12 \Delta) Y^\theta \omega\in L^2_{loc}([0,\infty); L^2(\T_\nu \times \R_+))$, $\theta=0,1,2$, with the Dirichlet boundary condition. Hence, the maximal regularity for the heat equation implies 
$$\pa_\tau Y^\theta \omega, \, \Delta (Y^\theta \omega) \in L^2_{loc}([0,\infty); L^2(\T_\nu \times \R_+)).$$
}
\end{remark}

To prove Proposition \ref{prop.Airy} let us recall that $\omega^{\bf j} = e^{-K\tau \nu^\frac12 (j+1)} B_{j_2} \pa_X^{j-j_2} \omega$ satisfies 
\begin{align}\label{proof.prop.Airy.1} 
\begin{split}
& -\nu^\frac12 (\Delta \omega)^{\bf j} + \pa_\tau \omega^{\bf j} + K\nu^\frac12 (j+1) \omega^{\bf j} + V \cdot \nabla \omega^{\bf j} \\
& =   -V_2 [B_{j_2}, \pa_Y] e^{-K\tau\nu^\frac12 (j+1)} \pa_X^{j_1} \omega\\
& \quad - \sum_{l=0}^{j-1} \sum_{\max\{0, l+j_2-j\} \leq l_2\leq \min\{l,j_2\}}  \binom{j_2}{l_2}  \binom{j-j_2}{l-l_2} V^{{\bf j-l}} \cdot (\nabla \omega)^{\bf l} \\
& \quad +H^{\bf j}.
\end{split}
\end{align}
Then \eqref{est.prop.Airy} is proved by taking the inner product in \eqref{proof.prop.Airy.1} with $Y^{2\theta} \omega^{\bf j}$ for each $\theta=0,1,2$, and then by taking the supremum about $j_2=0,\cdots,j$ and about $\tau_0\in (0,\frac{1}{K\nu^\frac12}]$. 
Hence the proof proceeds as in the proof of Proposition \ref{prop.mOS2}.

\begin{lemma}\label{lem.prop.airy1} There exists $C>0$  such that for any $K\geq 1$ and $\kappa\in (0,1]$,
\begin{align*}
\int_0^{\tau_0} \langle  -\nu^\frac12 (\Delta\omega)^{\bf j}, Y^{2\theta} \omega^{\bf j}\rangle \, d\tau & \geq \frac34 \nu^\frac12 \| Y^\theta (\nabla \omega)^{\bf j}\|_{L^2(0,\tau_0; L^2_{X,Y})}^2 - C \nu^\frac12 \big (\kappa \nu^\frac12 j_2)^2 \, M_{2,j-1,Y^\theta}[\pa_Y \omega] \big )^2\\
& \quad - C \theta^2 \nu^\frac12 M_{2,j,Y^{\theta-1}} [\omega]^2. 
\end{align*}
\end{lemma}

\begin{proof} The proof is similar (and much simpler) to the one of Lemma \ref{lem.prop.1}. Indeed, the only difference is the presence of the weight $Y^{2\theta}$ with $\theta=0,1,2$, which creates the term 
\begin{align*}
2\theta \nu^\frac12 \int_0^{\tau_0} \langle Y^\theta (\pa_Y \omega)^{\bf j},  Y^{\theta-1} \omega^{\bf j}\rangle \, d\tau
\end{align*}
after the integration by parts. This is responsible for the last term in the estimate of this lemma.
The details are omitted. The proof is complete.
\end{proof}

\begin{lemma}\label{lem.prop.airy2} There exists $K_{3,2}=K_{3,2}(C_1^*)\geq 1$ such that if $K\geq K_{3,2}$ then
\begin{align*}
& \int_0^{\tau_0} \langle \pa_\tau \omega^{\bf j} + K\nu^\frac12 (j+1) \omega^{\bf j} + V \cdot \nabla \omega^{\bf j} , Y^{2\theta} \omega^{\bf j}\rangle \, d\tau \\
& \geq \frac12 \| Y^\theta \omega^{\bf j} (\tau_0)\|^2  + \frac34 K\nu^\frac12 (j+1) \|Y^\theta \omega^{\bf j} \|_{L^2(0,\tau_0; L^2_{X,Y})}^2.
\end{align*}
\end{lemma}

\begin{proof} The proof is a simple modification of the one of Lemma \ref{lem.prop.2}. 
We note that the initial data is taken as zero, and the integration by parts gives 
\begin{align*}
\int_0^{\tau_0} \langle V \cdot \nabla \omega^{\bf j} , Y^{2\theta} \omega^{\bf j}\rangle \, d\tau \le \theta \| \frac{V_2}{Y}\|_{L^\infty} \| Y^\theta\| \omega^{\bf j} \|^2_{L^2(0,\tau_0; L^2)}.
\end{align*}
Then the desired estimate follows by taking $K$ large enough depending  only on $C_1^*$, for $\| \frac{V_2}{Y}\|_{L^\infty} \le \| \pa_Y V_2\|_{L^\infty} =\|\pa_X V_1\|_{L^\infty} \le C_1^*\nu^\frac12$.
The details are omitted. The proof is complete.
\end{proof}

\

\begin{lemma}\label{lem.prop.airy3} Let $j_2\geq 1$. It follows that
\begin{align*}
\int_0^{\tau_0} \langle  -V_2 [B_{j_2}, \pa_Y] e^{-K\tau\nu^\frac12 (j+1)} \pa_X^{j_1} \omega, Y^{2\theta} \omega^{\bf j} \rangle \, d\tau \leq C C_1^* \nu^\frac12 j_2 \| Y^\theta \omega^{\bf j} \|_{L^2(0,\tau_0; L^2_{X,Y})}^2.
\end{align*}
Here $C>0$ is a universal constant.
\end{lemma}

\begin{proof} The proof is similar to the one of Lemma \ref{lem.prop.3}. The details are omitted here. The proof is complete.
\end{proof}

\

\begin{lemma}\label{lem.prop.airy4} Let $j\geq 1$. It follows that
\begin{align*}
& \int_0^{\tau_0} \langle  - \sum_{l=0}^{j-1} \sum_{\max\{0, l+j_2-j\} \leq l_2\leq \min\{l,j_2\}}  \binom{j_2}{l_2} \binom{j-j_2}{l-l_2} V^{{\bf j-l}} \cdot (\nabla \omega)^{\bf l}, Y^{2\theta} \omega^{\bf j} \rangle \, d\tau \\
& \le \frac{C}{\kappa} R_{j,Lem\ref{lem.prop.airy4}}[\omega] \, M_{2,j,Y^\theta}[\omega],
\end{align*}
where 
\begin{align*}
R_{j,Lem\ref{lem.prop.airy4}}[\omega] = \sum_{l=0}^{j-1}  \min\{l+1,j-l+1\} \,  \binom{j}{l}\, N_{\infty,j-l}[V] \, M_{2,l+1,Y^\theta}[\omega].
\end{align*}
Here $C>0$ is a universal constant, and $N_{\infty,j-l}[V]$ is defined as in Lemma \ref{lem.prop.4}.
\end{lemma}

\begin{proof} The proof is similar to the one of Lemma \ref{lem.prop.4}. The details are omitted here. The proof is complete.
\end{proof}

\

\begin{lemma}\label{lem.prop.airy5} It follows that
\begin{align*}
& \int_0^{\tau_0} \langle H^{\bf j}, Y^{2\theta} \omega^{\bf j} \rangle\, d\tau \\
& \leq 
\begin{cases} 
& \displaystyle C M_{2,j,Y^{\theta+\frac12}} [H] \Big ( M_{2,j,Y^\theta} [\pa_Y \omega] + \kappa \nu^\frac12 j M_{2,j-1,Y^\theta}[\nabla\omega]\Big )^\frac12 (M_{2,j,Y^\theta}[\omega])^\frac12,~~\theta=0,\\
& \displaystyle  C M_{2,j,Y^{\theta+\frac12}} [H]  ( M_{2,j,Y^{\theta-1}} [\omega] )^\frac12 (M_{2,j,Y^\theta}[\omega])^\frac12,~~\theta=1,2.
\end{cases}
\end{align*}
Here $C>0$ is a universal constant.
\end{lemma}

\begin{proof} The estimate follows from the inequality 
\begin{align*}
\langle H^{\bf j}, Y^{2\theta} \omega^{\bf j}\rangle \le \| Y^{\theta+\frac12} H^{\bf j}\|\, \|  Y^{\theta-\frac12} \omega^{\bf j} \| \le \| Y^{\theta+\frac12} H^{\bf j} \| \, \| Y^{\theta-1} \omega^{\bf j} \|^\frac12  \, \| Y^{\theta} \omega^{\bf j} \|^\frac12
\end{align*}
and the Hardy inequality for $\theta=0$:
\begin{align*}
\| Y^{-1} \omega^{\bf j} \|\le C \| \pa_Y \omega^{\bf j}\|\le C \Big ( \| (\pa_Y\omega)^{\bf j}\| + \kappa \nu^\frac12 j_2 \| (\pa_Y\omega)^{(j_1,j_2-1)}\|\Big ).
\end{align*}
The proof is complete.
\end{proof}

\

\begin{proof}[Proof of Proposition \ref{prop.Airy}] It suffices to show the estimate \eqref{est.prop.Airy}, but it follows from Lemmas \ref{lem.prop.airy1}-\ref{lem.prop.airy5} by dividing into the case $\theta=0$ and the case $\theta=1,2$. The details are omitted here. The proof is complete.
\end{proof}

\

\begin{corollary}\label{cor.prop.Airy1} There exists $\kappa_3\in (0,1]$ such that the following statement holds for any $\kappa\in (0,\kappa_3]$.  There exists $K_3'=K_3'(\kappa, C_0^*, C_1^*)\geq 1$ such that if $K\geq K_3'$ then the system \eqref{eq.airy} admits a unique solution $\phi\in C([0,\infty); \dot{H}^1_0(\T\times \R_+))$ with $\omega=-\Delta \phi\in C([0,\infty); L^2(\T \times \R_+))$  satisfying, for  $\theta=0,1,2$,
\begin{align}\label{est.cor.prop.Airy1.1}
\begin{split}
& \sum_{j=0}^{\nu^{-\frac12}} \frac{(j+1)^{\frac{\theta}{2}-\frac14}}{(j!)^{\frac32} \nu^\frac{j}{2}} \Big ( \nu^\frac14 M_{2,j,Y^\theta} [\nabla \omega] + M_{\infty,j,Y^\theta}[\omega] + K^\frac12 \nu^\frac14 (j+1)^\frac12 M_{2,j,Y^\theta}[\omega]\Big ) \\
& \leq \frac{C}{K^\frac14} \sum_{\theta'=0}^\theta  \sum_{j=0}^{\nu^{-\frac12}} \frac{1}{(j!)^\frac32 \nu^{\frac{j}{2}+\frac14} (j+1)^\frac{1-\theta'}{2}} M_{2,j,Y^{\theta'+\frac12}}[H],
\end{split}
\end{align}
and 
\begin{align}\label{est.cor.prop.Airy1.2}
\cA \nabla \phi \cA_{2,{\bf 1}}' + \cA \pa_Y \phi|_{Y=0} \cA_{bc} \le \frac{C}{K^{\frac34}} \sum_{\theta'=0}^1 \sum_{j=0}^{\nu^{-\frac12}} \frac{1}{(j!)^\frac32 \nu^{\frac{j}{2}+\frac14}(j+1)^\frac{1-\theta'}{2}} M_{2,j,Y^{\theta'+\frac12}}[H].
\end{align}
Here $C>0$ is a universal constant.
\end{corollary}

\begin{proof} Let us first show \eqref{est.cor.prop.Airy1.1}.
In virtue of Proposition \ref{prop.Airy} we have for $\theta=0,1,2$,
\begin{align}
\begin{split}
& \sum_{\theta'=0}^{\theta} \sum_{j=0}^{\nu^{-\frac12}} \frac{(j+1)^{\frac{\theta'}{2}-\frac14}}{(j!)^{\frac32} \nu^\frac{j}{2}} \Big ( \nu^\frac14 M_{2,j,Y^{\theta'}} [\nabla \omega] + M_{\infty,j,Y^{\theta'}}[\omega] + K^\frac12 \nu^\frac14 (j+1) ^\frac12 M_{2,j,Y^{\theta'}}[\omega]\Big ) \\
& \leq C \sum_{\theta'=0}^{\theta} \sum_{j=0}^{\nu^{-\frac12}} \frac{(j+1)^{\frac{\theta'}{2}-\frac14}}{(j!)^{\frac32} \nu^\frac{j}{2}} \\
& \qquad \times  \Big ( \kappa \nu^\frac34 j  M_{2,j-1,Y^{\theta'}}[\nabla \omega] + \nu^\frac14 \theta' M_{2,j,Y^{\theta'-1}}[\omega]  +  \frac{1}{K^\frac14\nu^\frac14(j+1)^\frac14} M_{2,j,Y^{\theta'+\frac12}}[H]\\
& \qquad + \frac{1}{\kappa K^\frac12 \nu^\frac14 (j+1)^\frac12} \sum_{l=0}^{j-1} \min\{l+1, j-l+1\} \, \binom{j}{l}  \, N_{\infty,j-l} [V] \, M_{2,l+1,Y^{\theta'}} [\omega]\Big )
\end{split} \nonumber \\
\begin{split}\label{proof.cor.prop.Airy1.1}
& \leq C \kappa \sum_{\theta'=0}^{\theta} \sum_{j=0}^{\nu^{-\frac12}-1} \frac{(j+1)^{\frac{\theta'}{2}-\frac14}}{(j!)^{\frac32} \nu^\frac{j}{2}} \nu^\frac14 M_{2,j,Y^{\theta'}} [\nabla \omega ] \\
& \quad +  C \sum_{\theta'=0}^{\theta} \theta' \sum_{j=0}^{\nu^{-\frac12}} \frac{(j+1)^{\frac{\theta'-1}{2}-\frac14}}{(j!)^{\frac32} \nu^\frac{j}{2}}  \nu^\frac14 (j+1)^\frac12 M_{2,j,Y^{\theta'-1}} [\omega] \\
& \quad + C \sum_{\theta'=0}^{\theta}  \sum_{j=0}^{\nu^{-\frac12}} \frac{(j+1)^{\frac{\theta'}{2}-\frac14}}{(j!)^{\frac32} \nu^\frac{j}{2}}   \frac{1}{K^\frac14\nu^\frac14(j+1)^\frac14} M_{2,j,Y^{\theta'+\frac12}}[H]\\
& \quad + C \sum_{\theta'=0}^{\theta}  \sum_{j=0}^{\nu^{-\frac12}} \frac{(j+1)^{\frac{\theta'}{2}-\frac14}}{(j!)^{\frac32} \nu^\frac{j}{2}}  \frac{1}{\kappa K^\frac12 \nu^\frac14 (j+1)^\frac12} \sum_{l=0}^{j-1} \min\{l+1, j-l+1\} \,  \binom{j}{l} \, N_{\infty,j-l} [V] \, M_{2,l+1,Y^{\theta'}} [\omega].
\end{split}
\end{align}
Here $C>0$ is a universal constant.
As for the last term in \eqref{proof.cor.prop.Airy1.1}, arguing as at the end of the proof of Lemma  \ref{lem.prop.4},  we find that 
\begin{align*}
&  \sum_{j=0}^{\nu^{-\frac12}} \frac{(j+1)^{\frac{\theta'}{2}-\frac14}}{(j!)^{\frac32} \nu^\frac{j}{2}}  \frac{1}{K^\frac12 \nu^\frac14 (j+1)^\frac12} \sum_{l=0}^{j-1} \min\{l+1, j-l+1\} \, \binom{j}{l}  \, N_{\infty,j-l} [V] \, M_{2,l+1,Y^{\theta'}} [\omega] \\
& \le \frac{C C_0^*}{\kappa K^\frac12}  \sum_{j=0}^{\nu^{-\frac12}} \frac{(j+1)^{\frac{\theta'}{2}-\frac14}}{(j!)^{{\frac32}} \nu^\frac{j}{2}} \nu^\frac14 (j+1)^\frac12 M_{2,j,Y^{\theta'}} [\omega].
 \end{align*}
Hence \eqref{est.cor.prop.Airy1.1} follows by taking $\kappa$ small enough so that $C\kappa\le \frac12$, and then by taking $K$ large enough so that $\frac{CC_0^*}{\kappa K} \le \frac12$.

To show \eqref{est.cor.prop.Airy1.2} it suffices to prove the embedding inequality
\begin{align}
\cA \nabla \phi \cA_{2,{\bf 1}}'  & \le \sum_{j=0}^{\nu^{-\frac12}} \frac{(j+1)^\frac14}{(j!)^{{\frac32}} \nu^\frac{j}{2}}  \nu^\frac14 (j+1)^\frac12 M_{2,j,1}[\nabla \phi] \nonumber \\
& \le C \sum_{j=0}^{\nu^{-\frac12}} \frac{(j+1)^\frac14}{(j!)^{{\frac32}} \nu^\frac{j}{2}}  \nu^\frac14 (j+1)^\frac12 M_{2,j,Y}[\omega]\label{proof.cor.prop.Airy1.2}
\end{align}
and the interpolation inequality
\begin{align}
\cA \pa_Y \phi|_{Y=0} \cA_{bc} & :=  \sum_{j=0}^{\nu^{-\frac12}} \frac{\nu^\frac14 (j+1)^\frac12}{(j!)^{\frac32} \nu^\frac{j}{2}}   \| e^{-K\tau \nu^\frac12 (j+1)} \pa_X^j \pa_Y \phi|_{Y=0} \|_{L^2(0,\frac{1}{K\nu^\frac12}; L^2_X)} \nonumber \\
& \le C\Big (\sum_{j=0}^{\nu^{-\frac12}} \frac{(j+1)^{-\frac14}}{(j!)^{\frac32} \nu^\frac{j}{2}} \nu^\frac14 (j+1)^\frac12 M_{2,j,1}[\omega]\Big )^\frac12 \Big (\sum_{j=0}^{\nu^{-\frac12}} \frac{(j+1)^{\frac14}}{(j!)^{{\frac32}} \nu^\frac{j}{2}} \nu^\frac14 (j+1)^\frac12 M_{2,j,Y}[\omega]\Big )^\frac12.\label{proof.cor.prop.Airy1.3}
\end{align}
Then \eqref{est.cor.prop.Airy1.2} follows from \eqref{est.cor.prop.Airy1.1} with \eqref{proof.cor.prop.Airy1.2} and \eqref{proof.cor.prop.Airy1.3}.
The proof of \eqref{proof.cor.prop.Airy1.2} proceeds as in the proof of Proposition \ref{prop.stream.mOS}.
Indeed, we have from $\omega^{\bf j} = -\nabla \cdot (\nabla \phi)^{\bf j}+\frac{\nu^\frac12 j_2\chi_\nu'}{\chi_\nu} (\pa_Y \phi)^{\bf j}$ and from the integration by parts,
\begin{align*}
\| (\nabla \phi)^{\bf j}\|^2 & = \langle \omega^{\bf j}, \phi^{\bf j}\rangle - 2\nu^\frac12 j_2e^{-K\tau \nu^\frac12} \langle \chi_\nu' (\pa_Y \phi)^{\bf j}, (\pa_Y \phi)^{(j-j_2,j_2-1)}\rangle\\
& \le \| Y\omega^{\bf j} \|\|\frac{\phi^{\bf j}}{Y} \| + 2\nu^\frac12 j_2 \kappa \| (\pa_Y\phi)^{\bf j}\| \| (\pa_Y \phi)^{(j-j_2,j_2-1)}\|\\
& \leq  C \| Y\omega^{\bf j}\|  \|\pa_Y \phi^{\bf j} \| + 2\nu^\frac12 j_2 \kappa \| (\pa_Y\phi)^{\bf j}\| \| (\pa_Y \phi)^{(j-j_2,j_2-1)}\|.
\end{align*}
Here the Hardy inequality ise used in the last line. Then the identity $\pa_Y \phi^{\bf j} = (\pa_Y \phi)^{\bf j} + \nu^\frac12 j_2 \chi_\nu' e^{-K\tau \nu^\frac12} (\pa_Y \phi)^{(j-j_2,j_2-1)}$ yields
\begin{align*}
\| (\nabla \phi)^{\bf j} \|\le C \Big ( \| Y\omega^{\bf j} \| + \nu^\frac12 j_2 \kappa \| (\pa_Y \phi)^{(j-j_2,j_2-1)} \| \Big ).
\end{align*}
This estimate gives 
\begin{align*}
& \sum_{j=0}^{\nu^{-\frac12}} \frac{(j+1)^\frac14}{(j!)^{{\frac32}} \nu^\frac{j}{2}}  \nu^\frac14 (j+1)^\frac12 M_{2,j,1}[\nabla \phi]\\
& \le C  \sum_{j=0}^{\nu^{-\frac12}} \frac{(j+1)^\frac14}{(j!)^{{\frac32}} \nu^\frac{j}{2}}  \nu^\frac14 (j+1)^\frac12  \big ( M_{2,j,Y}[\omega] + \nu^\frac12 j \kappa M_{2,j-1,1}[\nabla \phi]\big ) \\
& \le  C  \sum_{j=0}^{\nu^{-\frac12}} \frac{(j+1)^\frac14}{(j!)^{{\frac32}} \nu^\frac{j}{2}}    \nu^\frac14 (j+1)^\frac12 M_{2,j,Y}[\omega]  + C\kappa \sum_{j=0}^{\nu^{-\frac12}} \frac{(j+1)^\frac14}{(j!)^\gamma \nu^\frac{j}{2}}  \nu^\frac14 (j+1)^\frac12 M_{2,j,1}[\nabla \phi],
\end{align*}
where $C>0$ is a universal constant. This proves \eqref{proof.cor.prop.Airy1.2} if $\kappa$ is small enough so that $C\kappa\le \frac12$. As for \eqref{proof.cor.prop.Airy1.3}, we observe from \eqref{proof.prop.mOS1.3} that 
\begin{align*}
& \| e^{-K\tau \nu^\frac12 (j+1)} \pa_X^j \pa_Y \phi|_{Y=0} \|_{L^2(0,\frac{1}{K\nu^\frac12}; L^2_X)} \\
& \le C \Big ( (j+1)^{-\frac14} \| \omega^{(j,0)} \|_{L^2(0,\frac{1}{K\nu^\frac12}; L^2_X)}\Big )^\frac12 \Big ( (j+1)^\frac14 \| \pa_Y \phi^{(j,0)} \|_{L^2(0,\frac{1}{K\nu^\frac12}; L^2_X)}\Big )^\frac12,
\end{align*}
which implies from the Schwarz inequality,
\begin{align*}
\cA \pa_Y \phi|_{Y=0} \cA_{bc}\le C \Big (\sum_{j=0}^{\nu^{-\frac12}} \frac{(j+1)^{-\frac14}}{(j!)^\gamma \nu^\frac{j}{2}} \nu^\frac14 (j+1)^\frac12 M_{2,j,1}[\omega] \Big )^\frac12 \Big (\sum_{j=0}^{\nu^{-\frac12}} \frac{(j+1)^{\frac14}}{(j!)^\gamma \nu^\frac{j}{2}} \nu^\frac14 (j+1)^\frac12 M_{2,j,1}[\nabla \phi] \Big )^\frac12.
\end{align*} 
Then \eqref{proof.cor.prop.Airy1.2} shows \eqref{proof.cor.prop.Airy1.3}. The proof is complete.
\end{proof}

\

\begin{corollary}\label{cor.prop.Airy2} In Corollary \ref{cor.prop.Airy1}, let $H=V\cdot \nabla \Delta \phi_{1,1}[h]$, where $\phi_{1,1}[h]$ is the solution to \eqref{eq.s} in Propositions \ref{prop.Stokes.velocity}-\ref{prop.Stokes.vorticity}.  Then 
\begin{align}\label{est.cor.prop.Airy2.1}
\begin{split}
& \sum_{j=0}^{\nu^{-\frac12}} \frac{(j+1)^{\frac{\theta}{2}-\frac14}}{(j!)^\frac32 \nu^\frac{j}{2}} \Big ( \nu^\frac14 M_{2,j,Y^\theta} [\nabla \omega] + M_{\infty,j,Y^\theta}[\omega] + K^\frac12 \nu^\frac14 (j+1)^\frac12 M_{2,j,Y^\theta}[\omega]\Big )  \leq \frac{CC_\kappa}{K^{{\frac14}}} \cA h\cA_{bc},
\end{split}
\end{align}
and 
\begin{align}\label{est.cor.prop.Airy2.2}
\cA \nabla \phi \cA_{2,{\bf 1}}' + \cA \pa_Y \phi|_{Y=0} \cA_{bc} \le \frac{CC_\kappa}{K^{{\frac34}}} \cA h\cA_{bc}.
\end{align}
Moreover, we have 
\begin{align}
\sum_{j=0}^{\nu^{-\frac12}}\frac{1}{(j!)^\frac32 \nu^{\frac{j}{2}+\frac14} (j+1)^\frac12} M_{2,j,\frac{1}{1+Y}}[\pa_X \phi]  & \leq \frac{CC_\kappa}{K^{{\frac34}}} \cA h\cA_{bc}.\label{est.cor.prop.Airy2.3}
\end{align}
Here $C>0$ is a universal constant.
\end{corollary}

\begin{proof} To show \eqref{est.cor.prop.Airy2.1} and \eqref{est.cor.prop.Airy2.2} it suffices to prove for $\theta'=0,1,2$,
\begin{align}\label{proof.cor.prop.Airy2.1} 
\begin{split}
&  \sum_{j=0}^{\nu^{-\frac12}} \frac{1}{(j!)^\frac32 \nu^{\frac{j}{2}}(j+1)^{\frac{1-\theta'}{2}}} M_{2,j,Y^{\theta'+\frac12}}[H] \\
& \le C C_\kappa   \sum_{j=0}^{\nu^{-\frac12}} \frac{1}{(j!)^\frac32 \nu^{\frac{j}{2}}(j+1)^{\frac{1-\theta'}{2}}} \big ( M_{2,j,Y^{\frac32+\theta'}} [\pa_X \Delta \phi_{1,1}] + \nu^\frac12 M_{2,j,Y^{\frac32+\theta'}} [\pa_Y \Delta \phi_{1,1}]\big ).
\end{split}
 \end{align}
Then \eqref{est.cor.prop.Airy2.1} and \eqref{est.cor.prop.Airy2.2} follow from \eqref{est.cor.prop.Airy1.1}, \eqref{est.cor.prop.Airy1.2}, \eqref{est.prop.Stokes.vorticity.2}  and \eqref{proof.cor.prop.Airy2.1}.
To show \eqref{proof.cor.prop.Airy2.1} we observe that 
\begin{align*}
H^{\bf j} = \sum_{l=0}^j \sum_{\max\{0,l+j_2-j\}\le l_2\le \min\{l,j_2\}}  \binom{j_2}{l_2} \binom{j-j_2}{l-l_2}  V^{\bf j-l} \cdot (\nabla \Delta \phi_{1,1})^{\bf l}.
\end{align*} 
Thus we have 
\begin{align*}
\| Y^{\theta'+\frac12} H^{\bf j} \| 
& \le   \sum_{l=0}^j  \binom{j}{l} \, \sum_{\max\{0,l+j_2-j\}\le l_2\le \min\{l,j_2\}} \big ( \| \pa_Y V_1^{\bf j-l}\|_{L^\infty} \| Y^{\frac32+\theta'} (\pa_X \Delta\phi_{1,1})^{\bf l}\| \\
& \qquad \qquad \qquad \qquad \qquad \qquad \qquad  +\| \pa_Y V_2^{\bf j-l}\|_{L^\infty} \| Y^{\frac32+\theta'}  (\pa_Y \Delta\phi_{1,1})^{\bf l}\| \big ).
\end{align*}
Set 
\begin{align}\label{proof.cor.prop.Airy2.2} 
N_{\infty,j}[\nabla V_1] = (j+1)^\frac12 \sup_{j_2=0,\cdots,j} \Big ( \nu^{-\frac12} \| (\pa_X V_1)^{\bf j} \|_{L^\infty_{\tau,X,Y}} + \| (\pa_Y V_1)^{\bf j}\|_{L^\infty_{\tau,X,Y}} \Big ).
\end{align}
Since 
\begin{align*}
\| \pa_Y V_1^{\bf j-l}\|_{L^\infty} & \le \| (\pa_Y V_1)^{\bf j-l}\|_{L^\infty} + \kappa \nu^\frac12 (j_2-l_2) \| (\pa_Y V_1)^{(j_1-l_1,j_2-l_2-1)} \|_{L^\infty}\\
& \le (j-l+1)^{-\frac12} N_{\infty,j-l}[\nabla V_1] + \kappa \nu^\frac12 (j-l)^\frac12 N_{\infty,j-l-1}[\nabla V_1]
\end{align*}
and similarly
\begin{align*}
\| \pa_Y V_2^{\bf j-l}\|_{L^\infty} & \le \| (\pa_Y V_2)^{\bf j-l}\|_{L^\infty} + \kappa \nu^\frac12 (j_2-l_2) \| (\pa_Y V_2)^{(j_1-l_1,j_2-l_2-1)} \|_{L^\infty}\\
& = \| (\pa_X V_1)^{\bf j-l}\|_{L^\infty} + \kappa \nu^\frac12 (j_2-l_2) \| (\pa_X V_1)^{(j_1-l_1,j_2-l_2-1)} \|_{L^\infty}\\
& \le \nu^\frac12 \Big ( (j-l+1)^{-\frac12} N_{\infty,j-l}[\nabla V_1] + \kappa \nu^\frac12 (j-l)^\frac12 N_{\infty,j-l-1}[\nabla V_1]\Big ),
\end{align*}
we obtain
\begin{align*}
M_{2,j,Y^{\theta'+\frac12}}[H] 
& \le {\sum_{l=0}^j} \binom{j}{l} \, \min\{l+1,j-l+1\} \\
& \quad \times \big ( (j-l+1)^{-\frac12} N_{\infty,j-l}[\nabla V_1] + \kappa \nu^\frac12 (j-l)^\frac12  N_{\infty,j-l-1}[\nabla V_1]\big ) \\
& \qquad \times \big (  M_{2,l,Y^{\frac32+\theta'}}[\pa_X \Delta \phi_{1,1}] + \nu^\frac12  M_{2,l,Y^{\frac32+\theta'}} [\pa_Y \Delta \phi_{1,1}]\big ) 
\end{align*}
Then \eqref{proof.cor.prop.Airy2.1} follows from the Young inequality for convolution in the $l^1$ space.
For example, we have, by using $\frac{(l+1)^{\frac{1-\theta'}{2}}}{(j+1)^\frac{1-\theta'}{2} (j-l+1)^\frac12} \leq C$ for $\theta'=0,1,2$ and $ (\frac{(j-l)! l!}{j!})^\frac12 \min\{l+1,j-l+1\}\le C$,
\begin{align*}
& \sum_{j=0}^{\nu^{-\frac12}} \sum_{l=0}^j \frac{1}{(j+1)^{\frac{1-\theta'}{2}}} (\frac{(j-l)! l!}{j!})^\frac12 \min\{l+1,j-l+1\} (j-l+1)^{-\frac12} (l+1)^{\frac{1-\theta'}{2}} \\
& \quad \times \Big ( \frac{1}{((j-l)!)^\frac32 \nu^\frac{j-l}{2}} N_{\infty,j-l}[\nabla V_1] \Big ) \Big ( \frac{1}{(l!)^\frac32 \nu^\frac{l}{2}(l+1)^{\frac{1-\theta'}{2}}} M_{2,l,Y^{\frac32+\theta'}}[\pa_X\Delta \phi_{1,1}]\Big ) \\
& \le C \sum_{j=0}^{\nu^{-\frac12}} \sum_{l=0}^j \Big ( \frac{1}{((j-l)!)^\frac32 \nu^\frac{j-l}{2}} N_{\infty,j-l}[\nabla V_1] \Big ) \Big ( \frac{1}{(l!)^\frac32 \nu^\frac{l}{2}(l+1)^{\frac{1-\theta'}{2}}} M_{2,l,Y^{\frac32+\theta'}}[\pa_X\Delta \phi_{1,1}]\Big ) \\
& \leq C C_\kappa   \sum_{j=0}^{\nu^{-\frac12}} \frac{1}{(j!)^\frac32 \nu^{\frac{j}{2}}(j+1)^{\frac{1-\theta'}{2}}}  M_{2,j,Y^{\frac32+\theta'}} [\pa_X \Delta \phi_{1,1}].
\end{align*}
The other terms are handled in the same manner and we omit the details.
The proof of \eqref{est.cor.prop.Airy2.1}-\eqref{est.cor.prop.Airy2.2} is complete.
Finally let us prove \eqref{est.cor.prop.Airy2.3}.
The key is to apply the interpolation-type inequality proved in Proposition \ref{prop.stream.a}.
Indeed, Proposition \ref{prop.stream.a} implies 
\begin{align*}
& \sum_{j=0}^{\nu^{-\frac12}}\frac{1}{(j!)^\frac32 \nu^{\frac{j}{2}+\frac14} (j+1)^\frac12} M_{2,j,\frac{1}{1+Y}}[\pa_X \phi]  \\
& \leq C \sum_{\theta=0}^1 \sum_{j=0}^{\nu^{-\frac12}-1} \frac{1}{(j!)^\frac32 \nu^{\frac{j}{2}+\frac14} (j+1)^\frac12} (j+1)^{\frac{\theta}{2}-\frac14} M_{2,j+1,Y^{1+{\theta}}} [\omega] \\
& \quad +  C  \sum_{j=0}^{\nu^{-\frac12}-1} \frac{1}{(j!)^\frac32 \nu^{\frac{j}{2}+\frac14} (j+1)^\frac12} \kappa \nu^\frac12 j \big ( M_{2,j-1,Y}[\omega] + M_{2,j-1,1}[\nabla \phi] \big )\\
& \qquad + \frac{1}{(j!)^\frac32 \nu^{\frac{j}{2}+\frac14} (j+1)^\frac12} M_{2,j,\frac{1}{1+Y}}[\pa_X \phi] \Big |_{j=\nu^{-\frac12}} \\
& \le C \sum_{\theta=0}^1 \sum_{j=0}^{\nu^{-\frac12}} \frac{\nu^\frac14(j+1)^{\frac{{\theta}}{2}+\frac34}}{(j!)^\frac32 \nu^{\frac{j}{2}}} M_{2,j,Y^{1+{\theta}}} [\omega] \\
& \quad +  C  \sum_{j=0}^{\nu^{-\frac12}} \frac{\nu^\frac14 (j+1)^\frac12}{(j!)^\frac32 \nu^{\frac{j}{2}}}  M_{2,j,Y}[\omega]  + C \cA \nabla \phi \cA_{2,{\bf 1}}'\\
& \le \frac{CC_\kappa}{{K^{\frac34}}} \cA h\cA_{bc}.
\end{align*}
Here we have used \eqref{est.cor.prop.Airy2.1} and \eqref{est.cor.prop.Airy2.2} in the last line. The proof is complete.
\end{proof}

\subsection{Full construction of boundary corrector}\label{subsec.Boundary}

We set $\phi_{app,1}=\phi_{app,1}[h]=\phi_{1,1}[h] + \phi_{1,2}[h]$, where $\phi_{1,1}[h]$ is the solution to \eqref{eq.s} in Propositions \ref{prop.Stokes.velocity}-\ref{prop.Stokes.vorticity}, and $\phi_{1,2}[h]$ is the solution to \eqref{eq.airy} with $H=V \cdot \nabla \Delta\phi_{1,1}[h]$ as in Corollary \ref{cor.prop.Airy2}. 
Then the approximate solution $\phi_{app}$ to the full system \eqref{eq.bc} is constructed in the form $\phi_{app}=\phi_{app,1} + \tilde \phi_1$, which leads to the equations for $\tilde \phi_1=\tilde \phi_1[h]$ as 
\begin{align}\label{eq.tilde}
\begin{split}
&\nu^\frac12 \Delta^2 \tilde \phi_1 - \pa_\tau \Delta \tilde \phi_1 - V \cdot \nabla \Delta \tilde \phi_1 + \nabla^\bot \tilde \phi_1 \cdot \nabla \Omega = -\nabla^\bot \phi_{app,1} \cdot \nabla \Omega,\qquad \tau>0,~X\in \T_\nu, ~Y>0,\\
& \tilde \phi_1|_{Y=0} =\Delta \tilde \phi_1|_{Y=0}=0,\qquad \tilde \phi_1|_{\tau=0}=0. 
\end{split} 
\end{align}

Let us first give the estimate for the force term $ -\nabla^\bot \phi_{app,1} \cdot \nabla \Omega$.
\begin{proposition}\label{prop.app} Let $\kappa_3\in (0,1]$ be the number in Corollary \ref{cor.prop.Airy1}.
For any $\kappa\in (0,\kappa_3]$ there exists $K_3'=K_3'(\kappa, C_*,C_j^*)\geq 1$ such that for any $K\ge K_3'$,
\begin{align}\label{est.prop.app.1}
\begin{split}
& \frac{1}{K^\frac12 \nu^\frac12} \cA \nabla^\bot \phi_{app,1}\cdot \nabla \Omega \cA_{2,\tilde \xi^{(2)}}' + \frac{1}{K^\frac12 \nu^\frac14} \| \nabla^\bot \phi_{app,1}\cdot \nabla \Omega \|_{L^2(0,\frac{1}{K\nu^\frac12}; \dot{H}^{-1})} \\
& \leq  \frac{1}{K^{\frac14}} \Big (\sum_{j=0}^{\nu^{-\frac12}} \frac{1}{(j!)^\frac32 \nu^{\frac{j}{2}+\frac14}(j+1)^\frac12} M_{2,j,\frac{1}{1+Y}} [\pa_X \phi_{app,1}] + 2 \cA \nabla \phi_{app,1}\cA_{2,{\bf 1}}'\Big ).
\end{split}
\end{align}
\end{proposition}

\begin{proof} Let us recall that 
\begin{align*}
& \frac{1}{\nu^\frac12} \cA \nabla^\bot \phi_{app,1}\cdot \nabla \Omega \cA_{2,\tilde \xi^{(2)}}' \\
& = \sum_{j=0}^{\nu^{-\frac12}} \frac{1}{(j!)^\frac32 \nu^\frac{j}{2} \nu^\frac14 (j+1)^\frac12} \sup_{j_2=0,\cdots,j} \| \xi_j e^{-K\tau\nu^\frac12 (j+1)} B_{j_2}\pa_X^{j-j_2} (\nabla^\bot \phi_{app,1}\cdot \nabla \Omega) \|_{L^2(0,\frac{1}{K\nu^\frac12}; L^2_{X,Y})}.
\end{align*} 
Thus we consider the estimate of 
\begin{align*}
& e^{-K\tau\nu^\frac12 (j+1)} B_{j_2}\pa_X^{j-j_2} (\nabla^\bot \phi_{app,1}\cdot \nabla \Omega) \\
& = (\nabla^\bot \phi_{app,1})^{\bf j} \cdot \nabla \Omega + \sum_{l=0}^{j-1} \sum_{\max\{0,l+j_2-j\} \le l_2\le \min\{l,j_2\}} \binom{j_2}{l_2}  \binom{j-j_2}{l-l_2}  (\nabla^\bot \phi_{app,1})^{\bf l} \cdot (\nabla \Omega)^{{\bf j-l}},
\end{align*}
where ${\bf j}=(j-j_2,j_2)$ and ${\bf l}=(l-l_2,l_2)$.
We observe that, from the definition of $\rho_j$ in \eqref{def.rhoj},  {point iii) in Assumption \ref{assume}}  and $K\geq 1$,
\begin{align*}
\|\xi_j \pa_X \phi_{app,1}^{\bf j} \pa_Y \Omega\|  & = \|\frac{\pa_Y \Omega}{\sqrt{\pa_Y \Omega + 2\rho_j}} \pa_X \phi_{app,1}^{\bf j} \| \\
& \leq C \|(|\pa_Y \Omega|^\frac12 + \sqrt{\rho_j}) \pa_X \phi_{app,1}^{\bf j} \|\\
& \le C \| (\frac{1+Y}{1+\nu^\frac12 Y})^2 \pa_Y \Omega\|_{L^\infty}^\frac12 \| \frac{1+\nu^\frac12 Y}{1+Y} \pa_X\phi_{app,1}^{\bf j} \| \\
& \quad + C (K^\frac14 C_*)^\frac12 \| \frac{1}{1+Y} \pa_X \phi_{app,1}^{\bf j} \| + C C_*^\frac12 \nu^\frac12 \| \pa_X \phi_{app,1}^{\bf j} \|\\
& \le C(C_1^* + K^\frac14 C_*)^\frac12  \| \frac{1}{1+Y} \pa_X \phi_{app,1}^{\bf j}\| + C (C_1^*+C_*)^\frac12 \nu^\frac12 \| \pa_X \phi_{app,1}^{\bf j} \|. 
\end{align*}
Here we have
On the other hand,
\begin{align*}
\|\xi_j (\pa_Y \phi_{app,1})^{\bf j} \pa_X \Omega\| & \le \| \frac{1+Y}{1+\nu^\frac12 Y} \pa_X \Omega \|_{L^\infty} \| \frac{1+\nu^\frac12 Y}{1+Y} \xi_j \|_{L^\infty} \| (\pa_Y \phi_{app,1})^{\bf j} \|\\
& \le C C_1^* \nu^\frac12 (j+1)^\frac12 \| (\pa_Y \phi_{app,1})^{\bf j}\|.
\end{align*}
Here we have used \eqref{est.lem.rhoj.2} {and point iii) in Assumption \ref{assume}}. Thus we have from $C_*\geq 1$,
\begin{align}\label{proof.prop.app.1}
\begin{split}
& \| \xi_j (\nabla^\bot \phi_{app,1})^{\bf j} \cdot \nabla \Omega \|_{L^2(0,\frac{1}{K\nu^\frac12}; L^2_{X,Y})}\\
& \le C(C_1^* + K^\frac14 C_*)^\frac12 M_{2,j,\frac{1}{1+Y}} [\pa_X \phi_{app,1}] + C (C_1^* + C_*) \nu^\frac12  (j+1)^\frac12 M_{2,j,1}[\pa_Y \phi_{app,1}].
\end{split}
\end{align} 
Next we see 
\begin{align*}
& \| \xi_j \sum_{l=0}^{j-1} \sum_{\max\{0,l+j_2-j\} \le l_2\le \min\{l,j_2\}} \binom{j_2}{l_2}  \binom{j-j_2}{l-l_2}  (\nabla^\bot \phi_{app,1})^{\bf l} \cdot (\nabla \Omega)^{{\bf j-l}}\| \\
& \le \sum_{l=0}^{j-1}  \binom{j}{l}   \sum_{\max\{0,l+j_2-j\} \le l_2\le \min\{l,j_2\}} \|\xi_j  (\nabla^\bot \phi_{app,1})^{\bf l} \cdot (\nabla \Omega)^{{\bf j-l}}\|,
\end{align*}
and 
\begin{align*}
 \|\xi_j  (\nabla^\bot \phi_{app,1})^{\bf l} \cdot (\nabla \Omega)^{{\bf j-l}}\| & \le \| (\frac{1+Y}{1+\nu^\frac12 Y})^2 (\pa_Y \Omega)^{\bf j-l} \|_{L^\infty} \| \frac{1+\nu^\frac12 Y}{1+Y} \xi_j \|_{L^\infty} \| \frac{1+\nu^\frac12 Y}{1+Y} \pa_X \phi_{app,1}^{\bf l} \| \\
 & \quad +  \| \frac{1+Y}{1+\nu^\frac12 Y} (\pa_X \Omega)^{\bf j-l} \|_{L^\infty} \| \frac{1+\nu^\frac12 Y}{1+Y} \xi_j \|_{L^\infty} \|  (\pa_Y \phi_{app,1})^{\bf l} \|\\
 & \le C  (j+1)^\frac12 N_{\infty, j-l, (\frac{1+Y}{1+\nu^\frac12 Y})^2}[\pa_Y \Omega]  \, \| \frac{1}{1+Y} \pa_X \phi_{app,1}^{\bf l} \| \\
 & \quad + C  \nu^\frac12 (j+1)^\frac12 N_{\infty,j-l} [\nabla \Omega] \, \| (\nabla \phi_{app,1})^{\bf l} \|.
\end{align*}
Thus we have
\begin{align}\label{proof.prop.app.2}
\begin{split}
& \| \xi_j \sum_{l=0}^{j-1} \sum_{\max\{0,l+j_2-j\} \le l_2\le \min\{l,j_2\}} \binom{j_2}{l_2} \binom{j-j_2}{l-l_2}  (\nabla^\bot \phi_{app,1})^{\bf l} \cdot (\nabla \Omega)^{{\bf j-l}}\|_{L^2(0,\frac{1}{K\nu^\frac12}; L^2_{X,Y})} \\
& \le C (j+1)^\frac12 \sum_{l=0}^{j-1} \min \{l+1,j-l+1\} \binom{j}{l}   \, N_{\infty,j-l} [\nabla \Omega] \\
& \qquad \qquad \times \Big ( M_{2,l,\frac{1}{1+Y}}[\pa_X \phi_{app,1}] + \nu^\frac12 M_{2,l,1} [\nabla \phi_{app,1}] \Big ).
\end{split}
\end{align}
We note that 
\begin{align*}
(j+1)^\frac12  \min \{l+1,j-l+1\}  \, (\frac{(j-l)! l!}{j!})^\frac12 \leq C\,, \qquad 1\le l\le j-1.
\end{align*}
Taking into account this uniform bound (by decomposing the sum $\sum_{l=0}^{j-1}$ into the term of $''l=0''$ and $\sum_{l=1}^{j-1}$) and collecting \eqref{proof.prop.app.1} and \eqref{proof.prop.app.2}, we obtain from the Young inequality for convolution in the $l^1$ space, 
\begin{align}
\begin{split}
& \frac{1}{K^\frac12\nu^\frac12} \cA \nabla^\bot \phi_{app,1}\cdot \nabla \Omega \cA_{2,\tilde \xi^{(2)}}' \\
& \le \frac{1}{K^{\frac14}} \Big (\sum_{j=0}^{\nu^{-\frac12}} \frac{1}{(j!)^\frac32 \nu^{\frac{j}{2}+\frac14}(j+1)^\frac12} M_{2,j,\frac{1}{1+Y}} [\pa_X \phi_{app,1}] + \cA \nabla \phi_{app,1}\cA_{2,{\bf 1}}'\Big ),
\end{split}
\end{align}
where $K$ has been taken large enough depending on $C_*$, $C_1^*$, and $C_\kappa$.
As for the estimate of $\| \nabla^\bot \phi_{app,1}\cdot \nabla \Omega \|_{L^2(0,\frac{1}{K\nu^\frac12}; \dot{H}^{-1})}$, let us take any $\eta\in \dot{H}^1_0(\T\times \R_+)$. Then we have 
\begin{align*}
\langle  \nabla^\bot \phi_{app,1}\cdot \nabla \Omega, \eta\rangle 
& = \langle \frac{1+Y}{1+\nu^\frac12 Y}  \nabla^\bot \phi_{app,1}\cdot \nabla \Omega, \frac{\eta}{1+Y}\rangle + \langle  \nabla^\bot \phi_{app,1}\cdot \nabla \Omega, \frac{\nu^\frac12 Y \eta }{1+\nu^\frac12 Y}\rangle\\
& = \langle \frac{1+Y}{1+\nu^\frac12 Y}  \nabla^\bot \phi_{app,1}\cdot \nabla \Omega, \frac{\eta}{1+Y}\rangle -  \langle  \Omega, \nabla^\bot \phi_{app,1}\cdot \nabla \big (  \frac{\nu^\frac12 Y\eta }{1+\nu^\frac12 Y} \big )\rangle.
\end{align*}
This implies 
\begin{align*}
|\langle  \nabla^\bot \phi_{app,1}\cdot \nabla \Omega, \eta\rangle|
& \le \|\frac{1+Y}{1+\nu^\frac12 Y}  \nabla^\bot \phi_{app,1}\cdot \nabla \Omega \| \, \| \frac{\eta}{1+Y} \| \\
& \quad + \| \frac{1+Y}{1+\nu^\frac12 Y} \Omega  \nabla^\bot \phi_{app,1}\| \, \| \frac{1+\nu^\frac12 Y}{1+Y} \nabla  \big (  \frac{\nu^\frac12  Y \eta }{1+\nu^\frac12 Y} \big )\|\\
& \le C \|\frac{1+Y}{1+\nu^\frac12 Y}  \nabla^\bot \phi_{app,1}\cdot \nabla \Omega \|  \, \| \pa_Y \eta \| + C\nu^\frac12 \| \frac{1+Y}{1+\nu^\frac12 Y} \Omega  \nabla^\bot \phi_{app,1}\| \, \| \nabla \eta \|,
\end{align*}
where the Hardy inequality was used several times.
Hence we obtain 
\begin{align*}
\| \nabla^\bot \phi_{app,1}\cdot \nabla \Omega \|_{\dot{H}^{-1}} & \le C \|\frac{1+Y}{1+\nu^\frac12 Y}  \nabla^\bot \phi_{app,1}\cdot \nabla \Omega \|  + C\nu^\frac12 \| \frac{1+Y}{1+\nu^\frac12 Y} \Omega  \nabla^\bot \phi_{app,1}\|\\
& \le C  \| \frac{1+Y}{1+\nu^\frac12 Y} \pa_X \Omega\|_{L^\infty}\, \| \pa_Y \phi_{app,1}\| \\
& \quad + C \| (\frac{1+Y}{1+\nu^\frac12 Y})^2 \pa_Y \Omega \|_{L^\infty} \|\frac{1+\nu^\frac12 Y}{1+Y}  \pa_X \phi_{app,1}\|\\
& \qquad  + C \nu^\frac12  \| \frac{1+Y}{1+\nu^\frac12 Y} \Omega\|_{L^\infty}  \| \nabla \phi_{app,1}\| \\
& \le C C_1^* \big ( \nu^\frac12 \| \nabla \phi_{app,1} \| + \| \pa_Y \pa_X \phi_{app,1} \|\big ).
\end{align*}
Then 
\begin{align*}
& \frac{1}{K^\frac12\nu^\frac14} \| \nabla^\bot \phi_{app,1}\cdot \nabla \Omega \|_{L^2(0,\frac{1}{K\nu^\frac12}; \dot{H}^{-1})} \\
& \le \frac{CC_1^*}{K^\frac12\nu^\frac14} \big ( \nu^\frac12 \| \nabla \phi_{app,1} \|_{L^2(0,\frac{1}{K\nu^\frac12}; L^2_{X,Y})} + \| \pa_X \pa_Y \phi_{app,1}\| _{L^2(0,\frac{1}{K\nu^\frac12}; L^2_{X,Y})}\big )\\
& \le \frac{1}{K^{\frac14}} \cA \nabla \phi_{app,1}\cA_{2,{\bf 1}}'.
\end{align*}
The proof is complete.
\end{proof}

\

Propositions \ref{prop.mOS1} and \ref{prop.app} yield
\begin{corollary}\label{cor.prop.app}  There exists $\kappa_4\in (0,1]$ such that the following statement holds for any $\kappa\in (0,\kappa_4]$.  There exists $K_4=K_4(\kappa, C_*, C_j^*)\geq 1$ such that if $K\geq K_4$ then the system \eqref{eq.tilde} admits a unique solution $\tilde \phi_1\in C([0,\infty); \dot{H}^1_0(\T_\nu \times \R_+))$ with $\tilde \omega_1=-\Delta \phi\in C([0,\infty); L^2(\T_\nu \times \R_+))$ satisfying
\begin{align}
\begin{split}
& \cA \tilde \omega_1 \cA_{\infty,\xi}'   + K^\frac12 \cA \tilde \omega_1 \cA_{2,\xi}' + K^\frac14 \cA \nabla \tilde \phi_1 \cA_{2,{\bf 1}}' + K^\frac14 \cA \pa_Y \tilde \phi_1|_{Y=0} \cA_{bc}   \leq \frac{1}{K^\frac12} \cA h\cA_{bc}.
\end{split}
\end{align}
\end{corollary}

\begin{proof} Propositions \ref{prop.mOS1} and \ref{prop.app} give
\begin{align*}
& \cA \omega \cA_{\infty,\xi}'   + K^\frac12 \cA \omega \cA_{2,\xi}' + K^\frac14 \cA \nabla \phi \cA_{2,{\bf 1}}' + K^\frac14  \cA \pa_Y \phi|_{Y=0} \cA_{bc} \\
& \leq \frac{C}{K^\frac14} \Big (\sum_{j=0}^{\nu^{-\frac12}} \frac{1}{(j!)^\frac32 \nu^{\frac{j}{2}+\frac14}(j+1)^\frac12} M_{2,j,\frac{1}{1+Y}} [\pa_X \phi_{app,1}]  +  \cA \nabla \phi_{app,1} \cA_{2,{\bf 1}}'\Big ).
\end{align*}
Here $C>0$ is a universal constant. Recall that $\phi_{app,1}[h]=\phi_{1,1}[h]+\phi_{1,2}[h]$.
Then the assertion follows from Proposition \ref{prop.Stokes.velocity} for $\phi_{1,1}[h]$ and Corollary \ref{cor.prop.Airy2} for $\phi_{1,2}[h]$.
The proof is complete.
\end{proof}

\

From the construction,  $\phi_{app}=\phi_{app}[h]=\phi_{app,1}[h] + \tilde \phi_1[h]$ satisfies 
\begin{align}
\begin{split}
&\nu^\frac12 \Delta^2  \phi_{app} - \pa_\tau \Delta  \phi_{app} - V \cdot \nabla \Delta \phi_{app} + \nabla^\bot  \phi_{app} \cdot \nabla \Omega =0,\qquad \tau>0,~X\in \T_\nu, ~Y>0,\\
& \phi_{app}|_{Y=0} =0,\qquad \pa_Y \phi_{app}|_{Y=0}=h + R_{bc}[h],\qquad \phi_{app}|_{\tau=0}=0. 
\end{split} 
\end{align}
Here $R_{bc}[h]$ is the linear operator defined as 
\begin{align}\label{def.R_bc}
R_{bc}[h] = \pa_Y\phi_{1,2}[h] \big |_{Y=0} + \pa_Y \tilde \phi_1[h] \big |_{Y=0}.
\end{align}
We note that the operator $R_{bc}$ is well-defined on the Banach space 
\begin{align}
Z_{bc} = \{h\in L^2(0,\frac{1}{K\nu^\frac12}; L^2_X)~|~ \|h\|_{Z_{bc}}:= \cA h\cA_{bc} <\infty \}.
\end{align}

\begin{proposition}\label{prop.boundary.correct} There exists $\kappa_5\in (0,1]$ such that the following statement holds for any $\kappa\in (0,\kappa_5]$.  There exists $K_5=K_5(\kappa, C_*, C_j^*)\geq 1$ such that if $K\geq K_5$ then the map $R_{bc}: Z_{bc}\rightarrow Z_{bc}$ defined by \eqref{def.R_bc} satisfies
\begin{align}\label{est.prop.boundary.correct.1} 
\cA R_{bc}[h] \cA_{bc} \leq \frac12 \cA h\cA_{bc}.
\end{align} 
Hence, the operator $I+R_{bc}$ is invertible in $Z_{bc}$, and the map
\begin{align}\label{est.prop.boundary.correct.2} 
\Phi_{bc} [h] := \phi_{app} [(I+R_{bc})^{-1} h], \qquad h\in Z_{bc}
\end{align}
gives the solution to \eqref{eq.bc} and satisfies 
\begin{align}\label{est.prop.boundary.correct.3} 
\cA \nabla \Phi_{bc} [h]\cA_{2,{\bf 1}}'\leq C \cA h\cA_{bc}.
\end{align}
Here $C>0$ is a universal constant.
\end{proposition}

\begin{proof} By the definition of $R_{bc}$ in \eqref{def.R_bc},  estimate \eqref{est.prop.boundary.correct.1} is a consequence of Corollaries \ref{cor.prop.Airy2} and \ref{cor.prop.app}, by taking $\kappa$ small first and then $K$ large enough depending only on $C_*$, $C_j^*$, and $C_\kappa$. In parituclar, we have 
\begin{align}
\cA (I+R_{bc})^{-1} h \cA_{bc}\le 2 \cA h\cA_{bc}, \qquad h\in Z_{bc}.
\end{align}
Then Proposition \ref{prop.Stokes.velocity} and Corollaries \ref{cor.prop.Airy2}-\ref{cor.prop.app} give \eqref{est.prop.boundary.correct.3}. The proof is complete.
\end{proof}

\section{Full estimate for linearization}\label{sec.linear}

We have constructed the solution to \eqref{PS} of the form 
\begin{align}\label{sol.construct}
\nabla^\bot \phi = \nabla^\bot \Phi_{slip} + \nabla^\bot \Phi_{bc}[h],\qquad h=-\pa_Y \Phi_{slip}|_{Y=0}\in Z_{bc},
\end{align}
and 
$$\Phi_{bc}[h]=\phi_{app,1}[(I+R_{bc})^{-1}h] + \tilde \phi_1[ (I+R_{bc})^{-1}h], \qquad \phi_{app,1}=\phi_{1,1}+\phi_{1,2}.$$
To simplify the notation we will write $\phi_{app,1}$ for $\phi_{app,1}[(I+R_{bc})^{-1}h]$ below.
So far we have the bound of $\nabla^\bot \phi_{1,1}$ only in the norm $\cA \cdot \cA_{2,{\bf 1}}'$. 
To obtain the estimates of $\cA\nabla \phi\cA_\infty$ and $\cA\Delta \phi \cA_\infty$ we need the extra work.

\begin{proposition}\label{prop.extra} There exists $\kappa_6\in (0,1]$ such that the following statement holds for any $\kappa\in (0,\kappa_6]$.  There exists $K_6=K_6(C_0^*, C_1^*)\geq 1$ such that if $K\geq K_6$ then the solution to \eqref{PS} constructed as \eqref{sol.construct} satisfies
\begin{align*}
& \nu^\frac14 \cA \omega \cA_\infty  + K^\frac12 \nu^\frac14 \cA \nabla \phi \cA_\infty  \\
& \leq \frac{C(C_0^*+C_1^*)}{\nu^\frac14} \big ( \cA \nabla \phi \cA_{2,{\bf 1}}' + \cA \Delta (\phi-\phi_{app,1}) \cA_{2,{\bf 1}}'  + \cA \Delta \phi_{app,1} \cA_{2,Y}' \big ) \\
& \quad + C\big( K^\frac12 [\|\nabla \phi_0 \|] + \nu^\frac14 [\|\Delta \phi_0\|] + \cA F\cA_2 \big ).
\end{align*}
Here $C>0$ is a universal constant.
\end{proposition}

The proof of Proposition \ref{prop.extra} is similar to the one of Proposition \ref{prop.mOS2},
and we postpone it to the appendix.
Admitting Proposition \ref{prop.extra}, we will now complete the proof of Theorem \ref{thm.main1}.
Let us recall \eqref{sol.construct}.
We first observe from Proposition \ref{prop.mOS1} and Remark \ref{rem.prop.mOS1} that 
\begin{align}\label{proof.thm.main1.1}
\cA \Delta \Phi_{slip} \cA_{2,{\bf 1}}' +  \cA \nabla \Phi_{slip} \cA_{2,{\bf 1}}'  + \cA \pa_Y \Phi_{slip} |_{Y=0} \cA_{bc} \le  \frac{1}{K^\frac18} \big ( \| \nabla \phi_0 \|_{L^2_{X,Y}} + \nu^{-\frac12} [\| \Delta \phi_0 \|]    + \nu^{-\frac34} \cA F\cA_2\big ),
\end{align}
by taking $K$ large enough.
On the other hand, Proposition \ref{prop.boundary.correct} (for $\nabla \Phi_{bc}$), Corollary  \ref{cor.prop.app} and (1) of Remark \ref{rem.prop.mOS1} (for $\Delta (\Phi_{bc}-\phi_{app,1})=\Delta \tilde \phi_1$), Proposition \ref{prop.Stokes.vorticity} and Corollary \ref{cor.prop.Airy2} and (for $\Delta \phi_{app,1}=\Delta \phi_{1,1}+\Delta \phi_{1,2}$), and \eqref{proof.thm.main1.1} give
\begin{align}
&  \cA \nabla \Phi_{bc} \cA_{2,1}'    + \cA \Delta (\Phi_{bc} - \phi_{app,1}) \cA_{2,{\bf 1}}'   + \cA \Delta \phi_{app,1} \cA_{2,Y}' \nonumber \\
& \le C \cA \pa_Y \Phi_{slip}|_{Y=0} \cA_{bc} \nonumber \\
& \le  \frac{C}{K^\frac18} \big (\| \nabla \phi_0 \|_{L^2_{X,Y}} + \nu^{-\frac12} [\| \Delta \phi_0 \|]    + \nu^{-\frac34} \cA F\cA_2\big ).
\end{align}
Here $C>0$ is a universal constant. By applying the estimate in Proposition \ref{prop.extra} and by taking $K$ large enough, the proof of Theorem \ref{thm.main1} is complete.

\section{Nonlinear stability: Proof of Theorem \ref{thm.main.nl}}\label{sec.nl}

Let us recall the nonlinear system \eqref{pns}. 
Theorem \ref{thm.main.nl} is a consequence of Theorem \ref{main.thm2} for the linear system \eqref{ps} and the bilinear estimate in Lemma \ref{lem.ap.bilinear} stated below.
We observe that $-w\cdot \nabla w= w\, {\rm rot}\, w + \nabla \tilde q$ for any solenoidal vector field $w$,
so the bilinear term we consider here is of the form $f{\rm rot}\, g$.
To this end we fix $K\geq 1$ and $\nu\in (0,1]$, and let $X$ be the Banach space of solenoidal vector fields $f=(f_1,f_2)$ on $[0,\frac1K]\times \R^2_+$ defined as
\begin{align*}
X = \{ f\in C([0,\frac{1}{K}]; H^1_{0,\sigma} (\T\times \R_+))~|~\| f\|_X =\| f\|_\infty + \nu^\frac12 \| {\rm rot}\, f \|_\infty <\infty\},
\end{align*}
where $\|\cdot \|_\infty$ is defined as \eqref{def.norm} with $p=\infty$.
\begin{lemma}\label{lem.ap.bilinear} There exists a universal constant $C>0$ such that for any $f,g\in X$,
\begin{align}
\| f \, {\rm rot}\, g \|_2\le \frac{C}{K^\frac12} \nu^{-\frac34} \| f\|_X \| g\|_X.
\end{align}
\end{lemma}

\begin{proof} We compute 
\begin{align*}
\|f  {\rm rot}\, g\|_2 & \le C \sum_{j=0}^{\nu^{-\frac12}} \frac{1}{j!^{\frac32}} \sup_{|\mathbf{j}| = j} \sum_{\mathbf{l} \le \mathbf{j}} \, \binom{\mathbf{j}}{\mathbf{l}} \| f^\mathbf{l} 
({\rm rot}\, g)^{\mathbf{j}-\mathbf{l}} \|_{L^2(0,\frac{1}{K}; L^2_{x,y})} \\
& \le \frac{C}{K^{\frac12}} \sum_{j=0}^{\nu^{-\frac12}} \frac{1}{j!^{\frac32}} \sup_{|\mathbf{j}| = j} \sum_{\mathbf{l} \le \mathbf{j}} \, \binom{\mathbf{j}}{\mathbf{l}} \| f^\mathbf{l} 
({\rm rot}\, g)^{\mathbf{j}-\mathbf{l}} \|_{L^\infty (0,\frac{1}{K}; L^2_{x,y})}  
\end{align*}
As $\binom{\mathbf{j}}{\mathbf{l}} \le \binom{|\mathbf{j}|}{|\mathbf{l}|}$ and as for all $l \in \N_0$,
\begin{align*}
& \sharp \, \{\mathbf{l}, \: |\mathbf{l}| = l, \:\mathbf{l} \le \mathbf{j} \}  \: = \:  \sharp \,   \{l_2, \:  \max(0, l-j+j_2) \le l_2 \le \min (j_2,l)  \} \le \min (l+1, j-l+1) 
\end{align*}  
we end up with 
\begin{align*}
& \|f  {\rm rot}\,  g\|_2  \le  \frac{C}{K^{\frac12}} \sum_{j=0}^{\nu^{-\frac12}} \frac{1}{j!^{\frac32}} \sum_{l=0}^{j} \min (l+1, j-l+1) \binom{j}{l} \sup_{|\mathbf{l}| = l} \: \sup_{|\mathbf{k}| = j-l} \| f^\mathbf{l}   ({\rm rot}\, g)^{\mathbf{k}} \|_{L^\infty_t L^2_{x,y}}  \\
& \le  \frac{C}{K^{\frac12}} \sum_{j=0}^{\nu^{-\frac12}}  \sum_{0\le l \le j/2} (l+1) \binom{j}{l}^{-\frac12}    \frac{1}{l!^{\frac32}}\sup_{|\mathbf{l}| = l}   \|f^\mathbf{l} \|_{L^\infty_{t,x,y}}    \frac{1}{(j-l)!^{\frac32}}   \sup_{|\mathbf{k}| = j-l} \| ({\rm rot}\, g)^{\mathbf{k}} \|_{L^\infty_t L^2_{x,y}}   \\
& + \frac{C}{K^{\frac12}} \sum_{j=0}^{\nu^{-\frac12}}  \sum_{j/2< l\le j}  (j-l+1)  \binom{j}{l}^{-\frac12}  \frac{1}{l!^{\frac32}}\sup_{|\mathbf{l}| = l}   \|f^\mathbf{l} \|_{L^\infty_{t}L^2_x L^\infty_y}      \frac{1}{(j-l)!^{\frac32}}   \sup_{\mathbf{|k|} = j-l} \| ({\rm rot}\, g)^{\mathbf{k}} \|_{L^\infty_t L^\infty_x L^2_y} \\
& \le  \frac{C}{K^{\frac12}} \sum_{j=0}^{\nu^{-\frac12}}  \sum_{0\le l \le j/2} (l+1)^\frac52 \binom{j}{l}^{-\frac12}    \frac{1}{(l+1)!^{\frac32}} \\
& \qquad \qquad \times \sup_{|\mathbf{l}| = l}   (\|\pa_x f^\mathbf{l} \|_{L^\infty_tL^2_xL^2_y}+\|f^\mathbf{l} \|_{L^\infty_tL^2_xL^2_y})^\frac12  ( \|\pa_x \pa_y f^\mathbf{l} \|_{L^\infty_tL^2_xL^2_y} + \|\pa_y f^\mathbf{l} \|_{L^\infty_tL^2_xL^2_y})^\frac12 \\
& \qquad \qquad \times   \frac{1}{(j-l)!^{\frac32}}   \sup_{|\mathbf{k}| = j-l} \| ({\rm rot}\, g)^{\mathbf{k}} \|_{L^\infty_t L^2_{x,y}}   \\
& + \frac{C}{K^{\frac12}} \sum_{j=0}^{\nu^{-\frac12}}  \sum_{j/2< l\le j}  (j-l+1)^\frac52  \binom{j}{l}^{-\frac12}  \frac{1}{l!^{\frac32}}\sup_{|\mathbf{l}| = l}   \|f^\mathbf{l} \|_{L^\infty_{t}L^2_x L^2_y}^\frac12  \|\pa_y f^\mathbf{l} \|_{L^\infty_{t}L^2_x L^2_y}^\frac12     \\
& \qquad \qquad \times  \frac{1}{(j-l+1)!^{\frac32}}   \sup_{\mathbf{|k|} = j-l} (\| \pa_x ({\rm rot}\, g)^{\mathbf{k}} \|_{L^\infty_t L^\infty_x L^2_y} + \|  ({\rm rot}\, g)^{\mathbf{k}} \|_{L^\infty_t L^\infty_x L^2_y} ).
\end{align*}
Here we have used the Sobolev embedding type inequality.
By using the bound
\begin{align*}
& \sup_{|\mathbf{l}| = l}   (\|\pa_x f^\mathbf{l} \|_{L^\infty_tL^2_xL^2_y}+\|f^\mathbf{l} \|_{L^\infty_tL^2_xL^2_y})^\frac12  ( \|\pa_x \pa_y f^\mathbf{l} \|_{L^\infty_tL^2_xL^2_y} + \|\pa_y f^\mathbf{l} \|_{L^\infty_tL^2_xL^2_y})^\frac12 \\
& \le \nu^{-\frac14} \sup_{l\le |{\bf l}| \le l+1} \| f^{\bf l} \|_{L^\infty_t L^2_{x,y}}+ \nu^\frac14 \sup_{l\le |{\bf l}| \le l+1} \| \pa_y f^{\bf l} \|_{L^\infty_t L^2_{x,y}}
\end{align*}
and by observing that there exists $C>0$ such that for $\binom{j}{l}^{-\frac12} (l+1)^{\frac52} \le C$ for $0\le l\le j/2$, we have 
\begin{align*}
&  \frac{C}{K^{\frac12}} \sum_{j=0}^{\nu^{-\frac12}}  \sum_{0\le l \le j/2} (l+1)^\frac52 \binom{j}{l}^{-\frac12}    \frac{1}{(l+1)!^{\frac32}} \\
& \qquad \qquad \times \sup_{|\mathbf{l}| = l}   (\|\pa_x f^\mathbf{l} \|_{L^\infty_tL^2_xL^2_y}+\|f^\mathbf{l} \|_{L^\infty_tL^2_xL^2_y})^\frac12  ( \|\pa_x \pa_y f^\mathbf{l} \|_{L^\infty_tL^2_xL^2_y} + \|\pa_y f^\mathbf{l} \|_{L^\infty_tL^2_xL^2_y})^\frac12 \\
& \qquad \qquad \times   \frac{1}{(j-l)!^{\frac32}}   \sup_{|\mathbf{k}| = j-l} \| ({\rm rot}\, g)^{\mathbf{k}} \|_{L^\infty_t L^2_{x,y}}   \\
& \le \frac{C}{K^\frac12 \nu^\frac14} \sum_{j=0}^{\nu^{-\frac12}}  \sum_{0\le l \le j/2}   \frac{1}{(l+1)!^{\frac32}} \sup_{l\le |{\bf l}| \le l+1} \| f^{\bf l} \|_{L^\infty_t L^2_{x,y}}  \frac{1}{(j-l)!^{\frac32}}   \sup_{|\mathbf{k}| = j-l} \| ({\rm rot}\, g)^{\mathbf{k}} \|_{L^\infty_t L^2_{x,y}}\\
& \quad + \frac{C\nu^\frac14}{K^\frac12} \sum_{j=0}^{\nu^{-\frac12}}  \sum_{0\le l \le j/2}   \frac{1}{(l+1)!^{\frac32}} \sup_{l\le |{\bf l}| \le l+1} \| \pa_y f^{\bf l} \|_{L^\infty_t L^2_{x,y}}  \frac{1}{(j-l)!^{\frac32}}   \sup_{|\mathbf{k}| = j-l} \| ({\rm rot}\, g)^{\mathbf{k}} \|_{L^\infty_t L^2_{x,y}}\\
& \le \frac{C}{K^\frac12\nu^\frac14} \| f\|_\infty \| {\rm rot}\, g\|_\infty + \frac{C\nu^\frac14}{K^\frac12} \| \pa_y f\|_\infty \| {\rm rot}\, g\|_\infty.
\end{align*}
where the discrete Young's convolution inequality is applied in the last line together with the estimate 
\begin{align*}
\sum_{j=0}^{\nu^{-\frac12}} \frac{1}{j!^\frac32} \sup_{|{\bf j}|=j} \| \pa_y f^{\bf j} \|_{L^\infty_t L^2_{x,y}}\le C \| \pa_y f\|_\infty.
\end{align*}
Similarly, since $(j-l+1)^\frac52 \binom{j}{l}^{-\frac12} \le C$ for $j/2\le l\le j$, we have 
\begin{align*}
& \frac{C}{K^{\frac12}} \sum_{j=0}^{\nu^{-\frac12}}  \sum_{j/2< l\le j}  (j-l+1)^\frac52  \binom{j}{l}^{-\frac12}  \frac{1}{l!^{\frac32}}\sup_{|\mathbf{l}| = l}   \|f^\mathbf{l} \|_{L^\infty_{t}L^2_x L^2_y}^\frac12  \|\pa_y f^\mathbf{l} \|_{L^\infty_{t}L^2_x L^2_y}^\frac12     \\
& \qquad \qquad \times  \frac{1}{(j-l+1)!^{\frac32}}   \sup_{\mathbf{|k|} = j-l} (\| \pa_x ({\rm rot}\, g)^{\mathbf{k}} \|_{L^\infty_t L^\infty_x L^2_y} + \|  ({\rm rot}\, g)^{\mathbf{k}} \|_{L^\infty_t L^\infty_x L^2_y} )\\
& \le \frac{C}{K^{\frac12}} \sum_{j=0}^{\nu^{-\frac12}}  \sum_{j/2< l\le j} \frac{1}{l!^{\frac32}}\sup_{|\mathbf{l}| = l}  ( \nu^{-\frac14} \|f^\mathbf{l} \|_{L^\infty_{t}L^2_x L^2_y} + \nu^\frac14  \|\pa_y f^\mathbf{l} \|_{L^\infty_{t}L^2_x L^2_y})  \\
& \qquad \qquad \times  \frac{1}{(j-l+1)!^{\frac32}}   \sup_{j-l\le \mathbf{|k|}\le j-l+1}  \|  ({\rm rot}\, g)^{\mathbf{k}} \|_{L^\infty_t L^\infty_x L^2_y} \\
& \le  \frac{C}{K^\frac12\nu^\frac14} \| f\|_\infty \| {\rm rot}\, g\|_\infty + \frac{C\nu^\frac14}{K^\frac12} \| \pa_y f\|_\infty \| {\rm rot}\, g\|_\infty.
\end{align*}
Hence the result follows from Lemma \ref{lem.BS.ap}. The proof is complete.
\end{proof}

\

\begin{proof}[Proof of Theorem \ref{thm.main.nl}] 
Let $C$ be the universal constant in Theorem \ref{main.thm2}.
Then the standard fixed point theorem in the closed convex set 
\begin{align*}
X_R = \{ f\in C([0,\frac{1}{K}]; H^1_{0,\sigma} (\T\times \R_+))~|~\| f\|_X =\| f\|_\infty + \nu^\frac12 \| {\rm rot}\, f \|_\infty \leq R\}, \qquad R=4 C\delta_0 \nu^\frac74
\end{align*}
is applied by using Theorem \ref{main.thm2} and Lemma \ref{lem.ap.bilinear}, if $\nu\le K^{-2}$ holds and if $\delta_0$ is sufficiently small. 
We note that the smallness condition $[|w_0|] + [|{\rm rot}\, w_0|]\le \delta_0 \nu^\frac94$, $\|r\|_2\le \delta_0\nu^\frac{11}{4}$ is needed to close the estimate. Since the argument is standard we omit the details. The proof is complete.
\end{proof}

\section*{Acknowledgements}
The first author acknowledges the support of SingFlows project, grant ANR-18- CE40-0027 of the French National Research Agency (ANR), and the support of the Institut Universitaire de France.
The second author acknowledges the support of JSPS KAKENHI Grant Number 17K05320.

\appendix

\section{Interpolation estimate for solutions to the Poisson equation}
 
\begin{lemma}\label{lem.stream.a1} Assume that $Y^k\omega \in L^2(\T_\nu \times \R_+)$ for $k=0,1,2$. 
Let $\phi\in \dot{H}^1_0 (\T_\nu \times \R_+)$  be the solution to the Poisson equation $-\Delta \phi=\omega$ in $\T_\nu \times \R_+$ with $\phi|_{Y=0}=0$.
Then there exists $C>0$ such that for any $j\geq 0$ we have
\begin{align}
\sup_{Y>0} \| \phi (\cdot, Y)\|_{L^2 (\T_\nu)} \leq C\Big ( (j+1)^{-\frac14} \| Y\omega\|_{L^2(\T_\nu \times \R_+)} + (j+1)^\frac14 \|Y^2 \omega \|_{L^2(\T_\nu \times\R_+)}\Big ).
\end{align}
\end{lemma}

\begin{proof} The solution is given by the formula
\begin{align*}
\phi (X,Y) = \int_0^Y e^{-(Y-Y')(-\pa_X^2)^\frac12} \int_{Y'}^\infty e^{-(Y''-Y')(-\pa_X^2)^\frac12} \omega (\cdot, Y'')\, d Y''\, dY'.
\end{align*} 
Here $e^{-Y(-\pa_X^2)^\frac12}$ is the Poisson semigroup. Then we have 
\begin{align*}
\| \phi (\cdot, Y)\|_{L^2 (\T_\nu)} \leq \int_0^Y \int_{Y'}^\infty \|\omega (Y'')\|_{L^2(\T_\nu)}\, dY''\, dY'.
\end{align*}
By decomposing the integral $\int_0^Y$ into $\int_0^{\min\{Y, (j+1)^{-\frac12}\}}$ and $\int_{\min\{Y, (j+1)^{-\frac12}\}}^Y$, we have from the H${\rm \ddot{o}}$lder inequality,
\begin{align*}
\sup_{Y>0} \| \phi (\cdot, Y)\|_{L^2 (\T_\nu)} \leq C (j+1)^{-\frac14} \| Y \omega\|_{L^2(\T_\nu \times \R)} + C (j+1)^{\frac14} \| Y^2 \omega \|_{L^2(\T_\nu \times \R_+)}. 
\end{align*}
The proof is complete.
\end{proof}

\

Lemma \ref{lem.stream.a1} yields the following 
\begin{proposition}\label{prop.stream.a} Let $\phi\in \dot{H}^1_0 (\T_\nu \times \R_+)$  be the solution to the Poisson equation $-\Delta \phi=\omega$ in $\T_\nu \times \R_+$ with $\phi|_{Y=0}=0$.
Then  for any $j\geq 0$ we have
\begin{align}
\begin{split}
M_{2,j,\frac{1}{1+Y}}[\pa_X \phi] & \le C (j+1)^{-\frac14} M_{2,j+1,Y}[\omega] + C (j+1)^\frac14 M_{2,j+1,Y^2}[\omega] \\
& \quad + C \kappa\nu^\frac12 j \big ( M_{2,j-1,Y}[\omega] + M_{2,j-1,1}[\nabla \phi]\big ).
\end{split}
\end{align}
Here $C>0$ is a universal constant.
\end{proposition}

\begin{proof} Since $-\Delta \pa_X \phi=\pa_X \omega$ we have $-(\Delta \pa_X \phi)^{\bf j} = \pa_X \omega^{\bf j}$.
Then we use the commutator relation 
\begin{align*}
-(\Delta  \phi)^{\bf j} & = -\nabla \cdot (\nabla \phi)^{\bf j} +\nu^\frac12 j_2 \frac{\chi_\nu'}{\chi_\nu} ( \pa_Y\phi)^{\bf j}\\\
& = -\Delta  \phi^{\bf j} + \pa_Y \big ( \nu^\frac12 j_2 \frac{\chi_\nu'}{\chi_\nu} \phi^{\bf j}\big ) + \nu^\frac12 j_2 \frac{\chi_\nu'}{\chi_\nu} (\pa_Y \phi)^{\bf j}. 
\end{align*}
Thus we have the Poisson equation for $\phi^{\bf j}$:
\begin{align*}
-\Delta \phi^{\bf j} = \omega^{\bf j} - \pa_Y \big ( \nu^\frac12 j_2 \frac{\chi_\nu'}{\chi_\nu}  \phi^{\bf j}\big ) - \nu^\frac12 j_2 \frac{\chi_\nu'}{\chi_\nu} (\pa_Y \phi)^{\bf j}.
\end{align*}
Then we decompose $\phi^{\bf j}$ into $\phi_1+\phi_{2,1}+\phi_{2,2}$ so that 
\begin{align*}
-\Delta \phi_1= \omega^{\bf j}, \quad -\Delta  \phi_{2,1}=  - \pa_Y \big ( \nu^\frac12 j_2 \frac{\chi_\nu'}{\chi_\nu} \phi^{\bf j}\big ), \quad   -\Delta \phi_{2,2}=  - \nu^\frac12 j_2 \frac{\chi_\nu'}{\chi_\nu} (\pa_Y \phi)^{\bf j},
\end{align*}
subject to the Dirichlet boundary condition.
Then Lemma \ref{lem.stream.a1} implies for $\pa_X\phi_1$,
\begin{align}
\sup_{Y>0} \| \pa_X \phi_1 (\cdot,Y) \|_{L^2(\T_\nu)} \le C\Big ( (j+1)^{-\frac14} \|Y\pa_X \omega^{\bf j} \|_{L^2(\T_\nu\times \R_+)}  + (j+1)^\frac14 \|Y^2\pa_X \omega^{\bf j}\|_{L^2(\T_\nu\times \R_+)}\Big ).
\end{align}
On the other hand, the simple energy estimate gives 
\begin{align*}
\| \nabla \phi_{2,1} \| \leq \nu^\frac12 j_2 \|\frac{\chi_\nu'}{\chi_\nu} \phi^{\bf j} \|\le \kappa \nu^\frac12 j_2 \| (\pa_Y \phi)^{(j_1,j_2-1)}\|.
\end{align*}
As for $\phi_{2,2}$, we have from $\frac{1}{\chi_\nu}(\pa_Y\phi)^{\bf j} = e^{-K\tau \nu^\frac12} (\pa_Y^2 \phi)^{(j_1,j_2-1)} = e^{-K\tau \nu^\frac12} \big ( -\omega^{(j_1,j_2-1)} -\pa_X^2 \phi^{(j_1,j_2-1)}\big )$,
the Hardy inequality and the integration by parts imply
\begin{align*}
\| \nabla \phi_{2,2} \|\le C\kappa \nu^\frac12 j_2  \Big ( \| Y\omega^{(j_1,j_2-1)} \| + \| \pa_X \phi^{(j_1,j_2-1)} \| \Big ).
\end{align*}
Hence we obtain the desired estimate by taking the $L^2$ norm in time and by taking the supremum about ${\bf j}$ such that ${\bf j}=j$. The proof is complete.
\end{proof}

\section{Proof of Proposition \ref{prop.extra}}

Let us go back to \eqref{eq.mOS} with $G=0$, but we impose the noslip boundary condition $\phi|_{Y=0}=\pa_Y\phi|_{Y=0}=0$ in this appendix. Then we have 
\begin{align}
\begin{split}
-\nu^\frac12 (\Delta \omega)^{\bf j} + \big (\pa_\tau +K\nu^\frac12 (j+1) \big ) \omega^{\bf j} & = - (V\cdot \nabla \omega)^{\bf j} - (\nabla^\bot \phi \cdot \nabla \Omega)^{\bf j} + ({\rm rot}\, F)^{\bf j} \\
&= ({\rm div}\, H)^{\bf j},
\end{split}
\end{align}
where 
\begin{align*}
H=- V \omega - \Omega \nabla^\bot \phi + (F_2,-F_1).
\end{align*}
The idea is to take the $L^2$ inner product with $\pa_\tau \phi^{\bf j}$,
which gives the estimates of $\cA \nabla \phi\cA_\infty$ and $\cA \Delta \phi \cA_\infty$ in terms of $\cA \nabla \phi \cA_{2,{\bf 1}}'$. The most technical part is the computation of the viscous term $\langle (\Delta \omega)^{\bf j}, \pa_\tau \phi^{\bf j}\rangle$ when $j_2\ne 0$, for which one needs to convert the vertical derivative $\pa_Y^2\omega$ into the tangential ones by using the equation.

\begin{lemma}\label{lem.ap.B.1} For any $\kappa\in (0,1]$ and $K\geq 1$ we have
\begin{align*}
& \int_0^{\tau_0} \langle \big ( \pa_\tau + K\nu^\frac12 (j+1)\big ) \omega^{\bf j} , \pa_\tau \phi^{\bf j}\rangle \, d\tau \\
& \geq \frac12 \| \pa_\tau (\nabla \phi)^{\bf j}\|_{L^2(0,\tau_0; L^2_{X,Y})}^2 +\frac{K\nu^\frac12 (j+1)}{2} \big ( \| (\nabla \phi)^{\bf j} (\tau_0)\|^2 - \| (\nabla \phi)^{\bf j} (0) \|^2 \big ) \\
& \quad - C \kappa^2 K \nu^\frac12 j  \, (\nu^\frac12 j^\frac32)^2 M_{\infty,j-1,1}[\nabla \phi]^2 - C (\kappa \nu^\frac12 j)^2  M_{2,j-1,1}[\pa_\tau \nabla \phi]^2 .
\end{align*}
Here $C$ is a universal constant.
\end{lemma}

\begin{proof} Let us recall the identity
\begin{align}
\omega^{\bf j} = -(\Delta \phi)^{\bf j} = - \nabla \cdot (\nabla \phi)^{\bf j} + \nu^\frac12 j_2 \frac{\chi_\nu'}{\chi_\nu} (\pa_Y \phi)^{\bf j},
\end{align}
which implies 
\begin{align*}
& \langle \big ( \pa_\tau + K\nu^\frac12 (j+1)\big ) \omega^{\bf j} , \pa_\tau \phi^{\bf j}\rangle \\
& = \| \pa_\tau (\nabla \phi)^{\bf j} \|^2 +2\nu^\frac12 j_2 \langle \frac{\chi_\nu'}{\chi_\nu} \pa_\tau (\pa_Y \phi)^{\bf j}, \pa_\tau \phi^{\bf j} \rangle + \frac{K \nu^\frac12 (j+1)}{2} \pa_\tau \| (\nabla \phi)^{\bf j} \|^2\\
& \quad + 2 \nu^\frac12 j_2 K\nu^\frac12 (j+1)  \langle \frac{\chi_\nu'}{\chi_\nu} (\pa_Y\phi)^{\bf j}, \pa_\tau \phi^{\bf j}\rangle .
\end{align*}
Then from $\pa_\tau \phi^{\bf j} = \chi_\nu \pa_\tau \big ( e^{-K\tau \nu^\frac12} (\pa_Y \phi)^{(j_1,j_2-1)}\big )$ for $j_2\geq 1$ we have 
\begin{align*}
& \int_0^{\tau_0} 2\nu^\frac12 j_2 \langle \frac{\chi_\nu'}{\chi_\nu} \pa_\tau (\pa_Y \phi)^{\bf j}, \pa_\tau \phi^{\bf j} \rangle \, d\tau \\
& \geq -\frac14 \| \pa_\tau (\nabla \phi)^{\bf j} \|_{L^2(0,\tau_0; L^2)}^2 -C (\kappa \nu^\frac12 j)^2 \big (  M_{2,j-1,1}[\pa_\tau \nabla \phi]^2 + (K\nu^\frac12)^2 M_{2,j-1,1}[\nabla \phi]^2 \big ) ,
\end{align*}
while we have from the integration by parts in time,
\begin{align*}
& \int_0^{\tau_0}  2 \nu^\frac12 j_2 K\nu^\frac12 (j+1)  \langle \frac{\chi_\nu'}{\chi_\nu} (\pa_Y\phi)^{\bf j}, \pa_\tau \phi^{\bf j}\rangle\, d\tau \\
&  =  2 \nu^\frac12 j_2 K\nu^\frac12 (j+1) \Big ( e^{-K\tau_0 \nu^\frac12}\langle \chi_\nu' (\pa_Y\phi)^{\bf j} , (\pa_Y \phi)^{(j_1,j_2-1)} \rangle (\tau_0) - \langle \chi_\nu' (\pa_Y\phi)^{\bf j}, (\pa_Y \phi)^{(j_1,j_2-1)} \rangle(0) \Big ) \\
& \quad - 2 \nu^\frac12 j_2 K\nu^\frac12 (j+1) \int_0^{\tau_0} e^{-K\tau \nu^\frac12} \langle \pa_\tau (\pa_Y \phi)^{\bf j}, \chi_\nu' (\pa_Y \phi)^{(j_1,j_2-1)} \rangle \, d\tau  \\
& \geq  2 \nu^\frac12 j_2 K\nu^\frac12 (j+1) \Big ( e^{-K\tau_0 \nu^\frac12}\langle \chi_\nu' (\pa_Y\phi)^{\bf j} (\tau_0), (\pa_Y \phi)^{(j_1,j_2-1)} (\tau_0) \rangle  - \langle \chi_\nu' (\pa_Y\phi)^{\bf j} (0), (\pa_Y \phi)^{(j_1,j_2-1)} (0) \rangle \Big ) \\
& \quad - \frac14 \| \pa_\tau (\pa_Y \phi)^{\bf j} \|_{L^2(0,\tau_0; L^2)}^2 - C (K \kappa \nu j^2)^2  M_{2,j-1,1}[\nabla \phi]^2.
\end{align*}
We also observe that for $j_2\geq 1$, 
\begin{align*}
\langle \chi_\nu' (\pa_Y \phi)^{\bf j} , (\pa_Y \phi)^{(j_1,j_2-1)} \rangle & = e^{-K\tau \nu^\frac12} \langle \chi_\nu' \chi_\nu (\pa_Y \pa_Y\phi)^{(j_1,j_2-1)}, (\pa_Y \phi)^{(j_1,j_2-1)}\rangle \\
& =  - \frac{e^{-K\tau \nu^\frac12}}{2} \langle \pa_Y (\chi_\nu' \chi_\nu) \,  (\pa_Y\phi)^{(j_1,j_2-1)}, (\pa_Y \phi)^{(j_1,j_2-1)}\rangle \\
& \quad - e^{-K\tau \nu^\frac12} \nu^\frac12 (j_2-1) \| \chi_\nu'  (\pa_Y \phi)^{(j_1,j_2-1)}\|^2
\end{align*}
Thus we conclude also from $K\tau \nu^\frac12 \le 1$ that 
\begin{align*}
& \int_0^{\tau_0}  2 \nu^\frac12 j_2 K\nu^\frac12 (j+1)  \langle \frac{\chi_\nu'}{\chi_\nu} (\pa_Y\phi)^{\bf j}, \pa_\tau \phi^{\bf j}\rangle\, d\tau \\
& \geq -C  K \nu^\frac12 (\kappa \nu^\frac12 j)^2 \Big ( j \| (\pa_Y \phi)^{(j_1,j_2-1)} (\tau_0) \|^2  +\| (\pa_Y \phi)^{(j_1,j_2-1)} (0) \|^2 \Big )  \\
& \quad - \frac14 \| \pa_\tau (\pa_Y \phi)^{\bf j} \|_{L^2(0,\tau_0; L^2)}^2 - C (K \kappa \nu j^2)^2  M_{2,j-1,1}[\nabla \phi]^2.
\end{align*}
Collecting these above and $M_{2,j-1,1}[\nabla \phi]^2\le (K\nu^\frac12)^{-1} M_{\infty,j-1,1}[\nabla \phi]^2$, we obtain the desired estimate. The proof is complete.
\end{proof}

\begin{lemma}\label{lem.ap.B.2} For any $\kappa\in (0,1]$ and $K\geq 1$ we have 
\begin{align*}
& \int_0^{\tau_0} \langle -\nu^\frac12 (\Delta \omega)^{\bf j}, \pa_\tau\phi^{\bf j} \rangle\, d\tau  \\
& \geq \frac{\nu^\frac12}{2} \big ( \| \omega^{\bf j} (\tau_0)\|^2 - \| \omega^{\bf j} (0) \|^2 \big ) - \frac14 M_{2,j,1}[\pa_\tau \nabla \phi]^2 \\
& \quad - C (\kappa \nu^\frac12 j)^2 \Big (  M_{2,j-1,1}[\pa_\tau \nabla \phi]^2  + (\nu^\frac12 (j-1))^2 M_{2,j-2,1}[\pa_\tau \nabla \phi]^2 \Big )\\
& \quad -C\kappa^2 \nu^\frac12 \Big ( M_{\infty,j,1}[\omega]^2 + (\nu^\frac12  j)^2 M_{\infty,j-1,1}[\omega]\Big ) \\
& \quad - C K \nu^\frac12 j (\kappa \nu^\frac12 j^\frac32)^2 \Big ( M_{\infty,j-1,1}[\nabla \phi]^2 +  (\nu^\frac12 (j-1))^2 M_{\infty,j-2,1}[\nabla \phi]^2 \Big ) \\
& \quad -C \Big ( M_{2,j,1}[H]^2 + (\nu^\frac12 j)^2 M_{2,j-1,1}[H]^2 \Big ).
\end{align*}
Here $C$ is a universal constant.
\end{lemma}

\begin{proof} We observe from 
\begin{align}\label{proof.lem.ap.B.2.-1}
\begin{split}
(\Delta \omega)^{\bf j} & = \nabla\cdot (\nabla \omega)^{\bf j} - \nu^\frac12 j_2 \frac{\chi_\nu'}{\chi_\nu} (\pa_Y \omega)^{\bf j},\qquad \chi_\nu' =\kappa e^{-\kappa \nu^\frac12 Y},\\
\nabla \partial_\tau \phi^{\bf j} & = \partial_\tau (\nabla \phi)^{\bf j} +\nu^\frac12 j_2\frac{\chi_\nu'}{\chi_\nu} \partial_\tau \phi^{\bf j} {\bf e}_2
\end{split}
\end{align}
and the integration by parts that 
\begin{align*}
\langle -\nu^\frac12 (\Delta \omega)^{\bf j}, \pa_\tau\phi^{\bf j} \rangle = \nu^\frac12 \langle (\nabla \omega)^{\bf j}, \partial_\tau (\nabla \phi)^{\bf j}\rangle + 2\nu j_2 \langle \frac{\chi_\nu'}{\chi_\nu} (\pa_Y \omega)^{\bf j}, \pa_\tau \phi^{\bf j}\rangle.
\end{align*}
Then the similar identities
\begin{align}\label{proof.lem.ap.B.2.0}
\begin{split}
(\nabla \omega)^{\bf j}  & = \nabla \omega^{\bf j} - \nu^\frac12 j_2\frac{\chi_\nu'}{\chi_\nu}  \omega^{\bf j} {\bf e}_2,\\
\nabla\cdot \pa_\tau (\nabla \phi)^{\bf j} & = \pa_\tau (\Delta \phi)^{\bf j} + \nu^\frac12 j_2 \frac{\chi_\nu'}{\chi_\nu} \pa_\tau (\pa_Y \phi)^{\bf j}
\end{split}
\end{align}
together with the integration by parts yield
\begin{align}\label{proof.lem.ap.B.2.1}
\langle -\nu^\frac12 (\Delta \omega)^{\bf j}, \pa_\tau\phi^{\bf j} \rangle = \nu^\frac12 \langle \omega^{\bf j}, \pa_\tau \omega^{\bf j}\rangle  -  2\nu j_2 \langle \frac{\chi_\nu'}{\chi_\nu} \omega^{\bf j}, \pa_\tau (\pa_Y \phi)^{\bf j}\rangle+ 2\nu j_2 \langle \frac{\chi_\nu'}{\chi_\nu} (\pa_Y \omega)^{\bf j}, \pa_\tau \phi^{\bf j}\rangle.
\end{align}
Again from the above identities about the commutators we have for $j_2\geq 1$,
\begin{align*}
\langle \frac{\chi_\nu'}{\chi_\nu} \omega^{\bf j}, \pa_\tau (\pa_Y\phi)^{\bf j} \rangle 
= - \langle \frac{\chi_\nu'}{\chi_\nu} (\pa_Y \omega)^{\bf j}, \pa_\tau \phi^{\bf j}\rangle 
- \nu^\frac12 \langle \frac{\chi_\nu''}{\chi_\nu}\omega^{\bf j}, \pa_\tau \phi^{\bf j}\rangle -\nu^\frac12 (2j_2 - 1) \langle (\frac{\chi_\nu'}{\chi_\nu})^2 \omega^{\bf j}, \pa_\tau \phi^{\bf j} \rangle.
\end{align*}   
Here $\chi_\nu''=-\kappa^2 e^{-\kappa \nu^\frac12 Y}$.
Thus \eqref{proof.lem.ap.B.2.1} is written as 
\begin{align}\label{proof.lem.ap.B.2.2}
\begin{split}
\langle -\nu^\frac12 (\Delta \omega)^{\bf j}, \pa_\tau\phi^{\bf j} \rangle 
& = \nu^\frac12 \langle \omega^{\bf j}, \pa_\tau \omega^{\bf j}\rangle +4\nu j_2 \langle \frac{\chi_\nu'}{\chi_\nu} (\pa_Y \omega)^{\bf j}, \pa_\tau \phi^{\bf j}\rangle \\
& \quad + 2\nu^\frac 32 j_2 \langle \frac{\chi_\nu''}{\chi_\nu}\omega^{\bf j}, \pa_\tau \phi^{\bf j}\rangle  + 2\nu^\frac32 j_2 (2j_2-1)  \langle (\frac{\chi_\nu'}{\chi_\nu})^2 \omega^{\bf j}, \pa_\tau \phi^{\bf j} \rangle.
\end{split}
\end{align}
Let us compute the term $\langle \frac{\chi_\nu'}{\chi_\nu} (\pa_Y \omega)^{\bf j}, \pa_\tau \phi^{\bf j}\rangle$. From the identity 
\begin{align*}
\frac{1}{\chi_\nu} (\pa_Y \omega)^{\bf j} = e^{-K\tau\nu^\frac12} (\pa_Y^2 \omega)^{(j_1,j_2-1)} = e^{-K\tau\nu^\frac12} \big ( (\Delta \omega)^{(j_1,j_2-1)} - \pa_X^2 \omega^{(j_1,j_2-1)}\Big ),
\end{align*}
we have 
\begin{align*}
\langle \frac{\chi_\nu'}{\chi_\nu} (\pa_Y \omega)^{\bf j}, \pa_\tau \phi^{\bf j}\rangle = e^{-K\tau \nu^\frac12} \langle \chi_\nu' (\Delta \omega)^{(j_1,j_2-1)} ,\pa_\tau \phi^{\bf j}\rangle + \langle \chi_\nu' \omega^{(j_1+1,j_2-1)}, \pa_\tau \pa_X \phi^{\bf j}\rangle.
\end{align*}
Since $\nu^\frac12 (\Delta \omega)^{(j_1,j_2-1)} = (\pa_\tau + K\nu^\frac12 j) \omega^{(j_1,j_2-1)}  - ({\rm div}\, H)^{(j_1,j_2-1)}$, the identity \eqref{proof.lem.ap.B.2.2} is written as 
\begin{align}\label{proof.lem.ap.B.2.3}
\begin{split}
\langle -\nu^\frac12 (\Delta \omega)^{\bf j}, \pa_\tau\phi^{\bf j} \rangle 
& = \nu^\frac12 \langle \omega^{\bf j}, \pa_\tau \omega^{\bf j}\rangle +4\nu^\frac12 j_2 e^{-K\tau \nu^\frac12} \langle \chi_\nu' \big ( \pa_\tau + K\nu^\frac12 j\big ) \omega^{(j_1,j_2-1)}, \pa_\tau \phi^{\bf j}\rangle \\
& \quad - 4\nu^\frac12 j_2 e^{-K\tau \nu^\frac12} \langle \chi_\nu' ({\rm div}\, H)^{(j_1,j_2-1)}, \pa_\tau \phi^{\bf j}\rangle \\
& \quad + 4\nu j_2 \langle \chi_\nu' \omega^{(j_1+1,j_2-1)}, \pa_\tau \pa_X \phi^{\bf j}\rangle\\
& \quad + 2\nu^\frac 32 j_2 \langle \frac{\chi_\nu''}{\chi_\nu}\omega^{\bf j}, \pa_\tau \phi^{\bf j}\rangle  + 2\nu^\frac32 j_2 (2j_2-1)  \langle (\frac{\chi_\nu'}{\chi_\nu})^2 \omega^{\bf j}, \pa_\tau \phi^{\bf j} \rangle.
\end{split}
\end{align}
Next we compute the term $\nu^\frac12 j_2 e^{-K\tau \nu^\frac12} \langle \chi_\nu' \big ( \pa_\tau + K\nu^\frac12 j\big ) \omega^{(j_1,j_2-1)}, \pa_\tau \phi^{\bf j}\rangle$ in \eqref{proof.lem.ap.B.2.3}: from the identities as in \eqref{proof.lem.ap.B.2.0} we have 
\begin{align*}
\begin{split}
e^{-K\tau \nu^\frac12} \langle \chi_\nu' \big ( \pa_\tau + K\nu^\frac12 j\big ) \omega^{(j_1,j_2-1)}, \pa_\tau \phi^{\bf j}\rangle
& = e^{-K\tau \nu^\frac12} \langle \chi_\nu' (\pa_\tau + K\nu^\frac12 j) (\nabla \phi)^{(j_1,j_2-1)}, \pa_\tau (\nabla \phi)^{{\bf j}}\rangle \\
&  \quad  + 2\nu^\frac12 j_2 e^{-K\tau \nu^\frac12} \langle \frac{(\chi_\nu')^2}{\chi_\nu} (\pa_\tau + K\nu^\frac12 j) (\pa_Y \phi)^{(j_1,j_2-1)}, \pa_\tau \phi^{{\bf j}}\rangle  \\
& \quad + \nu^\frac12 e^{-K\tau \nu^\frac12} \langle \chi_\nu'' (\pa_\tau + K\nu^\frac12 j) (\pa_Y  \phi)^{(j_1,j_2-1)}, \pa_\tau  \phi^{{\bf j}}\rangle .
\end{split}
\end{align*}
By setting $(\nabla \phi)^{\bf \tilde j-1}=e^{-K\tau \nu^\frac12} (\nabla \phi)^{(j_1,j_2-1)}$ for simplicity, we have 
\begin{align*}
\begin{split}
& e^{-K\tau \nu^\frac12} \langle \chi_\nu' \big ( \pa_\tau + K\nu^\frac12 j\big ) \omega^{(j_1,j_2-1)}, \pa_\tau \phi^{\bf j}\rangle\\
& = \langle \chi_\nu' \pa_\tau  (\nabla \phi)^{\bf \tilde j-1}, \pa_\tau (\nabla \phi)^{{\bf j}}\rangle \\
&  \quad  + 2\nu^\frac12 j_2  \langle \frac{(\chi_\nu')^2}{\chi_\nu} \pa_\tau  (\pa_Y \phi)^{\bf \tilde j-1}, \pa_\tau \phi^{{\bf j}}\rangle   + \nu^\frac12  \langle \chi_\nu'' \pa_\tau  (\pa_Y  \phi)^{\bf \tilde j-1}, \pa_\tau  \phi^{{\bf j}}\rangle \\
& \quad + K\nu^\frac12 j \Big ( \langle \chi_\nu'  (\nabla \phi)^{\bf \tilde j-1}, \pa_\tau (\nabla \phi)^{{\bf j}}\rangle   + 2\nu^\frac12 j_2  \langle \frac{(\chi_\nu')^2}{\chi_\nu}  (\pa_Y \phi)^{\bf \tilde j-1}, \pa_\tau \phi^{{\bf j}}\rangle  + \nu^\frac12  \langle \chi_\nu''  (\pa_Y  \phi)^{\bf \tilde j-1}, \pa_\tau  \phi^{{\bf j}}\rangle \Big ).
\end{split}
\end{align*}
Since 
\begin{align*}
& \pa_\tau  (\nabla \phi)^{{\bf j}}=\chi_\nu \pa_\tau (\pa_Y \nabla \phi)^{\bf \tilde j-1}=\chi_\nu \pa_Y \pa_\tau (\nabla \phi)^{\bf \tilde j-1} - \nu^\frac12 (j_2-1) \chi_\nu' \pa_\tau (\nabla \phi)^{\bf \tilde j -1},\\
&\pa_\tau \phi^{{\bf j}}=\chi_\nu \pa_\tau (\pa_Y  \phi)^{\bf \tilde j-1},
\end{align*}
we then arrive at 
\begin{align}\label{proof.lem.ap.B.2.4}
& \nu^\frac12 j_2 e^{-K\tau \nu^\frac12} \langle \chi_\nu' \big ( \pa_\tau + K\nu^\frac12 j\big ) \omega^{(j_1,j_2-1)}, \pa_\tau \phi^{\bf j}\rangle \nonumber \\
\begin{split}
& = \nu^\frac12 j_2 \Big \{ -\frac12 \langle \pa_Y ( \chi_\nu' \chi_\nu) \pa_\tau  (\nabla \phi)^{\bf \tilde j-1}, \pa_\tau (\nabla \phi)^{\bf \tilde j-1}\rangle \\
& \quad - \nu^\frac12 (j_2-1) \langle  (\chi_\nu')^2  \pa_\tau  (\nabla \phi)^{\bf \tilde j-1}, \pa_\tau ( \nabla \phi)^{\bf \tilde j-1}\rangle \\
&  \quad  + 2\nu^\frac12 j_2  \langle (\chi_\nu')^2 \pa_\tau (\pa_Y \phi)^{\bf \tilde j-1}, \pa_\tau (\pa_Y \phi)^{\bf \tilde j-1}\rangle  + \nu^\frac12  \langle \chi_\nu'' \pa_\tau (\pa_Y  \phi)^{\bf \tilde j-1}, \chi_\nu \pa_\tau  (\pa_Y \phi)^{\bf \tilde j-1}\rangle \\
& \quad +  K\nu^\frac12 j \Big ( \langle \chi_\nu'  (\nabla \phi)^{\bf \tilde j-1}, \pa_\tau ( \nabla \phi)^{\bf j}\rangle  \\
& \qquad \qquad \quad  + 2\nu^\frac12 j_2  \langle (\chi_\nu')^2  (\pa_Y \phi)^{\bf \tilde j-1},  \pa_\tau (\pa_Y \phi)^{\bf \tilde j-1}\rangle  + \nu^\frac12  \langle \chi_\nu''  (\pa_Y  \phi)^{\bf \tilde j-1}, \chi_\nu \pa_\tau  (\pa_Y\phi)^{\bf \tilde j-1}\rangle \Big ) \Big \}\nonumber 
\end{split} \\
\begin{split}
& \geq - C (\kappa \nu^\frac12 j_2)^2 \| \pa_\tau (\nabla \phi)^{\bf \tilde j-1}\|^2 \\
& \quad  +  K\nu j_2 j \Big ( \langle \chi_\nu'  (\nabla \phi)^{\bf \tilde j-1}, \pa_\tau ( \nabla \phi)^{\bf j}\rangle    + \nu^\frac12 j_2 \pa_\tau  \| \chi_\nu'   (\pa_Y \phi)^{\bf \tilde j-1}\|^2 + \frac{\nu^\frac12}{2}\pa_\tau  \langle \chi_\nu''  (\pa_Y  \phi)^{\bf \tilde j-1}, \chi_\nu  (\pa_Y\phi)^{\bf \tilde j-1}\rangle \Big ).
\end{split}
\end{align}
Here we have used the fact that it suffices to consider the case $j_2\geq 1$, and  $C$ is a universal constant.
Hence, by going back to \eqref{proof.lem.ap.B.2.3}, we have 
\begin{align}\label{proof.lem.ap.B.2.5}
\begin{split}
\langle -\nu^\frac12 (\Delta \omega)^{\bf j}, \pa_\tau\phi^{\bf j} \rangle 
& \geq  \nu^\frac12 \langle \omega^{\bf j}, \pa_\tau \omega^{\bf j}\rangle - C(\kappa \nu^\frac12 j_2)^2 \| \pa_\tau (\nabla \phi)^{\bf \tilde j -1} \|^2 \\
&\quad  +  K\nu j_2 j \Big ( \langle \chi_\nu'  (\nabla \phi)^{\bf \tilde j-1}, \pa_\tau ( \nabla \phi)^{\bf j}\rangle  \\
& \qquad  + \nu^\frac12 j_2 \pa_\tau  \| \chi_\nu'   (\pa_Y \phi)^{\bf \tilde j-1}\|^2 + \frac{\nu^\frac12}{2}\pa_\tau  \langle \chi_\nu''  (\pa_Y  \phi)^{\bf \tilde j-1}, \chi_\nu  (\pa_Y\phi)^{\bf \tilde j-1}\rangle \Big )\\
& \quad - 4\nu^\frac12 j_2 e^{-K\tau \nu^\frac12} \langle \chi_\nu' ({\rm div}\, H)^{(j_1,j_2-1)}, \pa_\tau \phi^{\bf j}\rangle \\
& \quad + 4\nu j_2 \langle \chi_\nu' \omega^{(j_1+1,j_2-1)}, \pa_\tau \pa_X \phi^{\bf j}\rangle\\
& \quad + 2\nu^\frac 32 j_2 \langle \frac{\chi_\nu''}{\chi_\nu}\omega^{\bf j}, \pa_\tau \phi^{\bf j}\rangle  + 2\nu^\frac32 j_2 (2j_2-1)  \langle (\frac{\chi_\nu'}{\chi_\nu})^2 \omega^{\bf j}, \pa_\tau \phi^{\bf j} \rangle.
\end{split}
\end{align}
Here $C$ is a universal constant. Next we observe from $\pa_\tau \phi^{\bf j} = \chi_\nu \pa_\tau (\pa_Y\phi)^{\bf \tilde j-1}$ that 
\begin{align}\label{proof.lem.ap.B.2.6}
- 4\nu^\frac12 j_2 e^{-K\tau \nu^\frac12} \langle \chi_\nu' ({\rm div}\, H)^{(j_1,j_2-1)}, \pa_\tau \phi^{\bf j}\rangle \geq - C \kappa \nu^\frac12 j_2 \big ( \| H_1^{(j_1+1,j_2-1)} \| + \| H_2^{\bf j} \| \big )  \| \pa_\tau (\pa_Y \phi)^{\bf \tilde j-1} \|,
\end{align}
and also 
\begin{align}\label{proof.lem.ap.B.2.7}
 4\nu j_2 \langle \chi_\nu' \omega^{(j_1+1,j_2-1)}, \pa_\tau \pa_X \phi^{\bf j}\rangle & \geq -C \kappa \nu j_2 \| \omega^{(j_1+1,j_2-1)} \| \, \| \pa_\tau \pa_X \phi^{\bf j} \|,\\ 
2\nu^\frac 32 j_2 \langle \frac{\chi_\nu''}{\chi_\nu}\omega^{\bf j}, \pa_\tau \phi^{\bf j}\rangle   & \geq -C \kappa^2 \nu^\frac32 j_2 \| \omega^{\bf j} \| \, \| \pa_\tau (\pa_Y \phi)^{\bf \tilde j-1} \|.\label{proof.lem.ap.B.2.7'}
\end{align}
Finally let us compute the term $\nu^\frac12 \langle (\frac{\chi_\nu'}{\chi_\nu})^2 \omega^{\bf j}, \pa_\tau \phi^{\bf j} \rangle$ when $j_2\geq 1$. If $j_2=1$ then 
\begin{align}\label{proof.lem.ap.B.2.8}
\nu^\frac12 \langle (\frac{\chi_\nu'}{\chi_\nu})^2 \omega^{\bf j}, \pa_\tau \phi^{\bf j} \rangle 
& = \nu^\frac12 \langle (\chi_\nu')^2 e^{-K\tau \nu^\frac12}(\pa_Y \omega)^{(j_1,0)}, \pa_\tau (e^{-K\tau \nu^\frac12}(\pa_Y \phi)^{(j_1,0)} )\rangle \nonumber \\
& = \nu^\frac12 \langle e^{-K\tau \nu^\frac12} \nabla \pa_Y \phi^{(j_1,0)}, \nabla \Big ( (\chi_\nu')^2\pa_\tau (e^{-K\tau \nu^\frac12}(\pa_Y \phi)^{(j_1,0)} )\Big ) \rangle \nonumber \\
& = \frac12 \nu^\frac12 \pa_\tau \| \chi_\nu' e^{-K\tau \nu^\frac12} \nabla \pa_Y \phi ^{(j_1,0)} \|^2 \nonumber \\
& \quad + 2\nu \langle \chi_\nu''\chi_\nu' e^{-K\tau \nu^\frac12}  \pa_Y^2 \phi^{(j_1,0)}, \pa_\tau (e^{-K\tau \nu^\frac12}(\pa_Y \phi)^{(j_1,0)} )\rangle \nonumber \\
& \geq \frac12 \nu^\frac12 \pa_\tau \| \chi_\nu' e^{-K\tau \nu^\frac12} \nabla \pa_Y \phi ^{(j_1,0)} \|^2 - C\kappa^3 \nu \| \omega^{(j_1,0)} \| \, \| \pa_\tau (\pa_Y \phi)^{\bf \tilde j-1}\|.
\end{align} 
If $j_2\geq 2$ then 
\begin{align}\label{proof.lem.ap.B.2.9}
\nu^\frac12 \langle (\frac{\chi_\nu'}{\chi_\nu})^2 \omega^{\bf j}, \pa_\tau \phi^{\bf j} \rangle 
& = e^{-2K\tau \nu^\frac12} \nu^\frac12 \langle (\chi_\nu')^2 (\pa_Y^2\omega)^{(j_1,j_2-2)}, \pa_\tau \phi^{\bf j}\rangle,
\end{align}\label{proof.lem.ap.B.2.10}
and then by using the identity $\nu^\frac12 (\Delta \omega)^{(j_1,j_2-2)} = (\pa_\tau + K\nu^\frac12 (j-1)) \omega^{(j_1,j_2-2)} - ({\rm div}\, H)^{(j_1,j_2-2)}$, we have 
\begin{align}
\begin{split}
\nu^\frac12 \langle (\frac{\chi_\nu'}{\chi_\nu})^2 \omega^{\bf j}, \pa_\tau \phi^{\bf j} \rangle 
& =  - \nu^\frac12 \langle (\chi_\nu')^2  \omega^{(j_1+2,j_2-2)}, \pa_\tau  \phi^{\bf j}\rangle\\
& \quad + e^{-2K\tau \nu^\frac12} \langle (\chi_\nu')^2 (\pa_\tau + K\nu^\frac12 (j-1)) \omega^{(j_1,j_2-2)}, \pa_\tau \phi^{\bf j}\rangle\\
& \qquad -   e^{-2K\tau \nu^\frac12} \langle (\chi_\nu')^2 ({\rm div}\, H)^{(j_1,j_2-2)}, \pa_\tau \phi^{\bf j}\rangle.
\end{split}
\end{align}
As for the second term of the right-hand side of \eqref{proof.lem.ap.B.2.10}, we have for $j\geq j_2\geq 2$,
\begin{align*}
& e^{-2K\tau \nu^\frac12} \langle (\chi_\nu')^2 (\pa_\tau + K\nu^\frac12 (j-1)) \omega^{(j_1,j_2-2)}, \pa_\tau \phi^{\bf j}\rangle \\
& = e^{-2K\tau \nu^\frac12} \langle (\chi_\nu')^2 (\pa_\tau + K\nu^\frac12 (j-1)) \pa_X \phi ^{(j_1,j_2-2)}, \pa_\tau \pa_X \phi^{\bf j}\rangle \\
& \quad - e^{-2K\tau \nu^\frac12} \langle (\chi_\nu')^2 (\pa_\tau + K\nu^\frac12 (j-1)) \big ( e^{2K\tau \nu^\frac12} (\pa_Y\phi)^{\bf \tilde j-1} \big ) , \pa_\tau (\pa_Y \phi)^{\bf \tilde j-1} \rangle\\
&  \geq - \kappa^2 \big ( \| \pa_\tau \pa_X \phi^{(j_1,j_2-2)} \| + K\nu^\frac12 j \| \pa_X \phi^{(j_1,j_2-2)} \| \big ) \| \pa_\tau \pa_X \phi^{\bf j} \|  \\
& \qquad - \kappa^2 \big ( \| \pa_\tau (\pa_Y \phi)^{\bf \tilde j-1}\| + K\nu^\frac12 j \| (\pa_Y \phi)^{\bf \tilde j-1} \| \big ) \| \pa_\tau (\pa_Y \phi)^{\bf \tilde j-1} \| .
\end{align*}
Since it is straightforward to see that 
\begin{align*}
 - \nu^\frac12 \langle (\chi_\nu')^2  \omega^{(j_1+2,j_2-2)}, \pa_\tau  \phi^{\bf j}\rangle & \geq - \kappa^2 \nu^\frac12  \| \omega^{(j_1+1,j_2-2)} \|\,   \|\pa_\tau (\pa_X \phi)^{\bf j}\| \\
  -   e^{-2K\tau \nu^\frac12} \langle (\chi_\nu')^2 ({\rm div}\, H)^{(j_1,j_2-2)}, \pa_\tau \phi^{\bf j}\rangle & \geq - \kappa^2 \big ( \| H_1^{(j_1+1,j_2-2)} \| + \|H_2^{(j_1,j_2-1)} \|\big ) \| \pa_\tau (\pa_Y \phi)^{\bf \tilde j-1} \|,
 \end{align*}
 we obtain for $j_2\geq 2$,
 \begin{align}\label{proof.lem.ap.B.2.11}
\begin{split}
& \nu^\frac12 \langle (\frac{\chi_\nu'}{\chi_\nu})^2 \omega^{\bf j}, \pa_\tau \phi^{\bf j} \rangle \\
& \geq  -\kappa^2 \Big ( \| \pa_\tau \pa_X \phi^{(j_1,j_2-2)} \| + K\nu^\frac12 j \| \pa_X \phi^{(j_1,j_2-2)} \|  + \nu^\frac12  \| \omega^{(j_1+1,j_2-2)} \| \Big )\|\pa_\tau (\pa_X \phi)^{\bf j}\|  \\
& \quad -\kappa^2 \Big (  \| \pa_\tau (\pa_Y \phi)^{\bf \tilde j-1}\| + K\nu^\frac12 j \| (\pa_Y \phi)^{\bf \tilde j-1} \| + \| H_1^{(j_1+1,j_2-2)} \| + \|H_2^{(j_1,j_2-1)} \| \Big ) \| \pa_\tau (\pa_Y \phi)^{\bf \tilde j-1} \|.
\end{split}
\end{align}
Collecting \eqref{proof.lem.ap.B.2.5}-\eqref{proof.lem.ap.B.2.7'} with \eqref{proof.lem.ap.B.2.8} (for $j_2=1$) and \eqref{proof.lem.ap.B.2.11} (for $j_2\geq 2$), we conclude the desired estimate by using the bound such as
\begin{align*}
M_{2,j,1}[f]^2=\sup_{|{\bf j}|=j} \| f^{\bf j} \|_{L^2(0,\frac{1}{K\nu^\frac12}; L^2_{X,Y})}^2 \leq \frac{1}{K\nu^\frac12}\sup_{|{\bf j}|=j} \| f^{\bf j} \|_{L^\infty (0,\frac{1}{K\nu^\frac12}; L^2_{X,Y})}^2 =\frac{1}{K\nu^\frac12} M_{\infty,j,1}[f]^2.
\end{align*}
The proof is complete.
\end{proof}

As a consequence of Lemmas \ref{lem.ap.B.1} and \ref{lem.ap.B.2}, we obtain 
\begin{corollary}\label{cor.ap.B} There exists $\kappa_B\in (0,1]$ such that for any $\kappa\in (0,\kappa_B]$ and $K\geq 1$, 
\begin{align*}
& \nu^\frac14 \sum_{j=0}^{\nu^{-\frac12}} \frac{1}{(j!)^\frac32 \nu^\frac{j}{2}}M_{\infty,j,1}[\omega] + K^\frac12 \nu^\frac14 \sum_{j=0}^{\nu^{-\frac12}} \frac{(j+1)^\frac12}{(j!)^\frac32\nu^\frac{j}{2}} M_{\infty,j,1}[\nabla \phi]  + \sum_{j=0}^{\nu^{-\frac12}} \frac{1}{(j!)^\frac32\nu^\frac{j}{2}} M_{2,j,1}[\pa_\tau \nabla \phi] \\
& \leq C\Big ( \nu^\frac14 \sum_{j=0}^{\nu^{-\frac12}} \frac{1}{(j!)^\frac32 \nu^\frac{j}{2}} \|\omega^{\bf j}(0) \| + K^\frac12 \sum_{j=0}^{\nu^{-\frac12}} \frac{\nu^\frac14 (j+1)^\frac12}{(j!)^\frac32 \nu^\frac{j}{2}} \| (\nabla \phi)^{\bf j} (0)\| +  \sum_{j=0}^{\nu^{-\frac12}} \frac{1}{(j!)^\frac32\nu^\frac{j}{2}} M_{2,j,1}[H]\Big ).
\end{align*}
Here $C$ is a universal constant.
\end{corollary}

\

We note that 
$$\sum_{j=0}^{\nu^{-\frac12}} \frac{\nu^\frac14 (j+1)^\frac12}{(j!)^\frac32 \nu^\frac{j}{2}} \| (\nabla \phi)^{\bf j} (0)\| \le C  \sum_{j=0}^{\nu^{-\frac12}} \frac{1}{(j!)^\frac32 \nu^\frac{j}{2}} \| (\nabla \phi)^{\bf j} (0)\| = C[\|\nabla \phi (0)\|]$$ 
since $j\le \nu^{-\frac12}$.
In virtue of Corollary \ref{cor.ap.B} it remains to estimate $\sum_{j=1}^{\nu^{-\frac12}} \frac{1}{(j!)^\frac32 \nu^\frac{j}{2}}M_{2,j,1}[H]$. Recall that $H=-V\omega - \Omega \nabla^\bot \phi +(F_2,-F_1)$. 
Hence it suffices to show
\begin{lemma}\label{lem.ap.H} For any $\kappa\in (0,1]$ and $K\geq 1$ we have 
\begin{align}
\sum_{j=0}^{\nu^{-\frac12}}\frac{1}{(j!)^\frac32 \nu^\frac{j}{2}} M_{2,j,1}[\Omega \nabla \phi] & \le \frac{C  (C_0^*+C_1^*)}{\nu^\frac14} \cA \nabla \phi \cA_{2,{\bf 1}}', \label{est.lem.ap.H1}\\
\sum_{j=0}^{\nu^{-\frac12}}\frac{1}{(j!)^\frac32 \nu^\frac{j}{2}}  M_{2,j,1}[V\omega] & \le \frac{C (C_0^*+C_1^*)}{\nu^\frac14} \big (  \cA \Delta (\phi - \phi_{app,1}) \cA_{2,{\bf 1}}' + \cA \Delta \phi_{app,1} \cA_{2,Y}' \big ).\label{est.lem.ap.H2}
\end{align}
Here $\phi_{app,1}=(\phi_{1,1} + \phi_{1,2}) [(I+R_{bc})^{-1}h]$ with $h=-\pa_Y \Phi_{slip}|_{Y=0}$ and $C$ is a universal constant.
\end{lemma}

\begin{proof} We give a sketch of the proof only for \eqref{est.lem.ap.H2}, for \eqref{est.lem.ap.H1} is proved in the similar manner. Let $|{\bf j}|=j$. Then 
\begin{align*}
\sum_{j=1}^{\nu^{-\frac12}} \frac{1}{j!^\frac32 \nu^\frac{j}{2}}M_{2,j,1}[V\omega]
\le \sum_{j=0}^{\nu^{-\frac12}} \frac{1}{j!^\frac32 \nu^\frac{j}{2}}\max_{|{\bf j}|=j} \sum_{{\bf l}\le {\bf j}} \binom{{\bf j}}{{\bf l}} \| V^{\bf l} \omega^{{\bf j-l}}\|_{L^2(0,\frac{1}{K\nu^\frac12}; L^2)}.
\end{align*}
Here $V^{\bf j}=e^{-K\tau \nu^\frac12 j} B_{j_2} \pa_X^{j_1} V$, while $\omega^{\bf j} = e^{-K\tau \nu^\frac12 (j+1)} B_{j_2}\pa_X^{j_1} \omega$.
Since $\omega=-\Delta (\phi-\phi_{app,1})-\Delta \phi_{app,1}$ in virtue of the construction, we have 
\begin{align*}
& \| V^{\bf l} \omega^{{\bf j-l}}\|_{L^2(0,\frac{1}{K\nu^\frac12}; L^2)}\\
& \le \|V^{\bf l}\|_{L^\infty} \| (\Delta (\phi-\phi_{app,1}))^{\bf j-l} \|_{L^2(0,\frac{1}{K\nu^\frac12}; L^2)} + \| \pa_Y V^{\bf l} \|_{L^\infty} \| Y (\Delta \phi_{app,1})^{\bf j-l} \|_{L^2(0,\frac{1}{K\nu^\frac12}; L^2)}.
\end{align*}
By using $\binom{{\bf j}}{{\bf l}}\le \binom{j}{l}$ with $l=|{\bf l}|$, we have 
\begin{align*}
& \frac{1}{j!^\frac32 \nu^\frac{j}{2}} \sum_{{\bf l}\le {\bf j}} \binom{{\bf j}}{{\bf l}} \| V^{\bf l} \omega^{{\bf j-l}}\|_{L^2(0,\frac{1}{K\nu^\frac12}; L^2)} \\
& \le \sum_{{\bf l}\le {\bf j}} (\frac{l!(j-l)!}{j!})^\frac12\frac{M_{2,j-l,1}[\Delta (\phi-\phi_{app,1})] + M_{2,j-l,Y}[\Delta \phi_{app,1}]}{(j-l)!^\frac32 \nu^\frac{j-l}{2}}  \frac{1}{l!^\frac32 \nu^\frac{l}{2}} \max_{|{\bf l}|=l}\big ( \| V^{\bf l}\|_{L^\infty} + \|\pa_Y V^{\bf l}\|_{L^\infty} \big ) .
\end{align*}
Next we observe that for all $l\in \N\cup\{0\}$,
\begin{align*}
\# \{{\bf l}~|~ |{\bf l}|=l, {\bf l}\le {\bf j}\} = \#\{l_2, \max(0,l-j+j_2)\le l_2\le \min (j_2,l)\}\le \min (l+1,j-l+1),
\end{align*}
which gives the bound of the form $\displaystyle \sum_{{\bf l}\le {\bf j}} \le \sum_{l=0}^j \min (l+1, j-l+1)$. Hence we have 
\begin{align*}
& \frac{1}{j!^\frac32 \nu^\frac{j}{2}} \sum_{{\bf l}\le {\bf j}} \binom{{\bf j}}{{\bf l}} \| V^{\bf l} \omega^{{\bf j-l}}\|_{L^2(0,\frac{1}{K\nu^\frac12}; L^2)} \\
& \le \sum_{l=0}^j \min (l+1,j-l+1)  (\frac{l!(j-l)!}{j!})^\frac12 \\
& \qquad \times \frac{M_{2,j-l,1}[\Delta (\phi-\phi_{app,1})] + M_{2,j-l,Y}[\Delta \phi_{app,1}]}{(j-l)!^\frac32 \nu^\frac{j-l}{2}}  \frac{1}{l!^\frac32 \nu^\frac{l}{2}} \max_{|{\bf l}|=l}\big ( \| V^{\bf l}\|_{L^\infty} + \|\pa_Y V^{\bf l}\|_{L^\infty} \big ) .
\end{align*}
Since $ \min (l+1,j-l+1)  (\frac{l!(j-l)!}{j!})^\frac12$ is unformly bounded about $0\le l \le j$,
the Young inequality for $l^1$ convolution gives the inequality 
\begin{align*}
& \sum_{j=0}^{\nu^{-\frac12}} \frac{1}{j!^\frac32 \nu^\frac{j}{2}}\max_{|{\bf j}|=j} \sum_{{\bf l}\le {\bf j}} \binom{{\bf j}}{{\bf l}} \| V^{\bf l} \omega^{{\bf j-l}}\|_{L^2(0,\frac{1}{K\nu^\frac12}; L^2)} \\
& \le C \sum_{j=0}^{\nu^{-\frac12}} \frac{1}{j!^\frac32 \nu^{\frac{j}{2}}} \max_{|{\bf j}|=j} \big ( \| V^{\bf j}\|_{L^\infty} + \| \pa_Y V^{\bf j} \|_{L^\infty}\big )  \\
& \qquad \times   \sum_{j=0}^{\nu^{-\frac12}} \frac{1}{j!^\frac32 \nu^{\frac{j}{2}}} \max_{|{\bf j}|=j} \big ( M_{2,j,1} [\Delta (\phi-\phi_{app,1}) ]  + M_{2,j,Y}[\Delta \phi_{app,1}]\big ).
\end{align*}
Then the desired estimate follows by noticing $\pa_Y V^{\bf j} = (\pa_Y V)^{\bf j} + \nu^\frac12 j_2 \chi_\nu' (\pa_Y V)^{(j_1,j_2-1)}$ and the bound of the form $\cA f \cA_2 \le \nu^{-\frac14} \cA f \cA_{2,{\bf 1}}'$.
The proof is complete.
\end{proof}

\

Proposition \ref{prop.extra} follows from Corollary \ref{cor.ap.B} and Lemma \ref{lem.ap.H}. 

\section{Estimate of the Biot-Savart law}

\begin{lemma}\label{lem.BS.ap} The following statement holds if $\kappa$ is sufficiently small. Assume that $f\in C([0,\frac{1}{K}); H^1 (\T\times \R_+)^2)$ satisfies ${\rm div}\, f=0$ for $y>0$ and $f_2|_{y=0}=0$. Then 
\begin{align*}
\| \nabla f \|_p \le C \| {\rm rot}\, f\|_p,\qquad p\in [1,\infty].
\end{align*}
Here $C$ is a universal constant.
\end{lemma}

\begin{proof} We observe that $\pa_y f_1 = {\rm rot}\, f+\pa_x f_2$ and $\pa_y f_2=-\pa_x f_1$. Hence it suffices to show $\|\pa_x f\|_p \le C \| {\rm rot}\, f \|_p$. Since $f=\nabla^\bot \phi$ with the streamfunction $\phi$ and $-\Delta \phi = \omega$ with $\omega={\rm rot}\, g$ and $\phi|_{y=0}=0$, we have 
\begin{align*}
-(\Delta \pa_x \phi)^{\bf j} = \pa_x \omega^{\bf j},  \quad \omega^{\bf j}=e^{-Kt(j+1)} \chi^{j_2} \pa_y^{j_2} \pa_x^{j_1}\omega, \quad j_1+j_2=j.
\end{align*}
In virtue of the identity $-(\Delta \pa_x \phi)^{\bf j}=-\nabla \cdot (\pa_x \nabla \phi)^{\bf j} +j_2\frac{\chi'}{\chi}(\pa_y \pa_x\phi)^{\bf j}$ the integration by parts gives 
\begin{align*}
\| (\nabla \pa_x \phi)^{\bf j} \|^2 + 2j_2\langle \frac{\chi'}{\chi} (\pa_y \pa_x \phi)^{\bf j}, \pa_x \phi^{\bf j}\rangle = -\langle \omega^{\bf j}, \pa_x^2\phi^{\bf j}\rangle
\end{align*}
Since $\pa_x \phi^{\bf j}=e^{-Kt}\chi (\pa_y\pa_x\phi)^{(j_1,j_2-1)}$ we thus have 
\begin{align*}
\| (\nabla \pa_x \phi)^{\bf j} \| \le C \big ( \| \omega^{\bf j} \| + \kappa j \| (\pa_y \pa_x \phi)^{(j_1,j_2-1)} \| \big ),
\end{align*}
where $C$ is a universal constant. This estimate implies $\| \pa_x \nabla \phi\|_p\le C\big ( \| \omega\|_p + \kappa \| \pa_x \pa_y \phi \|_p\big )$, and thus, by taking $\kappa$ small enough, we obtain $\| \pa_x\nabla \phi\|_p\le C \| \omega\|_p$. The proof is complete.
\end{proof}

\bibliography{biblio_Prandtl_Gevrey2}
\bibliographystyle{abbrv}

\end{document}